\documentclass{amsart}

\usepackage{amsmath, amsfonts, amssymb, amsthm}
\usepackage{hyperref}
\usepackage{verbatim}
\usepackage{wasysym}
\usepackage{graphicx, epstopdf}

\newtheorem{theorem}{Theorem}[subsection]
\newtheorem{lemma}[theorem]{Lemma}
\newtheorem{proposition}[theorem]{Proposition}
\newtheorem{corollary}[theorem]{Corollary}
\newtheorem{conjecture}[theorem]{Conjecture}

\newtheorem*{claim*}{Claim}

\theoremstyle{definition} \newtheorem{definition}[theorem]{Definition}
\newtheorem{example}[theorem]{Example}
\newtheorem{convention}[theorem]{Convention}

\theoremstyle{remark} \newtheorem{remark}[theorem]{Remark}

\newtheoremstyle{dotless}{}{}{\itshape}{}{\bfseries}{}{ }{}

\theoremstyle{dotless}
\newtheorem*{theorem*}{Bordered pairing formula:}

\newcommand{\R}{\mathbb{R}} 
\newcommand{\Z}{\mathbb{Z}}

\newcommand{\A}{\mathcal{A}}
\newcommand{\B}{\mathcal{B}}

\newcommand{\I}{\mathcal{I}}

\newcommand{\D}{\widehat{D}}
\newcommand{\DD}{\widehat{DD}}

\DeclareMathOperator{\id}{id}

\DeclareMathOperator{\Hom}{Hom}

\begin{document}

\title{On bordered theories for Khovanov homology}
\author{Andrew Manion}
\address{Department of Mathematics\\ UCLA\\ 520 Portola Plaza\\ Los Angeles, CA 90095}
\email{manion@math.ucla.edu}
\thanks{This work was supported by the National Science Foundation GRFP under Award No.
DGE 1148900 and the National Science Foundation MSPRF under Award No. 1502686.}

\begin{abstract}
We describe how to formulate Khovanov's functor-valued invariant of tangles in the language of bordered Heegaard Floer homology. We then give an alternate construction of Lawrence Roberts' Type D and Type A structures in Khovanov homology, and his algebra $\mathcal{B}\Gamma_n$, in terms of Khovanov's theory of modules over the ring $H^n$. We reprove invariance and pairing properties of Roberts' bordered modules in this language. Along the way, we obtain an explicit generators-and-relations description of $H^n$ which may be of independent interest.
\end{abstract}

\maketitle

\section{Introduction}

We consider two tangle theories for Khovanov homology which are inspired by the bordered Heegaard Floer homology of Lipshitz, Ozsv{\'a}th, and Thurston \cite{LOTBorderedOrig}. The first theory is a reformulation of Khovanov's functor-valued invariant \cite{KhovFunctor} in the bordered language. The second theory was introduced by Lawrence Roberts in \cite{RtypeD} and \cite{RtypeA}.

These bordered Khovanov theories share the same basic structure. Each assigns a differential bigraded algebra $\mathcal{B}$ to a collection of $2n$ points on the line $\{0\} \times \R$ in the plane $\R \times \R$. To a tangle diagram $T_1$ in $\R_{\geq 0} \times \R$ with $2n$ endpoints on $\{0\} \times \R$, these theories assign a Type D structure $\D_{T_1}$ over $\mathcal{B}$. (The definitions of Type D structures, and other elements of ``bordered algebra,'' will be given in Section~\ref{BorderedAlgebraSection}.) 

To a tangle diagram $T_2$ in $\R_{\leq 0} \times \R$ with $2n$ endpoints on $\{0\} \times \R$, bordered theories assign a Type A structure (i.e. an $\mathcal{A}_{\infty}$ module) $\widehat{A}_{T_2}$ over $\B$. There is a natural pairing operation between Type D and Type A structures over $\B$, called the box tensor product and denoted $\boxtimes$ (or $\boxtimes_{\B}$ when $\B$ is unclear). If $T_2 T_1$ denotes the link diagram obtained by concatenating $T_2$ and $T_1$ horizontally, bordered theories compute the Khovanov complex $CKh(T_2 T_1)$ as follows:
\begin{theorem*}
\[
CKh(T_2 T_1) \cong \widehat{A}_{T_2} \boxtimes_{\B} \D_{T_1}.
\]
\end{theorem*}

In Section~\ref{HnAsBorderedSection}, we will obtain a bordered theory with the above structure by taking $\mathcal{B}$ to be Khovanov's arc algebra $H^n$ from \cite{KhovFunctor}, viewed as a differential bigraded algebra with the differential and one of the two gradings identically equal to zero. The Type D and Type A structures $\D_{T_1}$ and $\widehat{A}_{T_2}$ will be referred to as $\D(T_1)$ and $\widehat{A}(T_2)$ in this setting. Both come from Khovanov's tangle invariants $[T_i]^{Kh}$, which are chain complexes of projective graded $H^n$ modules up to homotopy equivalence. 
\begin{theorem}\label{IntroKhPairingThm}
\[ 
CKh(T_2 T_1) \cong \widehat{A}(T_2) \boxtimes_{H^n} \D(T_1),
\]
after multiplying the intrinsic gradings on $\widehat{A}(T_2) \boxtimes_{H^n} \D(T_1)$ by $-1$.
\end{theorem}
\noindent A more precise version of Theorem~\ref{IntroKhPairingThm} is given in Proposition~\ref{TypeATypeDPairing} below.

Roberts \cite{RtypeD} \cite{RtypeA} has a different construction of a bordered theory for Khovanov homology, including a differential bigraded algebra $\B \Gamma_n$ as well as Type D and Type A structures for tangles. The goal of Section~\ref{RobertsFromKhovanovSection} and Section~\ref{ModuleSection} is to construct Roberts' theory using Khovanov's theory. 

We take the first step toward this goal in Section~\ref{LinQuadratSection}. In Section~\ref{LinQuadratAlgSection}, we discuss quadratic and linear-quadratic algebras following Polishchuk-Positselski \cite{PP}. In Section~\ref{HnAsLinQuadratSection}, we show that $H^n$ may be viewed as a linear-quadratic algebra:
\begin{theorem}[Proposition~\ref{HnIsLinQuad}]\label{IntroLinQuadThm}
With the set of generators specified at the beginning of Section~\ref{HnAsLinQuadratSection}, $H^n$ is a linear-quadratic algebra. 
\end{theorem}
This theorem allows us to write $H^n$ as the quotient of the tensor algebra on the specified generators by an ideal generated by certain explicitly-given relations, which are listed in items \ref{TwoBridgeRels}-\ref{LinQuadRels} of the proof of Proposition~\ref{HnIsLinQuad}. See Corollary~\ref{HnDescriptionCorr} for a more precise statement.

A combinatorial lemma about noncrossing partitions, Lemma~\ref{NonCrossPartLemma}, is needed to prove Theorem~\ref{IntroLinQuadThm}. While Theorem~\ref{IntroLinQuadThm} is not necessary for the remainder of the paper, Lemma~\ref{NonCrossPartLemma} is important for Section~\ref{RobertsFromKhovanovSection}. Proofs of Lemma~\ref{NonCrossPartLemma} were found by D{\"o}m{\"o}t{\"o}r P{\'a}lv{\"o}lgyi \cite{Palvolgyi} and independently by Aaron Potechin \cite{Potechin}. This lemma, and Theorem~\ref{IntroLinQuadThm}, may be of interest to readers independently of the other constructions in this paper. 

In Section~\ref{QLQDualitySection}, we consider a notion from \cite{PP} of quadratic duality for linear-quadratic algebras. In Section~\ref{TypeDDBimodSect}, we discuss a bordered-algebra version of this duality using Type DD bimodules. Generalized Koszul duality between two algebras $\B$ and $\B'$ in bordered Floer homology is defined (see \cite{LOTMorphism}) by the existence of a quasi-invertible rank-one Type DD bimodule over $\B$ and $\B'$. The algebras used in Lipshitz, Ozsv{\'a}th, and Thurston's construction have interesting Koszul self-duality properties. However, it seems that no such properties hold for $H^n$. Viewing $H^n$ as a linear-quadratic algebra, we will see in Section~\ref{DualOfHnSect} that its quadratic dual is infinite-dimensional, whereas $H^n$ is always finitely generated over $\Z$.

While it would be interesting to ask whether the duality between $H^n$ and this infinite-dimensional algebra is a (generalized) Koszul duality, it would also be interesting to explore related theories in which everything stays finite-dimensional. We will pursue the latter goal here.

In Section~\ref{RobertsFromKhovanovSection}, we will outline an alternate construction, based on $H^n$, of Roberts' algebra $\B \Gamma_n$. We define an algebra $\B = \B_R(H^n)$, and we show in Proposition~\ref{BrHnLinQuad} that the algebra $\B$ is linear-quadratic. The proof is very similar to the proof of Proposition~\ref{HnIsLinQuad} asserting that $H^n$ is linear-quadratic, and it also uses Lemma~\ref{NonCrossPartLemma} in an essential way. We deduce that $\B$ is isomorphic to the subalgebra $\B_R \Gamma_n$ of $\B \Gamma_n$ generated by right-pointing generators $\overrightarrow{e}$.

The quadratic dual $\B^!$ of $\B$ is closely related to the subalgebra $\B_L \Gamma_n$ of $\B \Gamma_n$ generated by left-pointing generators $\overleftarrow{e}$. In more detail, a mirroring operation $m$ is defined on certain algebras in Definition~\ref{AlgebraMirroringDef}. We will see in Proposition~\ref{KhovRobertsLSQuotient} that $\B_L \Gamma_n$ is a quotient of the mirror $m(\B^!)$ of $\B^!$ by certain additional relations, listed in that proposition. As Remark~\ref{MagicFGRemark} points out, $\B^!$ is finitely generated for idempotent reasons.

In Section~\ref{FullAlgebraSection}, we define a product algebra $\B \astrosun m(\B^!)$ of $\B$ and $m(\B^!)$. We get a description of Roberts' full algebra $\B \Gamma_n$ as a quotient: 
\begin{theorem}[Corollary~\ref{FullAlgQuotientCorr}]\label{IntroBGnQuotientThm}
$\B \Gamma_n$ is isomorphic to the quotient of $\B \astrosun m(\B^!)$ by the extra relations on $m(\B^!)$ listed in Proposition~\ref{KhovRobertsLSQuotient}.
\end{theorem}

The duality properties of $\B \Gamma_n$ and $\B \astrosun m(\B^!)$ seem more promising than those of $H^n$. In Proposition~\ref{FullAlgDDBimodProp} we define a rank-one Type DD bimodule over $\B \astrosun m(\B^!)$ and its mirror version $m(\B \astrosun m(\B^!))$. Conjecture~\ref{BiggerAlgConjecture} predicts that this DD bimodule is quasi-invertible and thus yields a Koszul duality. By taking quotients of the Type DD algebra outputs, we can obtain a related rank-one DD bimodule over $\B \Gamma_n$ and its mirror version $m(\B \Gamma_n)$. Thus, we could also ask if Conjecture~\ref{BiggerAlgConjecture} is true with $\B \astrosun m(\B^!)$ replaced by $\B \Gamma_n$. A proof of either conjecture would establish that, with regard to Koszul duality, Roberts' bordered theory (or the version over $\B \astrosun m(\B^!)$) has closer formal parallels with bordered Floer homology than Khovanov's $H^n$ theory does.

In Section~\ref{ModuleSection}, we show how to obtain Type A and Type D structures over $\B \astrosun m(\B^!)$ from chain complexes of graded projective $H^n$-modules satisfying certain algebraic conditions: 
\begin{theorem}\label{IntroTypeADSummaryThm}
The following constructions are well-defined:
\begin{itemize}
\item Let $M$ be a chain complex of projective graded right $H^n$-modules satisfying the algebraic condition $C_{module}$ of Definition~\ref{CModuleAlgCondition}. To $M$ we may associate a Type A structure $\widehat{A}(M)$ over $\B \astrosun m(\B^!)$.
\item Let $N$ be a chain complex of projective graded left $H^n$-modules satisfying the condition $C_{module}$ for complexes of left modules; see the discussion preceding Definition~\ref{XBoxWithDDDef}. To $N$ we may associate a Type D structure $\widehat{D}(N)$ over $\B \astrosun m(\B^!)$.
\end{itemize}
\end{theorem}
\noindent Theorem~\ref{IntroTypeADSummaryThm} is a summary of Definition~\ref{FullTypeAStructureDefn}, Proposition~\ref{FullMTwoActionWellDef}, and Definition~\ref{TypeDOverBomBDef}.

The chain complexes $[T_i]^{Kh}$ associated to tangles by Khovanov satisfy $C_{module}$, so Theorem~\ref{IntroTypeADSummaryThm} gives us Type A and Type D structures $\widehat{A}([T_2]^{Kh})$ and $\D([T_1]^{Kh})$ over $\B \astrosun m(\B^!)$. By Proposition~\ref{TangleTypeADescends}, the extra relations of Theorem~\ref{IntroBGnQuotientThm} act as zero on the Type A structure $\widehat{A}([T_2]^{Kh})$, so we get a Type A structure over the quotient algebra $\B \Gamma_n$. We may also take quotients of the algebra outputs of the Type D structure $\D([T_1]^{Kh})$ to get a Type D structure over $\B \Gamma_n$. 

\begin{theorem}[Proposition~\ref{RobertsKhovanovTypeA} and Proposition~\ref{RobertsKhovanovTypeD}]
The Type A structure $\widehat{A}([T_2]^{Kh})$ over $\B \Gamma_n$, and the Type D structure $\D([T_1]^{Kh})$ over $\B \Gamma_n$, are isomorphic to the Type A and D structures Roberts associates to $T_2$ and $T_1$ in \cite{RtypeA} and \cite{RtypeD}. 
\end{theorem}

By Proposition~\ref{XBoxQuotientDescend}, the pairing $\boxtimes$ is the same over $\B \astrosun m(\B^!)$ and its quotient $\B \Gamma_n$. We show that the pairing of the bordered modules over $\B \astrosun m(\B^!)$ agrees with the tensor product of the original chain complexes over $H^n$:
\begin{theorem}[Proposition~\ref{TensorXBoxAgree}]
Given $M$ and $N$ as in Theorem~\ref{IntroTypeADSummaryThm}, we have
\[
\widehat{A}(M) \boxtimes_{\B \astrosun m(\B^!)} \D(N) \cong M \otimes_{H^n} N,
\]
after multiplying the intrinsic gradings on $M \otimes_{H^n} N$ by $-1$.
\end{theorem}
Thus, we get an alternate proof that the pairing of Roberts' Type D and Type A structures computes Khovanov homology. 

Finally, in Section~\ref{TypeAEquivSect} and Section~\ref{TypeDEquivSect} we show that the homotopy types of $\widehat{A}([T_2]^{Kh})$ and $\D([T_1]^{Kh})$, as Type A and Type D structures over $\B \astrosun m(\B^!)$, are invariants of the tangles underlying the diagrams $T_1$ and $T_2$:
\begin{theorem}[Corollary~\ref{BomBHtpyEquivCorr} and Corollary~\ref{TypeDHtpyCorrFinal}]
Performing a Reidemeister move on $T_2$ or $T_1$ yields a homotopy equivalence between the corresponding Type A structures $\widehat{A}([T_2]^{Kh})$ or Type D structures $\D([T_1]^{Kh})$ over $\B \astrosun m(\B^!)$.
\end{theorem}
With the help of Proposition~\ref{HomotopyEqDescendProp}, we also obtain an alternate proof that Roberts' Type A and Type D structures over $\B \Gamma_n$ are homotopy-invariant under Reidemeister moves.

\subsection*{Acknowledgments}
I would like to thank Zolt{\'a}n Szab{\'o}, Lawrence Roberts, Victor Reiner, and Robert Lipshitz for useful comments and discussions during the writing of this paper. I would especially like to thank D{\"o}m{\"o}t{\"o}r P{\'a}lv{\"o}lgyi and Aaron Potechin for independently finding proofs of Lemma~\ref{NonCrossPartLemma}.

\section{Some bordered algebra}\label{BorderedAlgebraSection}

The standard reference for the algebra of bordered Floer homology is Lipshitz, Ozsv{\'a}th, and Thurston \cite{LOTBimodules}. We will use only a subset of the full algebraic machinery; however, we will work with coefficients in $\Z$ rather than $\Z/2\Z$. For this sign lift, we will follow the conventions of Roberts in \cite{RtypeD} and \cite{RtypeA}.

\subsection{Differential graded algebras and modules}\label{DiffGrAlgModSect}

\begin{convention}\label{FGConvention}
Unless otherwise specified, all algebras and modules discussed in this paper will be assumed to be finitely generated over $\Z$.
\end{convention}

The following is the notion of differential graded algebra which will be most useful for us; we will not need to use more general $\A_{\infty}$ algebras. In this paper, the coefficient ring $R$ is always a direct product of finitely many copies of $\Z$.

\begin{definition}\label{DgAlgDef} A differential bigraded algebra, or dg algebra, is a bigraded unital associative algebra $\B$ over a coefficient ring $R$ (with $R$ a finite direct product of copies of $\Z$), equipped with an $R$-bilinear differential $\mu_1$ which is homogeneous of degree $(0,+1)$ with respect to the bigrading. The two gradings on a dg algebra will be called the intrinsic and homological gradings (in that order). Thus, the differential should preserve the intrinsic grading and increase the homological grading by 1. 

The differential must satisfy the following Leibniz rule:
\[
\mu_1(xy) = (-1)^{\deg_h y} (\mu_1(x))y + x(\mu_1(y)),
\]
where $\deg_h$ denotes the homological degree, for elements $x$ and $y$ of $\B$ which are homogeneous with respect to the homological grading. The coefficient ring $R$ is required to coincide with the summand $\B_{0,0}$ of $\B$ in bigrading $(0,0)$.
\end{definition}

\begin{definition} Suppose $R = \Z^{\times k}$. The elements $e_1 = (1,0,\ldots,0), \ldots, e_k = (0,\ldots,0,1)$ will be called the minimal, or elementary, idempotents of $\B$. The coefficient ring $R$ will also be referred to as the idempotent ring of $\B$. For each elementary idempotent $e_i$, there is a left $R$-module $Re_i \simeq \Z$.
\end{definition}

\begin{remark} The usual convention in bordered Floer homology is to have the differential decrease the homological grading by 1; we have chosen to reverse this convention since the differentials in Khovanov homology increase homological grading by 1. 
\end{remark}

\begin{remark} Bordered Floer homology requires more general gradings by a (possibly nonabelian) group $G$ and a distinguished element $\lambda$ in the center of $G$; we use here only the special case where $G$ is the abelian group $\Z^2$ and $\lambda$ is $(0,1)$.
\end{remark}

When dealing with bigraded algebras or modules, we will use the following degree shift convention: if $X = \oplus_{i,j} X_{i,j}$ is any type of bigraded object, then $X[m,n]$ is the same type of bigraded object, and the summand of $X[m,n]$ in bigrading $(i,j)$ is $X_{i-m,j-n}$.

Since we are working over $\Z$, the following notation will also be useful, following Roberts \cite{RtypeD} and \cite{RtypeA}. If $X$ is any type of bigraded object, then $\left| \id \right|: X \to X$ is defined by multiplication by $(-1)^{\deg_h}$, where $\deg_h$ denotes the homological degree. Similarly, $\left|\id\right|^j: X \to X$ is defined by multiplication by $(-1)^{j \deg_h}$, and $\left|\id\right|^{j \otimes k}$ is the $k$-fold tensor product of $\left|\id\right|^j$. In this notation, if $\mu_2$ denotes the multiplication on a dg algebra $\B$, then the Leibniz rule for the differential $\mu_1$ on $\B$ can be written as 
\[
\mu_1 \circ \mu_2 = \mu_2(\mu_1 \otimes \left|\id\right|) + \mu_2(\id \otimes \mu_1).
\]
\begin{definition} A left differential bigraded module, or left dg module, over a dg algebra $\B$, is a bigraded left $\B$-module $M$ equipped with a differential $d$ of bidegree $(0,+1)$, such that the Leibniz rule
\[
d \circ m = m \circ (\mu_1 \otimes \left|\id\right|) + m \circ (\id \otimes d)
\]
is satisfied, where $m: \B \otimes_R M \to M$ is the action of $\B$ on $M$ and $\mu_1$ is the differential on $\B$.
\end{definition}
Similarly,
\begin{definition} A right differential bigraded module, or right dg module, over a dg algebra $\B$, is a bigraded right $\B$-module $M$ equipped with a differential $d$ of bidegree $(0,+1)$ such that the Leibniz rule
\[
d \circ m = m \circ (d \otimes \left|\id\right|) + m \circ (\id \otimes \mu_1)
\]
is satisfied, where $m: M \otimes_R \B \to M$ is the action of $\B$ on $M$ and $\mu_1$ is the differential on $\B$.
\end{definition}

If $M$ is a right dg module and $M'$ is a left dg module over $\B$, then we can take the tensor product of $M$ and $M'$ over $\B$ to produce a chain complex of graded abelian groups, or equivalently a differential bigraded $\Z$-module:

\begin{definition} Let $M$ be a right dg module and $M'$ be a left dg module over $\B$. The differential on the tensor product $M \otimes_{\B} M'$ is defined to be
\[
d_{M \otimes_{\B} M'} := d_M \otimes \left|\id_{M'}\right| + \id_M \otimes d_{M'}.
\]
\end{definition}

\subsection{Type D structures}

\begin{definition} Let $\B$ be a differential bigraded algebra over $R$ as in Definition~\ref{DgAlgDef}. Let $\mu_1$ and $\mu_2$ denote the differential and multiplication on $\B$, respectively.

A \emph{Type D structure} over $\B$ is, firstly, a bigraded left $R$-module $\D$ which is isomorphic to a finite direct sum of $R$-modules $Re_{i_{\alpha}}[j_{\alpha},k_{\alpha}]$, where the $e_{i_{\alpha}}$ are elementary idempotents of $\B$ (all in bigrading $(0,0)$) and $[j_{\alpha},k_{\alpha}]$ is a grading shift. The module $\D$ should be equipped with a bigrading-preserving $R$-linear map
\[
\delta: \D \to (\B \otimes_R \D)[0,-1],
\]
such that
\[
(\mu_1 \otimes \left|\id\right|) \circ \delta + (\mu_2 \otimes \id) \circ (\id \otimes \delta) \circ \delta = 0.
\]
\end{definition}

\begin{remark} The condition that $\D = \oplus_{\alpha} Re_{i_{\alpha}}[j_{\alpha},k_{\alpha}]$ would be unnecessary if $R$ were a direct product of copies of $\Z/2\Z$, rather than $\Z$. But over $\Z$, we want to exclude cases like $\B = R = \Z$, $\D = \Z/2\Z$, $\delta = 0$ from being valid Type D structures. The reason for this restriction is that we want Proposition~\ref{TensorIsProjModule} below, which is true over $\Z/2\Z$, to hold over $\Z$ as well.
\end{remark}

\begin{proposition}\label{TensorIsProjModule} If $(\D, \delta)$ is a Type D structure over $\B$, then $\B \otimes_R \D$ is a projective left dg $\B$-module when equipped with the differential
\[
d := \mu_1 \otimes \left|\id\right| + (\mu_2 \otimes \id) \circ (\id \otimes \delta),
\]
where $\mu_1$ and $\mu_2$ denote the differential and multiplication on $\B$, respectively.
\end{proposition}

\begin{proof} 
First, since $\D$ (as an $R$-module) is a direct sum of $R$-modules $R e_{i_{\alpha}}[j_{\alpha},k_{\alpha}]$, $\B \otimes_R \D$ is a direct sum of $\B$-modules $\B e_{i_{\alpha}}[j_{\alpha},k_{\alpha}]$. These are each projective because they are summands of $\B$: if $R = \Z^{\times k}$, we have $\B = \oplus_{i=1}^k \B e_i$ as left $\B$-modules. Thus, $\B \otimes_R \D$ is a projective $\B$-module.

Before showing that $d^2 = 0$, we check that $d$ satisfies the Leibniz rule. The action of the algebra $\B$ on $\B \otimes_R \D$ is given by the map
\[
m := \mu_2 \otimes \id: \B \otimes_R (\B \otimes_R \D) = (\B \otimes_R \B) \otimes_R \D \to \B \otimes_R \D.
\]

We want to show that $d \circ m = m \circ (\mu_1 \otimes \left|\id\right|) + m \circ (\id \otimes d),$ as maps from $\B \otimes_R \B \otimes_R \D$ to $\B \otimes_R \D$. We can write out the left side:
\begin{align*}
d \circ m &= (\mu_1 \otimes \left|\id\right| + (\mu_2 \otimes \id) \circ (\id \otimes \delta)) \circ (\mu_2 \otimes \id) \\
&= (\mu_1 \circ \mu_2) \otimes \left|\id\right| + (\mu_2 \otimes \id) \circ (\id \otimes \delta) \circ (\mu_2 \otimes \id) \\
&= (\mu_2 \circ (\mu_1 \otimes \left|\id\right|)) \otimes \left|\id\right| + (\mu_2 \circ (\id \otimes \mu_1)) \otimes \left|\id\right| \\
& \qquad + (\mu_2 \otimes \id) \circ (\id \otimes \delta) \circ (\mu_2 \otimes \id) \\
&= (\mu_2 \otimes \id) \circ (\mu_1 \otimes \left|\id\right| \otimes \left|\id\right|) + (\mu_2 \otimes \id) \circ (\id \otimes \mu_1 \otimes \left|\id\right|) \\ 
& \qquad + (\mu_2 \otimes \id) \circ (\id \otimes \delta) \circ (\mu_2 \otimes \id). \\
\end{align*}
Meanwhile, the right side is
\begin{align*}
m \circ (\mu_1 \otimes \left|\id\right|) &+ m \circ (\id \otimes d) = (\mu_2 \otimes \id) \circ (\mu_1 \otimes \left|\id\right| \otimes \left|\id\right|) + (\mu_2 \otimes \id) \circ (\id \otimes d) \\
&= (\mu_2 \otimes \id) \circ (\mu_1 \otimes \left|\id\right| \otimes \left|\id\right|) \\
& \qquad + (\mu_2 \otimes \id) \circ (\id \otimes (\mu_1 \otimes \left|\id\right| + (\mu_2 \otimes \id) \circ (\id \otimes \delta))) \\
&= (\mu_2 \otimes \id) \circ (\mu_1 \otimes \left|\id\right| \otimes \left|\id\right|)  + (\mu_2 \otimes \id) \circ (\id \otimes \mu_1 \otimes \left|\id\right|) \\
& \qquad + (\mu_2 \otimes \id) \circ (\id \otimes ((\mu_2 \otimes \id) \circ (\id \otimes \delta))). \\
\end{align*}

The first two terms on the left side cancel with those on the right side, and we only need show that
\[
(\mu_2 \otimes \id) \circ (\id \otimes \delta) \circ (\mu_2 \otimes \id) = (\mu_2 \otimes \id) \circ (\id \otimes ((\mu_2 \otimes \id) \circ (\id \otimes \delta))).
\]
This identity follows since
\begin{align*} (\mu_2 \otimes \id) &\circ (\id \otimes ((\mu_2 \otimes \id) \circ (\id \otimes \delta))) \\
&= (\mu_2 \otimes \id) \circ (\id \otimes \mu_2 \otimes \id) \circ (\id \otimes \id \otimes \delta) \\ 
&= (\mu_2 \otimes \id) \circ (\mu_2 \otimes \id \otimes \id) \circ (\id \otimes \id \otimes \delta) \\ 
&= (\mu_2 \otimes \id) \circ (\id \otimes \delta) \circ (\mu_2 \otimes \id).
\end{align*}

Now suppose $a \otimes x$ is a generator of $\B \otimes_R \D$; we want to show $d^2(a \otimes x) = 0$. We may write $a \otimes x$ as $m(a, 1 \otimes x)$ and apply the Leibniz rule: $d(a \otimes x) = (-1)^{\deg_h x} m(\mu_1(a), 1 \otimes x) + m(a, \delta(x))$, so
\begin{align*}
d^2(a \otimes x) &= (-1)^{\deg_h x} d(m(\mu_1(a), 1 \otimes x)) + d(m(a, \delta(x))) \\
&= (-1)^{\deg_h x} m(\mu_1(a),\delta(x)) + (-1)^{\deg_h x + 1} m(\mu_1(a), \delta(x)) \\
& \qquad + m(a, d(\delta(x))).
\end{align*}
The first two terms cancel each other, so it suffices to show that $d(\delta(x)) = 0$. Writing out $d$, this equation amounts to 
\[
(\mu_1 \otimes \left|\id\right|) \circ \delta + (\mu_2 \otimes \id) \circ (\id \otimes \delta) \circ \delta = 0.
\]
This is exactly the Type D structure relation.

\end{proof}

The following propositions will be useful in the description of Khovanov's functor-valued invariant as a bordered theory:

\begin{proposition}\label{DgModIsChainCx} Let $\B$ be a dg algebra over $R$. Suppose that $\B$ is concentrated in homological degree $0$ (it may have nontrivial intrinsic gradings). Then a dg module over $\B$ is the same as a chain complex of singly-graded $\B$-modules, with $\B$-linear grading-preserving differential maps.
\end{proposition}

\begin{proof}
Since $\B$ is concentrated in homological degree $0$, the differential on $\B$ must be zero. Let $M$ be a dg module over $\B$, with summand $M_{j,k}$ in bigrading $(j,k)$. Then, for each homological grading $k$, the summand $\oplus_j M_{j,k}$ of $M$ is preserved by $\B$; it is a singly-graded $\B$-module. Define a chain complex with chain module $C_k = \oplus_j M_{j,k}$. The differential $C_k \to C_{k+1}$ is the differential on $M$; it is $\B$-linear by the Leibniz rule, since $\B$ has no differential.

In the other direction, taking direct sums over chain modules yields a map from chain complexes to dg modules. These operations are inverse to each other.
\end{proof}

\begin{proposition}\label{TypeDIsChainCx} Let $\B$ be a dg algebra over $R$. Suppose that $\B$ is concentrated in homological degree $0$, and that all intrinsic gradings of $\B$ are nonnegative. Then a Type D structure over $\B$ is the same as a chain complex of singly-graded projective left $\B$-modules, with $\B$-linear grading-preserving differential maps.
\end{proposition}

\begin{proof} Given a Type D structure $\D$ over $\B$, Proposition~\ref{TensorIsProjModule} shows that $\B \otimes_R \D$ is a dg module over $\B$, or equivalently a chain complex of graded left $\B$-modules by Proposition~\ref{DgModIsChainCx}. In fact, each term of the chain complex is projective, since it is a sum of modules $\B e_{i_{\alpha}}[j_{\alpha},k_{\alpha}]$.

Conversely, suppose $\cdots \to C_k \to C_{k+1} \to \cdots$ is a chain complex of graded projective left $\B$-modules. Since each $C_k$ is assumed to be finitely generated, it may be written as a direct sum of indecomposable graded projective left $R$-modules $C_{k,\alpha}$. By Lemma 1 of Section 2.5 of Khovanov \cite{KhovFunctor}, each $C_{k,\alpha}$ is isomorphic to $\B e_{i_{k,\alpha}}[j_{k,\alpha}]$ for some elementary idempotent $e_{i_{k,\alpha}}$ and grading shift $j_{k,\alpha}$. Define $\D$ as a bigraded $R$-module to be the direct sum, over all $k$ and $\alpha$, of $R e_{i_{k,\alpha}} [j_{k,\alpha},k]$.

We may identify $\oplus_k C_k$ with $\B \otimes_R \D$, since $\oplus_k C_k = \oplus_{k,\alpha} \B e_{i_{k,\alpha}}[j_{k,\alpha}]$ and $\D = \oplus_{k,\alpha} R e_{i_{k,\alpha}}[j_{k,\alpha}, k]$. Let $d$ denote the differential on the dg module $\oplus_k C_k$. Then the Type D operation $\delta: \D \to \B \otimes_R \D$ is obtained by restricting $d$ to $\D \cong 1 \otimes_R \D \subset \B \otimes_R \D$. It has the correct grading properties because $d$ does. 

Since $d$ satisfies the Leibniz rule, we may write $d = \mu_1 \otimes \left|\id\right| + (\mu_2 \otimes \id) \circ (\id \otimes \delta)$. Thus, the Type D relations for $\delta$ are equivalent to $d \circ \delta = 0$, which holds because $\delta$ is a restriction of $d$.
\end{proof}

\subsection{Type A structures and pairing}

\begin{definition} Let $\B$ be a dg algebra over $R$ as in Definition~\ref{DgAlgDef}. Let $\mu_1$ and $\mu_2$ denote the differential and multiplication on $\B$, respectively.

A \emph{Type A structure} $\widehat{A}$ over $\B$, synonymous with $\A_{\infty}$-\emph{module over } $\B$, is a bigraded right $R$-module $\widehat{A}$, finitely generated over $R$ as usual by Convention~\ref{FGConvention}, together with $R$-linear bigrading-preserving maps $m_i: \widehat{A} \otimes_R \B^{\otimes(i-1)} \to \widehat{A} [0,i-2]$, $i \in \Z_{\geq 1}$,  satisfying 
\begin{align*}
&\sum_{i+j = n+1} (-1)^{j(i+1)} m_i \circ (m_j \otimes \left|\id\right|^{j \otimes (i-1)}) \\
& \qquad + (-1)^{n+1} \sum_{k > 0} m_n \circ ( \id^{\otimes k} \otimes \mu_1 \otimes \left|\id\right|^{\otimes n - k - 1} ) \\
& \qquad + \sum_{k > 0} (-1)^k m_{n-1} \circ (\id^{\otimes k} \otimes \mu_2 \otimes \id^{\otimes (n-k-2)}) \\
&= 0,
\end{align*}
for every $n \geq 1$. The Type A structure $\widehat{A}$ is called \emph{strictly unital} if $m_2(-,1) = \id_{\widehat{A}}$ and $m_n = 0$ for $n > 2$ when any of the algebra inputs to $m_n$ is $1$.

\end{definition}

\begin{example}\label{OrdinaryModTypeAEx} If $M$ is a (right) dg $\B$-module, then $M$ is a strictly unital Type A structure over $\B$ with $m_i = 0$ for $i \neq 1, 2$. If $M$ is an ordinary bigraded module over $\B$, with no differential, then $M$ is a strictly unital Type A structure with $m_i = 0$ for $i \neq 2$.
\end{example}

\begin{remark} We will only need to work with Type A structures which come from dg modules as in Example~\ref{OrdinaryModTypeAEx}. Thus, all our Type A structures will be strictly unital, so we will omit mention of this condition in what follows. However, although our Type A structures will have no nontrivial higher action terms, we will eventually need to work with $\mathcal{A}_{\infty}$ \emph{morphisms} between these Type A structures. We will need to consider morphisms which do have nontrivial higher $\mathcal{A}_{\infty}$ terms; see Section~\ref{TypeAEquivSect}.
\end{remark}

Given a Type D structure $(\D,\delta)$ and a Type A structure $(\widehat{A}, m_n: n \geq 1)$ over $\B$, the natural way to pair them is known as the box tensor product. It yields a differential bigraded abelian group $\widehat{A} \boxtimes \D$; see Lipshitz-Ozsv{\'a}th-Thurston \cite{LOTBimodules} for more details and algebraic properties of $\boxtimes$ over $\Z/2\Z$, and see Roberts \cite{RtypeA} for a definition over $\Z$ with the sign conventions we will use.

To define $\boxtimes$ in our setting, the following notation will be useful:
\begin{definition} Let $(\D, \delta)$ be a Type D structure over $\B$. The map $\delta^k: \D \to \B^k \otimes_R \D$ is 
\[
\delta^k = (\id \otimes \cdots \otimes \id \otimes \delta) \circ \cdots \circ (\id \otimes \delta) \circ \delta,
\]
where $\delta$ is applied $k$ times. In particular, $\delta = \delta^1$.
\end{definition}

\begin{definition}\label{XBoxWithSigns} $\widehat{A} \boxtimes \D$, as a bigraded abelian group, is the tensor product $\widehat{A} \otimes_{R} \D$. The differential on $\widehat{A} \boxtimes \D$ is 
\[
\partial^{\boxtimes} = \sum_{n=1}^{\infty}(m_n \otimes \left|\id\right|^n) \circ (\id \otimes \delta^{n-1}).
\]
By Convention~\ref{FGConvention}, only finitely many terms of the sum are nonzero.
\end{definition}

\begin{proposition} The operator $\partial^{\boxtimes}$, as defined in Definition~\ref{XBoxWithSigns}, satisfies 
\[
(\partial^{\boxtimes})^2 = 0.
\]
\end{proposition}

\begin{proof} 
In this proof, when referring to identity operators, we will use subscripts to explicitly indicate which identity operators we mean.

First, note that 
\begin{align*}
&(\id_{\widehat{A}} \otimes \delta^{i-1}) \circ (m_j \otimes \left|\id_{\D}\right|^j) \\
&= (-1)^{j(i+1)} (m_j \otimes \left|\id_{\B^{i-1} \otimes \D}\right|^j) \circ (\id_{\widehat{A} \otimes \B^{j-1}} \otimes \delta^{i-1}),
\end{align*}
as maps from $\widehat{A} \otimes_R \B^{j-1} \otimes_R \D$ to $\widehat{A} \otimes_R \B^{i-1} \otimes_R \D$. This identity is immediate over $\Z/2\Z$, and we need only verify that the signs are right. On the right side of the equality, we have $\left|\id_{\B^{i-1} \otimes \D}\right|^j$ which is computed from the homological degree of an output of $\delta^{i-1}$. Since $\delta$ increases homological degree by $1$, $\delta^{i-1}$ increases homological degree by $i-1$. Thus, compared with the left side, the right side has an extra factor of $(-1)^{j(i-1)} = (-1)^{j(i+1)}$.

Thus,
\begin{align*}
(\partial^{\boxtimes})^2 &= \sum_{n \geq 1} \sum_{i+j = n+1} (m_i \otimes \left|\id_{\D}\right|^i) \circ (\id_{\widehat{A}} \otimes \delta^{i-1}) \circ (m_j \otimes \left|\id_{\D}\right|^j) \circ (\id_{\widehat{A}} \otimes \delta^{j-1}) \\
&= \sum_{n \geq 1} \sum_{i + j = n+1} (-1)^{j(i+1)} (m_i \otimes \left|\id_{\D}\right|^i) \circ (m_j \otimes \left|\id_{\B^{i-1}}\right|^{j} \otimes \left|\id_{\D}\right|^j) \\
& \qquad \circ (\id_{\widehat{A}} \otimes \id_{\B^{j-1}} \otimes \delta^{i-1}) \circ (\id_{\widehat{A}} \otimes \delta^{j-1}) \\
&= \sum_{n \geq 1} \sum_{i + j = n+1} (-1)^{j(i+1)} ((m_i \circ (m_j \otimes \left|\id_{\B^{i-1}}\right|^j)) \otimes \left|\id_{\D}\right|^{n+1}) \\
& \qquad \circ (\id_{\widehat{A}} \otimes  \delta^{n-1}) \\
&= \sum_{n \geq 1} \bigg( (-1)^n \sum_{k =1}^{n-1} m_n \circ (\id_{\widehat{A}} \otimes \id_{\B^{k-1}} \otimes \mu_1 \otimes \left|\id_{\B^{n-k-1}}\right|) \\
& \qquad - \sum_{k=1}^{n-2} (-1)^k m_{n-1} \circ (\id_{\widehat{A}} \otimes \id_{\B^{k-1}} \otimes \mu_2 \otimes \id_{\B^{n-k-2}}) \bigg) \otimes \left|\id_{\D}\right|^{n+1}  \\
&\circ (\id_{\widehat{A}} \otimes \delta^{n-1}), \\
\end{align*}
where the Type A relations for $\widehat{A}$ were used in the final equality. It remains to show that the derivative terms
\[
\sum_{n \geq 1} \bigg( (-1)^n \sum_{k =1}^{n-1} m_n \circ (\id_{\widehat{A}} \otimes \id_{\B^{k-1}} \otimes \mu_1 \otimes \left|\id_{\B^{n-k-1}}\right|) \bigg) \otimes \left|\id_{\D}\right|^{n+1} \circ (\id_{\widehat{A}} \otimes \delta^{n-1})
\]
are equal to the multiplication terms
\[
\sum_{n \geq 1} \bigg(\sum_{k=1}^{n-2} (-1)^k m_{n-1} \circ (\id_{\widehat{A}} \otimes \id_{\B^{k-1}} \otimes \mu_2 \otimes \id_{\B^{n-k-2}}) \bigg) \otimes \left|\id_{\D}\right|^{n+1} \circ (\id_{\widehat{A}} \otimes \delta^{n-1}).
\]
For a fixed $n \geq 1$ and $1 \leq k \leq n-1$, we claim that the derivative term is equal to
\begin{align*}
(-1)^{k+1} (m_n \otimes \left|\id_{\D}\right|^n) &\circ \bigg( \id_{\widehat{A}} \otimes \bigg( (\id_{\B^{n-2}} \otimes \delta) \circ \cdots \\
& \qquad \circ (\id_{\B^{k-1}} \otimes \mu_1 \otimes \left|\id_{\D}\right|) \circ (\id_{\B^{k-1}} \otimes \delta) \circ \\
& \qquad \cdots \circ (\id_{\B} \otimes \delta) \circ \delta \bigg) \bigg). \\
\end{align*}
To see that this formula holds, note that when $k = n-1$, the sign in front of the above formula should be $(-1)^n$, by inspection, and each time $k$ is decreased by 1, the sign should flip because $\left|\id_{\D}\right|$ occurs after one fewer instance of $\delta$. Now,
\begin{align*}
(\id_{\B^{k-1}} &\otimes \mu_1 \otimes \left|\id_{\D}\right|) \circ (\id_{\B^{k-1}} \otimes \delta) \\
&= \id_{\B^{k-1}} \otimes ((\mu_1 \otimes \left|\id_{\D}\right|) \circ \delta) \\
&= -\id_{\B^{k-1}} \otimes ((\mu_2 \otimes \id_{\D}) \circ (\id_{\B} \otimes \delta) \circ \delta), \\
\end{align*}
by the Type D relations for $\D$. Thus, the sum of the derivative terms is
\begin{align*}
&\sum_{n \geq 1} \sum_{k = 1}^{n-1} (-1)^k (m_n \otimes \left|\id_{\D}\right|^n) \circ \bigg( \id_{\widehat{A}} \otimes \bigg( (\id_{\B^{n-2}} \otimes \delta) \circ \cdots \\
&\circ (\id_{\B^k} \otimes \delta) \circ (\id_{\B^{k-1}} \otimes \mu_2 \otimes \id_{\D}) \circ (\id_{\B^k} \otimes \delta) \circ (\id_{\B^{k-1}} \otimes \delta) \circ \cdots \circ \delta \bigg) \bigg) \\
&= \sum_{n \geq 1} \sum_{k = 1}^{n-1} (-1)^k (m_n \otimes \left|\id_{\D}\right|^n) \circ (\id_{\widehat{A} \otimes \B^{k-1}} \otimes \mu_2 \otimes \id_{\B^{n-k-1} \otimes \D}) \circ (\id_{\widehat{A}} \otimes \delta^n). \\
\end{align*}
Since the $n = 1$ multiplication term is zero, and $\left|\id_{\D}\right|^{n+1} = \left|\id_{\D}\right|^{n-1}$, the sum of the multiplication terms is 
\begin{align*}
&\sum_{n \geq 2} \sum_{k = 1}^{n-2} (-1)^k (m_{n-1} \otimes \left|\id_{\D}\right|^{n-1}) \circ (\id_{\widehat{A} \otimes \B^{k-1}} \otimes \mu_2 \otimes \id_{\B^{n-k-2} \otimes \D}) \circ (\id_{\widehat{A}} \otimes \delta^{n-1}) \\
&= \sum_{n \geq 1} \sum_{k = 1}^{n-1} (-1)^k (m_n \otimes \left|\id_{\D}\right|^n) \circ (\id_{\widehat{A} \otimes \B^{k-1}} \otimes \mu_2 \otimes \id_{\B^{n-k-1} \otimes \D}) \circ (\id_{\widehat{A}} \otimes \delta^n). \\
\end{align*}
These sums agree, proving that $(\partial^{\boxtimes})^2 = 0$.
\end{proof}

\begin{proposition}[Example 2.27 of Lipshitz-Ozsv{\'a}th-Thurston \cite{LOTBorderedOrig}]\label{SimpleXboxForModules} Let $\B$ be a dg algebra over $R$ as in Definition~\ref{DgAlgDef}. Let $\D$ be a Type D structure over $\B$, and let $\widehat{A}$ be a right dg module over $\B$. Then $\widehat{A} \boxtimes \D$ and $\widehat{A} \otimes_{\B} (\B \otimes_R \D)$ are isomorphic as differential bigraded abelian groups.
\end{proposition}

For completeness, and since we are working over $\Z$, we will give a proof of this proposition:
\begin{proof}[Proof of Proposition~\ref{SimpleXboxForModules}]
As bigraded abelian groups, 
\[
\widehat{A} \otimes_{\B} (\B \otimes_R \D) \cong (\widehat{A} \otimes_{\B} \B) \otimes_R \D \cong \widehat{A} \otimes_R \D.
\]
Thus, $\widehat{A} \otimes_{\B} (\B \otimes_R \D)$ and $\widehat{A} \boxtimes \D$ have the same underlying group; we must verify that the differentials agree.

The differential on $\widehat{A} \otimes_{\B} (\B \otimes_R \D)$ may be written as
\begin{align*}
&d_{\widehat{A}} \otimes (\left|\id_{\B}\right| \otimes \left|\id_{\D}\right|) + \id_{\widehat{A}} \otimes (\mu_1 \otimes \left|\id_{\D}\right| + (\mu_2 \otimes \id_{\D}) \circ (\id_{\B} \otimes \delta)) \\
&= d_{\widehat{A}} \otimes (\left|\id_{\B}\right| \otimes \left|\id_{\D}\right|) + \id_{\widehat{A}} \otimes (\mu_1 \otimes \left|\id_{\D}\right|) \\
& \qquad + (\id_{\widehat{A}} \otimes (\mu_2 \otimes \id_{\D})) \circ (\id_{\widehat{A}} \otimes (\id_{\B} \otimes \delta)).
\end{align*}
Regrouping the parentheses, we get
\[
(d_{\widehat{A}} \otimes \left|\id_{\B}\right|) \otimes \left|\id_{\D}\right| + (\id_{\widehat{A}} \otimes \mu_1) \otimes \left|\id_{\D}\right| + ((\id_{\widehat{A}} \otimes \mu_2) \otimes \id_{\D}) \circ ((\id_{\widehat{A}} \otimes \id_{\B}) \otimes \delta).
\]
Identifying $\widehat{A}$ with $\widehat{A} \otimes_{\B} \B$, the differential on $\widehat{A}$ becomes $d_{\widehat{A}} \otimes \left|\id_{\B}\right| + \id_{\widehat{A}} \otimes \mu_1$. Similarly, the algebra multiplication $m: \widehat{A} \otimes \B \to \widehat{A}$ becomes 
\[
\id_{\widehat{A}} \otimes \mu_2: (\widehat{A} \otimes_{\B} \B) \otimes \B \to (\widehat{A} \otimes_{\B} \B).
\]
Thus, we can identify the above formula for the differential on $\widehat{A} \otimes_{\B} (\B \otimes_R \D) \cong \widehat{A} \otimes_R \D$ with
\[
d_{\widehat{A}} \otimes \left|\id_{\D}\right| + (m \otimes \id_{\D}) \circ (\id_{\widehat{A}} \otimes \delta).
\]
This is also the differential on $\widehat{A} \boxtimes \D$.
\end{proof}

\section{Khovanov's functor-valued invariant as a bordered theory}\label{HnAsBorderedSection}

We will assume some familiarity with Khovanov's paper \cite{KhovFunctor}. Here we will briefly introduce some useful conventions and notation. 

Khovanov's arc algebra $H^n$ has one grading. We will view $H^n$ as a differential bigraded algebra, concentrated in homological degree 0 and with no differential. The usual grading on $H^n$ becomes the intrinsic component of the bigrading. The intrinsic gradings of $H^n$ are nonnegative. Thus, both Proposition~\ref{DgModIsChainCx} and Proposition~\ref{TypeDIsChainCx} apply to $H^n$ as a dg algebra. 

The component of $H^n$ in degree 0 (or, with our conventions, in bidegree $(0,0)$) will be denoted $\I_n$, and referred to as the idempotent ring of $H^n$. It is isomorphic to $\Z^{\times (C_n)}$, where $C_n$ is the $n^{th}$ Catalan number. The elementary idempotents of $H^n$ are the idempotents $1_a$ described by Khovanov in Section 2.4 of \cite{KhovFunctor}. The index $a$ runs over elements of the set $B^n$ of crossingless matchings of $2n$ points; since this set will be important later, we recall its definition here.

\begin{definition} Let $P$ be a set of $2n$ distinct points on the line $\{0\} \times \R \subset \R \times \R$. A crossingless matching $a$ of $P$ is a partition of $P$ into $n$ pairs of points, such that there exists an embedding of $n$ arcs $[0,1]^{\sqcup n}$ disjointly into $\R_{\geq 0} \times \R$ with each arc connecting a pair of points matched in $a$. The set of crossingless matchings of $2n$ points will be denoted $B^n$ (different choices of $P$ yield canonical bijections between the relevant sets $B^n$).
\end{definition}

\begin{remark}\label{CLMatchNCPart} The set $B^n$ is also in bijection with the set $NC_n$ of noncrossing partitions of $n$ points. A noncrossing partition $a$ of a set $P$ of $n$ points on $\{0\} \times \R$ is defined to be any partition of $P$ into $k$ disjoint subsets, such that there exists an embedding of $k$ acyclic graphs disjointly into $\R_{\geq 0} \times \R$, with each graph bounding one of the $k$ subsets of $a$.

To go from a crossingless matching $a$ of $2n$ points $p_1, \ldots, p_{2n}$ to a noncrossing partition $a'$ of $n$ points $q_1, \ldots, q_n$, checkerboard-color the half-plane $\R_{\geq 0}$ with respect to some embedding of arcs representing $a$, such that the unbounded region of the half-plane is colored white. Put the point $q_i$ on the line $\{0\} \times \R$ between the points $p_{2i-1}$ and $p_{2i}$. In the noncrossing partition $a'$, two points $q_i$ and $q_j$ are placed in the same subset if they can be connected in $\R_{\geq 0} \times \R$ by a path through the black region of the checkerboard coloring. The skeleton of the black region provides the planar graphs which verify that $a'$ is a noncrossing partition.

On the other hand, given a noncrossing partition $a'$ of $n$ points, one can pick an embedding of graphs representing $a'$, and fatten each graph to obtain a planar surface. The boundary of this surface is a crossingless matching of $2n$ points. These two constructions are inverse to each other.
\end{remark}

Let $T$ be an oriented tangle diagram in the half-plane, with $2n$ endpoints, and assume we have chosen an ordering of the crossings of $T$. Khovanov's construction assigns a bounded chain complex of finitely-generated projective graded $H^n$-modules, with $H^n$-linear differential maps, to $T$. We will use the notation $[T]^{Kh}$ to refer to this complex; we will often view $[T]^{Kh}$ as a dg $H^n$-module using Proposition~\ref{DgModIsChainCx}. If $T$ lies in $\R_{\geq 0} \times \R$, then $[T]^{Kh}$ is a left dg module; if $T$ lies in $\R_{\leq 0} \times \R$, then $[T]^{Kh}$ is a right dg module.

If $T_1$ is an oriented tangle diagram in $\R_{\geq 0} \times \R$, Proposition~\ref{TypeDIsChainCx} gives us a Type D structure over $H^n$, which we will call $\D(T_1)$, such that
\[
[T_1]^{Kh} \cong H^n \otimes_{\I_n} \D(T_1).
\]
In Section~\ref{KhovTypeDSect} below we will discuss $\D(T_1)$ explicitly.

If $T_2$ is an oriented tangle diagram in $\R_{\leq 0} \times \R$, we will simply take the Type A structure $\widehat{A}(T_2)$ of $T_2$ to be the right dg module $[T_2]^{Kh}$. Suppose $T_2$ and $T_1$ have consistent orientations; put them together to obtain an oriented link diagram $L$. Order the crossings of $L$ so that those of $T_1$ come before those of $T_2$, and let $CKh(L)$ be the Khovanov complex of $L$. Khovanov shows in \cite{KhovFunctor} that
\[
CKh(L) \cong [T_2]^{Kh} \otimes_{H^n} [T_1]^{Kh},
\]
after multiplying the intrinsic or $q$-gradings on $[T_2]^{Kh} \otimes_{H^n} [T_1]^{Kh}$ by $-1$. By Proposition~\ref{SimpleXboxForModules}, we have
\[
CKh(L) \cong \widehat{A}(T_2) \boxtimes \D(T_1),
\]
as in bordered Floer homology, after applying the same intrinsic-grading reversal to $\widehat{A}(T_2) \boxtimes \D(T_1)$. We will summarize this discussion more formally below in Proposition~\ref{TypeATypeDPairing}.

\begin{remark}\label{FirstGradingRevRem}
The reversal of the gradings here comes from Khovanov's choice, in pages 672 and 673 of \cite{KhovFunctor}, to make $H^n$ positively rather than negatively graded. It is only a convention; one could define the basic generators of $H^n$ to live in degrees $-1$ and $-2$, rather than $1$ and $2$, and then no grading reversal would be necessary.
\end{remark}

\begin{remark}\label{OrderingIsoIndep} Up to isomorphism, the bigraded chain complex $CKh(L)$ does not depend on the ordering of the crossings. Indeed, suppose we reverse the ordering of two adjacent crossings $i$ and $i+1$. Then an isomorphism 
\[
F: (CKh(L),\textrm{ first ordering}) \to (CKh(L), \textrm{ second ordering})
\]
can be defined, on the summand of $CKh(L)$ corresponding to a vertex $\rho$ of the cube of resolutions, to be $(-1)^{f(\rho)} \cdot \id$, where $f(\rho) := 1$ if $\rho$ resolves crossings $i$ and $i+1$ both as $1$, rather than $0$, and $f(\rho) := 0$ otherwise.

The same argument applies unchanged to the tangle complexes $[T]^{Kh}$: the isomorphism type of $[T]^{Kh}$ does not depend on the ordering of the crossings.
\end{remark}

\begin{remark}
Khovanov avoids having to choose an ordering of the crossings by using the skew-commutative cubes formalism. We will not do this here, but we will usually suppress mention of the choice of ordering of the crossings.
\end{remark}

\subsection{Type D structures}\label{KhovTypeDSect}

In this section we unpack Proposition~\ref{TypeDIsChainCx} to give a concrete definition of $\D(T_1)$. First, we recall some properties of $H^n$ and Khovanov's dg module $[T]^{Kh}$.

The algebra $H^n$ has an additive basis $\beta$, over $\Z$, consisting of elements which we will denote $((W(a)b), \sigma)$. Here, $a$ and $b$ are elements of $B^n$, the set of crossingless matchings of $2n$ points, and the operation $W$ mirrors the matching from the right to the left plane. The horizontal concatenation $W(a)b$ is a collection of disjoint circles in $\R^2$. The remaining data $\sigma$ consists of a choice of sign, $+$ or $-$, on each of these circles.

Certain of the basis elements $((W(a)b),\sigma)$ form a natural set of multiplicative generators for $H^n$. These generators come in two forms: the first are elements $h_{\gamma} = (W(a)a', \textrm{all plus})$, where $a \in B^n$, the element $a' \in B^n$ is obtained from $a$ by surgering one pair of arcs along a bridge $\gamma$, and all circles of $W(a)a'$ are labeled $+$. The other generators are elements $h_{\alpha} = (W(a)a, \textrm{minus on } W(\alpha)\alpha)$, where $a \in B^n$ and all circles of $W(a)a$ are labeled $+$ except one circle, $W(\alpha)\alpha$ for some arc $\alpha$ of $a$, which is labeled $-$. Each generator $h_{\gamma}$ and $h_{\alpha}$ has a unique left idempotent and right idempotent in $\I_n$. We will denote the set of multiplicative generators $\{h_{\gamma},h_{\alpha}\}$ as $\beta_{mult}$; it is a subset of $\beta$.

Now, let $T$ be an oriented tangle diagram in $\R_{\geq 0} \times \R$. To specify a generator $x_i$ of $[T]^{Kh}$, we first specify a resolution $\rho_i$ of all crossings of $T$; we can view $\rho_i$ as a function from the set of crossings to the two-element set $\{0,1\}$. If $T_{\rho_i}$ denotes the diagram $T$ with the crossings resolved according to $\rho_i$, then $T_{\rho_i}$ consists of a crossingless matching of $2n$ points together with some free circles contained in $\R_{> 0} \times \R$. The remaining data needed to specify $x_i$ is a choice of $+$ or $-$ on each free circle. Then $[T]^{Kh}$ has a $\Z$-basis consisting of elements $h \cdot x_i$, where the right idempotent of $h$ agrees with the matching obtained from $T_{\rho_i}$ by discarding the free circles.

\begin{definition}[Concrete definition of $\D(T_1)$ as an $\I_n$-module]\label{ConcreteDDefIdems}
Let $T_1$ be an oriented tangle diagram in $\R_{\geq 0} \times \R$, with $n_+$ positive crossings and $n_-$ negative crossings, and let $0 \leq r \leq n_+ + n_-$. Define $\D(T_1)$ to be generated as an (intrinsically) graded abelian group, in homological degree $r - n_-$, by the generators $1 \cdot x_i$ of $([T_1]^{Kh})_{r-n_-}$, where $([T_1]^{Kh})_{r-n_-}$ is the chain space of $[T_1]^{Kh}$ in degree $r - n_-$. These generators have the same crossingless matching on the left and right sides of $\{0\} \times \R$, and all circles touching the boundary line have a $+$ sign. With this definition, $\D(T_1)$ is an $\I_n$-submodule of $[T_1]^{Kh}$.
\end{definition}

\begin{proposition}\label{TypeDRewrite} As $H^n$-modules, 
\[
[T_1]^{Kh} \cong H^n \otimes_{\I_n} \D(T_1),
\]
with $\D(T_1)$ as defined in Definition~\ref{ConcreteDDefIdems}.
\end{proposition}

\begin{proof}
This follows from Definition~\ref{ConcreteDDefIdems}.
\end{proof}

In the following, let $\iota := \iota_{\D(T_1)}$ (the inclusion of $\D(T_1)$ into $[T_1]^{Kh}$) and let $d$ be the differential on $[T_1]^{Kh}$. Let $\mu$ denote the multiplication on $H^n$.

\begin{definition}[Concrete definition of Type D operation on $\D(T_1)$]\label{ConcreteDeltaDef}
The Type D differential $\delta$ on $\D(T_1)$ is defined by restricting the differential $d$ to the $\I_n$-submodule $\D(T_1)$ of $[T_1]^{Kh}$:
\[
\delta := \D(T_1) \xrightarrow{\iota} [T_1]^{Kh} \xrightarrow{d} [T_1]^{Kh} \cong H^n \otimes_{\I_n} \D(T_1).
\] 
It is a $\I_n$-linear map because $\iota$ and $d$ are.
\end{definition}

\begin{lemma} Under the identification $[T_1]^{Kh} \cong H^n \otimes_{\I_n} \D(T_1)$ from Proposition~\ref{TypeDRewrite}, we have 
\[
d = (\mu_2 \otimes \id) \circ (\id \otimes \delta),
\]
where $\mu_2: H^n \otimes H^n \to H^n$ is the algebra multiplication.
\end{lemma}

\begin{proof} 
Let $h \cdot \iota(x)$ denote a generator of $[T_1]^{Kh}$. Then, by the Leibniz property for $[T_1]^{Kh}$, we have $d(h \cdot \iota(x)) = h \cdot d \iota(x)$, since $H^n$ has no differential. But since $d \iota(x) = \delta(x)$, we can conclude that
\[
d(h \cdot \iota(x)) = (\mu_2 \otimes \id) \circ (\id \otimes \delta) (h \cdot \iota(x)).
\]
\end{proof}

\begin{proposition} $(\D(T_1), \delta)$ satisfies the Type D relations:
\[
(\mu_1 \otimes \left|\id\right|) \circ \delta + (\mu_2 \otimes \id) \circ (\id \otimes \delta) \circ \delta = 0.
\]
\end{proposition}

\begin{proof} There is no differential on $H^n$, so the $\mu_1$ term is zero. For the other term, if $x$ is a generator of $\D(T_1)$, then
\begin{align*}
((\mu_2 \otimes \id) \circ (\id \otimes \delta) \circ \delta)(x) &= (d) \circ (d \circ \iota) (x) \\
&= 0,
\end{align*}
since $d^2 = 0$ on $[T_1]^{Kh}$.
\end{proof}

One can check that the Type D structure $\D(T_1)$ defined in Definition~\ref{ConcreteDDefIdems} and Definition~\ref{ConcreteDeltaDef} agrees with the Type D structure obtained from $[T]^{Kh}$ by using Proposition~\ref{TypeDIsChainCx}, as discussed above.

\subsection{Type A structures and pairing}
Let $T_2$ be an oriented tangle diagram in $\R_{\leq 0} \times \R$. Since dg modules are special cases of Type A structures, Proposition~\ref{DgModIsChainCx} tells us that the right dg module $[T_2]^{Kh}$ is a valid example of a Type A structure over $H^n$. We will define $\widehat{A}(T_2)$ to be $[T_2]^{Kh}$.

\begin{proposition}\label{TypeATypeDPairing}
Let $T_1$ and $T_2$ be oriented tangle diagrams in $\R_{\geq 0} \times \R$ and $\R_{\leq 0} \times \R$ respectively, with orderings chosen of the crossings of $T_1$ and $T_2$. Assume that $T_1$ and $T_2$ have consistent orientations, so that their horizontal concatenation is an oriented link diagram $L$ in $\R^2$. Order the crossings of $L$ such that those of $T_1$ come before those of $T_2$. Then
\[
CKh(L) \cong \widehat{A}(T_2) \boxtimes \D(T_1),
\]
after multiplying the intrinsic gradings on $\widehat{A}(T_2) \boxtimes \D(T_1)$ by $-1$.
\end{proposition}

\begin{proof}
Since, up to a grading reversal, $CKh(L) \cong [T_2]^{Kh} \otimes_{H^n} [T_1]^{Kh}$, which is the same as $\widehat{A}(T_2) \otimes_{H^n} (H^n \otimes_{\I_n} \D(T_1))$, this proposition follows from Proposition~\ref{SimpleXboxForModules}, Remark~\ref{FirstGradingRevRem}, and Khovanov's results from \cite{KhovFunctor}.
\end{proof}

\section{Quadratic and linear-quadratic algebras and duality}\label{LinQuadratSection}

\subsection{Quadratic and linear-quadratic algebras}\label{LinQuadratAlgSection}
We now consider a method of describing algebras using explicit generators and relations. It will be important for the following sections where we relate the bordered Khovanov theory discussed above with Roberts' constructions in \cite{RtypeD} and \cite{RtypeA}. The definitions and basic properties of quadratic and linear-quadratic algebras here all follow Polishchuk-Positselski \cite{PP}, with some minor modifications.

Let $\B$ be a unital associative algebra over a ring $R$, where $R \cong \Z e_1 \times \cdots \times \Z e_k$ as above. We will not assume $\B$ is graded; however, we will assume $\B$ comes equipped with an \emph{augmentation}, i.e. an algebra homomorphism from $\B$ to the coefficient ring $R$. The algebras of interest to us have a grading of some form, and $R$ is the degree-zero summand. Such an algebra has a natural augmentation given by projection onto this summand.

Suppose $b_1, \ldots, b_m$ is a set of multiplicative generators of $\B$, each in the kernel of the augmentation map. We may assume that for each $b_i$, there is a unique idempotent $e_j$ such that $e_j b_i = b_i$, and $e_j' b_i = 0$ for $j' \neq j$. Indeed, if $e_j b_i = 0$ for all $j$, then $b_i = 0$ and is irrelevant as a generator, and if $e_{j_a} b_i$ were nonzero for multiple indices $a$, we could replace $b_i$ in the list of generators with all of the nonzero $e_{j_a} b_i$. So we may assume $e_j b_i \neq 0$ for exactly one $j$, and then $b_i = 1 b_i = (\sum_{j'} e_{j'}) b_i = e_j b_i$. The idempotent $e_j$ will be called the left idempotent of $b_i$ and denoted $e_L(b_i)$.

Similarly, we may further assume that for each $b_i$, there exists a unique right idempotent $e_R(b_i)$ such that $b_i e_R(b_i) = b_i$ and $b_i e_j = 0$ for $e_j \neq e_R(b_i)$.

Let $V$ be the free $\Z$-module spanned by $\{b_1,\ldots,b_m\}$. The assumptions above equip $V$ with left and right module structures over $R$. The statement that the $b_i$ generate $\B$ multiplicatively means that $\B$ is isomorphic to $T(V) / J$, where 
\[
T(V) = \oplus_{n \geq 0} T^n(V) = R \oplus V \oplus (V \otimes_R V) \oplus (V \otimes_R V \otimes_R V) \oplus \cdots,
\] 
and $J$ is the kernel of the natural map $T(V) \to \B$ sending a string of generators to their product in $\B$. We will assume that the augmentation map, applied to any element of $J$, gives zero. As above, we may assume that each generator of the ideal $J$ has unique left and right idempotents.

\begin{definition} The augmented algebra $\B$, with its choice of generators, is a quadratic algebra if the ideal of relations $J \subset T(V)$ is generated multiplicatively by its intersection with $T^2(V) = V \otimes_R V$. In other words,
\[
J = T(V) \cdot I \cdot T(V),
\]
where $I = J \cap (V \otimes_R V)$. Note that $J$ always contains the right side of this equality, so $\B$ is a quadratic algebra if $J \subset T(V) \cdot I \cdot T(V)$.
\end{definition}

\begin{remark} If $\B$ is a quadratic algebra, then $\B$ obtains a grading by word-length in the generators $\{b_1,\ldots,b_k\}$.
\end{remark}

\begin{remark}\label{lexorder}
Let $\B$ be a quadratic algebra. At various points it will be helpful to work with the generators and relations of $\B$ more explicitly. Following Chapter 4.1 of \cite{PP}, choose an ordering of the multiplicative generators $\{b_1,\ldots,b_k\}$; we may assume that $b_i < b_j$ when $i < j$. Use this order to put a lexicographic ordering on monomials in these generators: the leftmost factor in a product is defined to be the most significant part.

Let $Q$ denote the set of quadratic monomials in the $b_i$. Then $Q$ can be naturally partitioned into two subsets $Q_1$ and $Q_2$: $Q_1$ consists of the monomials which cannot be written as sums of lesser monomials with respect to the lexicographic order, and $Q_2$ consists of the monomials which can. If $b_i b_j \in Q_2$ then 
\[
b_i b_j = \sum_{(i',j') < (i,j)} c_{i,j;i',j'} b_{i'} b_{j'},
\]
and the coefficients $c_{i,j;i',j'}$ are uniquely determined if we require that $c_{i,j;i',j'} = 0$ for $b_{i'} b_{j'}$ in $Q_2$. By Lemma 1.1 in Chapter 4 of \cite{PP}, a set of generators for the quadratic relation ideal $I = J \cap T^2(V)$ of $\B$ is obtained by taking 
\[
I_{i,j} = b_i b_j - \sum_{(i',j') < (i,j)} c_{i,j;i',j'} b_{i'} b_{j'}.
\]
for all $(i,j)$ such that $b_i b_j$ is in $Q_2$.
\end{remark}

\begin{definition} The augmented algebra $\B$, with its choice of generators, is a linear-quadratic algebra if the ideal of relations $J \subset T(V)$ is generated multiplicatively by its intersection with $T^1(V) \oplus T^2(V)$. In other words, writing $J_2 := J \cap (V \oplus (V \otimes_R V))$, $\B$ is linear-quadratic if
\[
J = T(V) \cdot J_2 \cdot T(V),
\]
or equivalently
\[
J \subset T(V) \cdot J_2 \cdot T(V).
\]
We will furthermore assume that $J \cap V = 0$, so that there are no linear redundancies among the chosen generators.
\end{definition}

\begin{remark}\label{FiltrationRemark}
If $\B$ is a linear-quadratic algebra, we get a word-length filtration on $\B$ rather than a grading. An element of $\B$ has filtration level $\leq k$ if it is a sum of products of word-length $\leq k$ in the generators $b_i$. 
\end{remark}

\begin{definition}\label{BzeroDef}
Let $\B$ be a linear-quadratic algebra, so that $\B \cong T(V) / J$ with $J \subset T(V) \cdot J_2 \cdot T(V)$. The quadratic algebra $\B^{(0)}$ is defined as
\[
\B^{(0)} = T(V) / (T(V) \cdot I \cdot T(V)),
\]
where $I \subset T^2(V)$ is defined as the image of $J_2 \subset (V \oplus T^2(V))$ under the projection $(V \oplus T^2(V)) \to T^2(V)$ onto the second summand.
\end{definition}
Every generator $r$ of $I$ is the image of some generator $v \oplus r$ of $J_2$, where $v \in V$. Furthermore, if $v \oplus r$ and $v' \oplus r$ were both in $J_2$ with $v \neq v'$, then $(v - v') \oplus 0$ would be a nonzero element of $J_2 \cap V$, contradicting the assumption that $J_2 \cap V = 0$. Thus, the following definition makes sense:
\begin{definition}\label{VarphiDef}
The function $\varphi: I \to V$ is defined by sending a generator $r \in I$ to the unique element $\varphi(r)$ of $V$ such that $\varphi(r) \oplus r$ is in $J_2$. 
\end{definition}

\begin{proposition}
The map $\varphi$ respects the left and right $R$-actions on $I$ and $V$.
\end{proposition}

\begin{proof}
Suppose $e$ is the left idempotent of $r$ (which exists without loss of generality). Then $e(\varphi(r) \oplus r)$ is in $J_2$, and $e(\varphi(r) \oplus r) = e\varphi(r) \oplus er = e\varphi(r) \oplus r$, so by the uniqueness above, $\varphi(r) = e \varphi(r)$. If $e'$ is any idempotent not equal to $e$, then $e'(\varphi(r) \oplus r)$ is still in $J_2$, but now this expression equals $e'\varphi(r) \oplus 0$. Since $J_2 \cap V = 0$, we must have $e' \varphi(r) = 0$ for $e' \neq e$. Thus, $\varphi$ respects the left $R$-action on $I$ and $V$. The right action is analogous.
\end{proof}

Let $\varphi^{12}$ denote $\varphi \otimes \id_V: I \otimes_R V \to V \otimes_R V$, and let $\varphi^{23}$ denote $\id_V \otimes \varphi: V \otimes_R I \to V \otimes_R V$. 
\begin{proposition}[Chapter 5, Proposition 1.1 of \cite{PP}]
The map $\varphi^{12} - \varphi^{23}: (V \otimes_R I) \cap (I \otimes_R V) \to (V \otimes_R V)$ has image contained in $I$, and
\[
\varphi \circ (\varphi^{12} - \varphi^{23}) = 0.
\]
\end{proposition}

\begin{proof}
The definition of $\varphi$ implies that the image of the map $\varphi \oplus \iota: I \to (V \oplus (V \otimes_R V))$ is contained in $J_2$, where $\iota$ denotes the inclusion map of $I$ into $V \otimes_R V$. Thus, the map
\[
(\varphi \oplus \iota) \otimes \id_V: I \otimes_R V \to (V \oplus (V \otimes_R V)) \otimes_R V = (V \otimes_R V) \oplus (V \otimes_R V \otimes_R V)
\]
has image contained in $J$. On the other hand, $(\varphi \oplus \iota) \otimes \id_V$ is equal to the map
\[
\varphi^{12} \oplus (\iota \otimes \id_V): I \otimes_R V \to (V \otimes_R V) \oplus (V \otimes_R V \otimes_R V).
\]
Thus, $\varphi^{12} \oplus (\iota \otimes \id_V)$ has image contained in $J$. By the same reasoning, $\varphi^{23} \oplus (\iota \otimes \id_V)$ has image contained in $J$ as well.

If $x$ is an element of $(V \otimes_R I) \cap (I \otimes_R V) \subset V^{\otimes 3}$, then we can apply $\varphi^{12} \oplus (\iota \otimes \id_V)$ and $\varphi^{23} \oplus (\iota \otimes \id_V)$ to $x$, producing two elements $\varphi^{12}(x) \oplus x$ and $\varphi^{23}(x) \oplus x$ of $J$. Subtracting, the $x$ terms cancel and $\varphi^{12}(x) - \varphi^{23}(x)$ is also in $J$.

Since $\varphi^{12}(x) - \varphi^{23}(x)$ is an element of both $J$ and $V \otimes_R V$, it is also an element of $J_2 = J \cap (V \oplus (V \otimes_R V))$. The corresponding element of $I$ is the same element $\varphi^{12}(x) - \varphi^{23}(x)$.

Hence we can conclude that $\varphi^{12} - \varphi^{23}$ has image contained in $I$, so it makes sense to postcompose this map with $\varphi$. Furthermore, the image of $\varphi^{12} - \varphi^{23}$ is contained not just in $I$ but in $J_2$, and so $\varphi(\varphi^{12} - \varphi^{23}) = 0$.
\end{proof}

\subsection{Khovanov's arc algebra as a linear-quadratic algebra}\label{HnAsLinQuadratSection}
In this section we present a proof of Lemma~\ref{NonCrossPartLemma}, found by P{\'a}lv{\"o}lgyi~\cite{Palvolgyi} and independently by Potechin~\cite{Potechin}. Besides being important for the constructions in Section~\ref{RobertsFromKhovanovSection} and Section~\ref{ModuleSection}, it will yield an explicit generators-and-relations description of $H^n$ in Corollary~\ref{HnDescriptionCorr}. This description is not necessary, strictly speaking, for Section~\ref{RobertsFromKhovanovSection} and Section~\ref{ModuleSection}, but it may be of interest independently. 

Let $V$ be the free $\Z$-module spanned by the generators $h_{\gamma}$ and $h_{\alpha}$ of $\beta_{mult}$ as defined in Section~\ref{KhovTypeDSect}; the idempotent ring $R = \I_n$ of $H^n$ has both left and right actions on $V$. We may write $H^n$ as $T(V) / J$ for some ideal $J$ of $T(V)$.
\begin{proposition}\label{HnIsLinQuad} With the generators $\{h_{\gamma}, h_{\alpha}\}$ and the augmentation coming from its grading, $H^n$ is a linear-quadratic algebra.
\end{proposition}
This proposition will be proved using Lemma~\ref{NonCrossPartLemma}. We begin with some background. 

Recall from Remark~\ref{CLMatchNCPart} that the elementary idempotents of $H^n$ are in bijection with the set $NC_n$ of noncrossing partitions of $n$ points. In fact, $NC_n$ has a natural partial ordering: suppose $p$ and $q$ are elements of $NC_n$. Then $p \leq q$ if $p$ is a refinement of $q$. As a poset, $NC_n$ is a lattice: any two noncrossing partitions have a unique least upper bound and a unique greatest lower bound, although we will not make use of this property.

The dual of a partially ordered set is defined by reversing the order relations. It is a standard fact that the poset $NC_n$ is self-dual:
\begin{proposition}\label{SelfDualityProp} $NC_n$ is order-isomorphic to the poset obtained by reversing all the order relations on $NC_n$.
\end{proposition}
\begin{proof}
We want to define a bijection $\phi: NC_n \to NC_n$ such that $p < q$ if and only if $\phi(p) > \phi(q)$. Let $p$ be a noncrossing partition. Pick an embedding of acyclic graphs in the half-plane representing $p$; as in Remark~\ref{CLMatchNCPart}, thicken these graphs to get planar surfaces embedded in the half-plane. Color the interiors of these surfaces black. Then the half-plane is divided into black and white regions; interchange these. We now have a checkerboard coloring of the half-plane in which the unbounded region is colored black. Identify the half-plane with the disk, rotate the $n$ points and the regions one step, and identify the disk back with the half plane so that the unbounded region is now colored white. The skeleton of the new black region represents the noncrossing partition $\phi(p)$. One may verify that $\phi$, defined in this way, reverses the order relations.
\end{proof}

Associated to the partial order on $NC_n$ is a Hasse diagram $G_n$, which is a directed graph whose vertices are the elements of $NC_n$ and which has an edge from $p$ to $q$ precisely when $p < q$ and there exists no vertex $q'$ with $p < q' < q$.

We will view $G_n$ as an undirected graph, ignoring the orientations on edges. For any two vertices $p,q$ of $G_n$ connected by an edge, there is a generator $h_{\gamma}$ of $H^n$ with left idempotent $p$ and right idempotent $q$. The generator $h_{\gamma^{\dagger}}$ has right idempotent $q$ and left idempotent $p$. All the $h_{\gamma}$ generators of $H^n$ are of this form.

Monomials in the generators $h_{\gamma}$ either correspond to paths in $G_n$, or are zero for idempotent reasons. We will be especially concerned with paths of minimum length:

\begin{definition}\label{GpqDef} Let $p$ and $q$ be vertices of $G_n$. The graph $G_{p,q}$ has one vertex for each minimal-length path, or geodesic, $\alpha$ from $p$ to $q$ in $G_n$. If $\alpha$ and $\beta$ are two vertices of $G_{p,q}$, they are connected by an edge when $\alpha$ and $\beta$ differ in exactly one vertex of $G_n$ (viewing paths in $G_n$ as sequences of vertices of $G_n$).
\end{definition}

The proof of the following lemma was found by D{\"o}m{\"o}t{\"o}r P{\'a}lv{\"o}lgyi and posted as an answer to a question on MathOverflow \cite{Palvolgyi}; independently, another proof was found by Aaron Potechin \cite{Potechin} and shared with the author privately.

\begin{lemma}[\cite{Palvolgyi}, \cite{Potechin}]\label{NonCrossPartLemma} Let $G_n$ denote the Hasse diagram of $NC_n$, viewed as an undirected graph. Let $p,q$ be vertices of $G_n$, and define $G_{p,q}$ as in Definition~\ref{GpqDef}. Then $G_{p,q}$ is a connected graph.

\end{lemma}

\begin{proof}
First, note that as partitions of a set of $n$ points, either $q$ contains a singleton part or the dual $\phi(q)$ of $q$ contains a singleton part. In the latter case, we can use Proposition~\ref{SelfDualityProp} to reduce to the former case: if $G_{\phi(p),\phi(q)}$ is connected, then so is $G_{p,q}$. Thus, we may assume without loss of generality that $q$ contains a singleton part, say $\{m\}$ where $m$ is one of the $n$ points on the line.

We will induct on both $n$ and the distance between $p$ and $q$; if this distance is $2$ or less, or if $n \leq 2$, there is nothing to prove.

Consider two minimal-length paths $\alpha$ and $\beta$ from $p$ to $q$ in $G_n$. Since $p$ is a partition of the $n$ points on the line, the point $m$ must be contained in one of the partitioning subsets which make up $p$, say $S$. If $S$ contains only $m$, then for each vertex along either $\alpha$ or $\beta$, the point $m$ must be in a singleton set; otherwise $\alpha$ or $\beta$ would not have minimal length (the length could be reduced by removing the steps that connect and disconnect $m$ from the other points on the line). Hence we may ignore $m$ and view $\alpha$ and $\beta$ as paths in $G_{n-1}$. By induction, $\alpha$ may be modified one vertex at a time to produce $\beta$, and we may reintroduce the singleton point $m$ without issue.

On the other hand, suppose $S$ contains additional points as well as $m$. Then we may modify both $\alpha$ and $\beta$, one vertex at a time, to get paths $\alpha'$ and $\beta'$ from $p$ to $q$ such that the first step of both $\alpha'$ and $\beta'$ separates $m$ from the other points in $S$. To do this, find the first step along $\alpha$ or $\beta$ after which $m$ is an isolated point, and commute this step to the beginning of $\alpha$ or $\beta$ by changing the path one vertex at a time.

Now let $p'$ denote the partition $p$ with the point $m$ isolated from $S$. Both $\alpha'$ and $\beta'$ start by moving from $p$ to $p'$ and then along a minimal-length path (say $\alpha''$ or $\beta''$) from $p'$ to $q$. Since the distance from $p'$ to $q$ is one less than the distance from $p$ to $q$, we may conclude by induction that $\alpha''$ may be modified one vertex at a time to obtain $\beta''$. Thus, the same is true for $\alpha'$ and $\beta'$, and hence for $\alpha$ and $\beta$ as well.
\end{proof}

\begin{proof}[Proof of Proposition~\ref{HnIsLinQuad}]
We want to show that $J \subset T(V) \cdot J_2 \cdot T(V)$. We start by analyzing the intersection $J_2 = J \cap (V \oplus (V \otimes_R V))$:

\begin{enumerate}

\item\label{TwoBridgeRels}
Whenever $\gamma$ and $\eta$ are two bridges which can be drawn without intersection on the same crossingless matching, the element $h_{\gamma} h_{\eta'} - h_{\eta} h_{\gamma'}$ is in $J_2$, for the natural choices of $\eta'$ and $\gamma'$.

\item\label{BridgeCircleRels}
Whenever $\gamma$ is a bridge and $\alpha$ is an arc such that $h_{\gamma}$ and $h_{\alpha}$ have the same left idempotent,  the element $h_{\gamma} h_{\alpha'} - h_{\alpha} h_{\gamma}$ is in $I \cap (V \oplus (V \otimes_R) V)$, for the natural choice of $\alpha'$. In fact, if one of the endpoints of the bridge $\gamma$ lies on the arc $\alpha$, then there will be two natural choices for $\alpha'$; both give elements of $J_2$.

\item\label{TwoCircleRels}
Whenever $\alpha_1$ and $\alpha_2$ are arcs in the same crossingless matching, so that $h_{\alpha}$ and $h_{\alpha'}$ have the same left idempotent, the element $h_{\alpha_1} h_{\alpha_2} - h_{\alpha_2} h_{\alpha_1}$ is in $J_2$. Furthermore, for every arc $\alpha$, the element $h_{\alpha}^2$ is in $J_2$.

\item\label{LinQuadRels}
Finally, if $\gamma$ is any bridge, such that $h_{\gamma}$ has left idempotent $e_L$ and right idempotent $e_R$, then there exists a dual bridge $\gamma^{\dagger}$ such that $h_{\gamma^{\dagger}}$ has left idempotent $e_R$ and right idempotent $e_L$. Surgery on $\gamma^{\dagger}$ reverses the effect of surgery on $\gamma$. The element $h_{\gamma} h_{\gamma^{\dagger}} - h_{\alpha_1} - h_{\alpha_2}$ is in $J_2$, where $\alpha_1$ and $\alpha_2$ are the arcs containing the endpoints of $\gamma$.

\end{enumerate}

To show that $J \subset T(V) \cdot J_2 \cdot T(V)$, it suffices to show that for a general element $r$ of $J$, one may successively add to $r$ elements of the ideal generated by the relation elements listed above, until one obtains zero.

Let $r$ be an arbitrary element of $J$. We may assume without loss of generality that $r$ has a unique left idempotent and right idempotent. Since $J$ is an ideal of the tensor algebra $T(V)$, $r$ may be written as a linear combination of monomials in the generators $h_{\gamma}$ and $h_{\alpha}$. Let
\[
r = \sum_i n_i (h_{i,1} \cdots h_{i,l_i}),
\]
where $n_i \in \Z$ and each $h_{i,j}$ is one of the generators $h_{\gamma}$ or $h_{\alpha}$.

Consider one of the monomial summands $m_i = h_{i,1} \cdots h_{i,l_i}$ of $r$. After adding elements of $T(V) \cdot J_2 \cdot T(V)$ to this monomial, we may assume that all the $h_{\gamma}$ generators among the $h_{i,j}$ come before (i.e. with lower $j$ than) the $h_{\alpha}$ generators. The necessary relation elements come from item \ref{BridgeCircleRels} above. Let $m_i'$ denote the monomial obtained by modifying $m_i$ in this way.

Write $m_i'$ as $m_{\gamma,i} \cdot m_{\alpha,i}$, where $m_{\gamma,i}$ is a product of $h_{\gamma}$ generators and $m_{\alpha,i}$ is a product of $h_{\alpha}$ generators. Let $e_L$ denote the left idempotent of $m_{\gamma,i}$ and let $e_R$ denote the right idempotent of $m_{\gamma,i}$. These are independent of $i$ since the left and right idempotents of the monomials $m'_i$ are independent of $i$, and the left idempotent of each monomial $m_{\alpha,i}$ equals the right idempotent of $m_{\alpha,i}$. Then $e_L$ and $e_R$ are vertices of $G_n$, the undirected Hasse diagram of $NC_n$, and to the monomial $m_{\gamma,i}$ we may associate a path $p_{m_{\gamma,i}}$ in $G_n$ from $e_L$ to $e_R$.

We claim that we may further modify $m_i'$ until $p(m_{\gamma,i})$ is a minimal-length path between $e_L$ and $e_R$. Indeed, suppose $p(m_{\gamma,i})$ is a path of non-minimal length. Write $m_{\gamma,i} = h_{\gamma_1} \cdots h_{\gamma_k}$. Then there exists a minimal index $2 \leq j \leq k$ such that $h_{\gamma_1} \cdots h_{\gamma_{j-1}}$ corresponds to a path of minimal length in $G_n$ but $h_{\gamma_1} \cdots h_{\gamma_j}$ does not.

Let $e_R(h_{\gamma_j})$ denote the right idempotent of $h_{\gamma_j}$. The distance between $e_L$ and $e_R(h_{\gamma_{j-1}})$ in $G_n$ is $j-1$, but the distance between $e_L$ and $e_R(h_{\gamma_j})$ is $j-2$ rather than $j$. Indeed, this distance must be less than $j$, but it cannot be less than $j-2$, or the distance between $e_L$ and $e_R(h_{\gamma_{j-1}})$ would be less than $j-1$. The distance between $e_L$ and $e_R(h_{\gamma_j})$ also cannot be $j-1$, because each edge in $G_n$ connects two noncrossing partitions whose number of parts differs by one modulo two. Hence this distance must be $j-2$.

Thus, there exists some monomial $h_{\gamma'_1} \cdots h_{\gamma'_{j-2}}$ corresponding to a path in $G_n$ from $e_L$ to $e_R(h_{\gamma_j})$. Appending $h_{\gamma_j^{\dagger}}$ to this monomial, we get $h_{\gamma'_1} \cdots h_{\gamma'_{j-2}} \cdot h_{\gamma_j^{\dagger}}$, which corresponds to a path of length $j-1$ in $G_n$ between $e_L$ and $e_R(h_{\gamma_{j-1}})$. By assumption, the distance between $e_L$ and $e_R(h_{\gamma_{j-1}})$ is $j-1$, so $h_{\gamma'_1} \cdots h_{\gamma'_{j-2}} \cdot h_{\gamma_j^{\dagger}}$ corresponds to a minimal-length path in $G_n$.

We now have two minimal-length paths in $G_n$ between $e_L$ and $e_R(h_{\gamma_{j-1}})$, namely $\alpha = p(h_{\gamma_1} \cdots h_{\gamma_{j-1}})$ and $\beta = p(h_{\gamma'_1} \cdots h_{\gamma'_{j-2}} \cdot h_{\gamma_j^{\dagger}})$. By Lemma~\ref{NonCrossPartLemma}, we may modify $\alpha$ one vertex at a time to obtain $\beta$. Such modifications correspond, on the level of monomials, to adding relation terms obtained from item \ref{TwoBridgeRels} above.

Thus, we may modify $m_{\gamma,i}$, which equals $h_{\gamma_1} \cdots h_{\gamma_j-1} \cdot h_{\gamma_j} \cdots h_{\gamma_k}$, by adding terms in $T(V) \cdot J_2 \cdot T(V)$ to obtain $h_{\gamma'_1} \cdots h_{\gamma'_{j-2}} \cdot h_{\gamma_j^{\dagger}} \cdot h_{\gamma_j} \cdots h_{\gamma_k}$. Inside this monomial is $h_{\gamma_j^{\dagger}} \cdot h_{\gamma_j}$, which may be replaced with a sum of $h_{\alpha}$ terms using the relation terms in item \ref{LinQuadRels} above. As before, these $h_{\alpha}$ terms may be commuted to the right side of $m_i'$.

After this modification, we have strictly reduced the length of $m_{\gamma,i}$ in the factorization of $m_i'$ as $m_{\gamma,i} \cdot m_{\alpha,i}$. If the new $m_{\gamma,i}$ still does not represent a minimal-length path $p(m_{\gamma,i})$ in $G_n$, we can repeat the same procedure, and eventually it will terminate.

At this point, we have shown that we can modify our original $r = \sum_i n_i (m_i)$ by adding terms in $T(V) \cdot J_2 \cdot T(V)$, until each $m_i$ is a monomial factorizable as $m_{\gamma,i} \cdot m_{\alpha,i}$ with $m_{\alpha,i}$ a monomial in the generators $h_{\alpha}$ and $m_{\gamma,i}$ a monomial in the generators $h_{\gamma}$ representing a minimum-length path in $G_n$. The starting and ending points of all these paths are the same, namely the left and right idempotents of $h$. Thus, by Lemma~\ref{NonCrossPartLemma}, we may do further modifications until all of the $m_{\gamma_i}$ are the same monomial $m_{\gamma}$, and we have 
\[
r = m_{\gamma} \sum_i n_i (m_{\alpha,i}) \textrm{ modulo } T(V) \cdot J_2 \cdot T(V).
\]
Let $r'$ denote the right side of the above equality; $r'$ is an element of $T(V)$, and we want to show that $r' = 0$ modulo $T(V) \cdot J_2 \cdot T(V)$.

Here we use the fact that $r \in J$, or in other words that $r = 0$ as an element of $H^n$. The same holds for $r'$, since $T(V) \cdot J_2 \cdot T(V)$ is a subset of $J$. The monomial $m_{\gamma}$ represents an element of $H^n$ of the form $(W(a)b, \textrm{all plus})$, where $a$ is the left idempotent of $r$ and $b$ is the right idempotent. The signs are all plus because $m_{\gamma}$ corresponds to a path of minimal length. The summand ${_a}(H^n)_b$ is a free abelian group with a basis element for every assignment of signs $\sigma$ to the circles of $W(a)b$. Thus, saying that $r' = 0$ in $H^n$ means that the coefficient of $r'$ on each of these basis elements is zero. In other words, for each assignment of signs $\sigma$ to the circles of $W(a)b$, the sum of the terms $n_i m_{\gamma} m_{\alpha,i}$ of $r'$ corresponding to $\sigma$ is zero in $H^n$.

We will show that for a fixed $\sigma$, the terms $n_i m_{\gamma} m_{\alpha,i}$ such that $m_{\gamma} m_{\alpha,i}$ equals $(W(a)b,\sigma)$ in $H^n$ actually sum to zero modulo the relation terms from items \ref{BridgeCircleRels} and \ref{TwoCircleRels} above. There may also be some terms $m_{\gamma} m_{\alpha,i}$ which are already zero in $H^n$ and thus which represent no basis element $(W(a)b,\sigma)$ of $H^n$. We will deal with these terms at the end.

Suppose $m_{\gamma} h_{\alpha} = m_{\gamma} h_{\alpha'} = (W(a)b, \sigma)$ in $H^n$, where $m_{\gamma}$ corresponds to a minimal-length path; here $\alpha$ and $\alpha'$ are arcs in $b$ which lie on the same circle in $W(a)b$, and $\sigma$ assigns $-$ to this circle while assigning $+$ to all other circles of $W(a)b$. Then we may use relations from item \ref{BridgeCircleRels} to write both $m_{\gamma} h_{\alpha}$ and $m_{\gamma} h_{\alpha'}$ as $h_{\tilde{\alpha}} m_{\gamma}$, where $\tilde{\alpha}$ is any arc in the left idempotent $a$ of $m_{\gamma}$ which, in $W(a)b$, lies in the same circle as $\alpha$ and $\alpha'$. This is true by induction on the length of $\gamma$.

Now, for a more general sum of terms $m_{\gamma} m_{\alpha,i}$ all representing $(W(a)b, \sigma)$ in $H^n$, we can use the above modifications to replace each of the monomials $m_{\alpha,i}$ with the same monomial $m_{\alpha}$. We do this by picking, for example, $m_{\alpha} = m_{\alpha,1}$, and then for $i \neq 1$, we move each factor of $m_{\alpha,i}$ to the left and back to the right so that it becomes identical to the factor appearing in $m_{\alpha,1}$. After doing this for all $i$, we use relations from item~\ref{TwoCircleRels} to replace each $m_{\alpha,i}$ with $m_{\alpha}$.

For a fixed $\sigma$, let $N$ be the sum of the $n_i$ such that $m_{\gamma} m_{\alpha,i}$ represents $(W(a)b,\sigma)$ in $H^n$. By the above paragraph, the sum of the terms $n_i m_{\gamma} m_{\alpha,i}$ of $r'$ with $m_{\gamma} m_{\alpha,i}$ representing $(W(a)b,\sigma)$ in $H^n$ is equivalent to $N m_{\gamma} m_{\alpha}$ modulo $T(V) \cdot J_2 \cdot T(V)$. We see that $N m_{\gamma} m_{\alpha} = 0$ in $H^n$. But since $m_{\gamma} m_{\alpha}$ is the basis element of ${_a}(H^n)_b$ corresponding to $\sigma$, we can conclude that $N = 0$. Thus, the sum of the terms $n_i m_{\gamma} m_{\alpha,i}$ of $r'$ under consideration is equal to zero modulo $T(V) \cdot J_2 \cdot T(V)$.

Finally, some of the terms $m_{\gamma} m_{\alpha,i}$ may not represent any $(W(a)b,\sigma)$ in $H^n$; this happens if and only if $m_{\gamma} m_{\alpha,i}$ is zero in $H^n$. In this case, by the above logic, we can use relations from items \ref{BridgeCircleRels} and \ref{TwoCircleRels} to rearrange $m_{\gamma} m_{\alpha,i}$ until it has $h_{\alpha}^2$ somewhere, for some generator $h_{\alpha}$. Thus, these terms $m_{\gamma} m_{\alpha,i}$ are in $T(V) \cdot J_2 \cdot T(V)$ by item~\ref{TwoCircleRels} above.

Starting with $r \in J$ above, we have successively modified $r$ using linear-quadratic relations until we obtained zero. Hence $J \subset T(V) \cdot J_2 \cdot T(V)$, and so $H^n$ is a linear-quadratic algebra.
\end{proof}

\begin{remark} In fact, one can show by analyzing the grading possibilities case-by-case that the linear-quadratic relations listed above in \ref{TwoBridgeRels}-\ref{LinQuadRels} are a full set of generators for $J_2$. 
\end{remark}

Thus, we get a description of $H^n$ in terms of generators and relations:
\begin{corollary}\label{HnDescriptionCorr}
Let $V$ denote the free $\Z$-module spanned by the degree-1 generators $h_{\gamma}$ and the degree-2 generators $h_{\alpha}$ of $H^n$, with left and right actions of $R = \I_n \cong \Z^{C_n}$ on $V$ given by multiplication in $H^n$. Then
\[
H^n \cong T(V) / (T(V) \cdot J_2 \cdot T(V)),
\]
where the tensor products in $T(V)$ are over $R$, and $J_2$ is generated by the explicit relations given above in items \ref{TwoBridgeRels}-\ref{LinQuadRels} of the proof of Proposition~\ref{HnIsLinQuad}.
\end{corollary}

\begin{remark}
All of the generators of $J_2$ listed in items \ref{TwoBridgeRels}-\ref{LinQuadRels} of the proof of Proposition~\ref{HnIsLinQuad} are homogeneous with respect to the intrinsic grading on $H^n$. Thus, Corollary~\ref{HnDescriptionCorr} also gives us a description of $H^n$ as a graded algebra. This grading differs from the word-length filtration which $H^n$ acquires as a linear-quadratic algebra by Remark~\ref{FiltrationRemark}, even on the basic multiplicative generators: $h_{\alpha}$ has intrinsic degree $2$ and word-length $1$, while $h_{\gamma}$ has intrinsic degree and word-length both equal to $1$.
\end{remark}

\subsection{Quadratic and linear-quadratic duality}\label{QLQDualitySection}

Next we discuss a type of quadratic duality for linear-quadratic algebras. The dual of a linear-quadratic algebra $\B$ is, in general, a quadratic algebra $\B^!$ with a differential. Even if $\B$ is finitely generated over $\Z$, following Convention~\ref{FGConvention}, the algebra $\B^!$ might be infinitely generated over $\Z$. Accordingly, Convention~\ref{FGConvention} will not be taken to hold for dual algebras $\B^!$ in general. However, $\B^!$ will still be generated multiplicatively by a finite set of elements.

We review the relevant definitions from Polishchuk-Positselski \cite{PP} (chapters 1 and 5). We start with the case of quadratic algebras, and then discuss the modification needed for linear-quadratic algebras.

Let $V^*$ denote $\Hom_{\Z}(V,\Z)$. Since $V$ is a free $\Z$-module, $V^*$ is free of the same rank as $V$. If $b_i$ is a generator of $V$, let $b_i^*$ denote the corresponding generator of $V^*$. We define left and right actions of $R$ on $V$ by declaring that $b^*_i$ has the same left and right idempotents as $b_i$.

\begin{definition}
Let $\B$ be a quadratic algebra, and write $\B = T(V) / J$ as above, with $I = J \cap T^2(V)$. The quadratic dual $\B^!$ of $\B$ is defined to be
\[
B^! := T(V^*) / (T(V^*) \cdot I^{\perp} \cdot T(V^*)),
\]
where $I^{\perp}$ is the submodule of $T^2(V^*) = V^* \otimes_R V^*$ annihilating $I$ via the natural action of $V^* \otimes_R V^*$ on $V \otimes_R V$.
\end{definition}

\begin{remark}\label{lexorder2}
Let $Q_1$, $Q_2$, and $b_i$ be as defined as in Remark~\ref{lexorder} above. We had a relation $I_{i,j}$ in $I$ for every monomial $b_i b_j$ in $Q_2$. If $b_i b_j$ is in $Q_1$ rather than $Q_2$, consider instead the dual monomial $b_i^* b_j^*$ in $\B^!$. We can define a relation in $I^{\perp}$ by
\[
I^!_{i,j} = b_i^* b_j^* + \sum_{(i',j') > (i,j)} c^!_{i,j;i',j'} b_{i'}^* b_{j'}^*
\]
where $c^!_{i,j;i',j'}$ is only nonzero if $b_{i'} b_{j'}$ is in $Q_2$, in which case $c^!_{i,j;i',j'}$ is defined to be the coefficient $c_{i',j';i,j}$ of the $(i'' = i, j'' = j)$ term in the relation 
\[
I_{i',j'} = b_{i'} b_{j'} - \sum_{(i'',j'') < (i',j')} c_{i',j';i'',j''} b_{i''} b_{j''}.
\]
The ideal $I^{\perp}$ is spanned by the relations $I^!_{i,j}$.
\end{remark}

We now extend quadratic duality to linear-quadratic algebras.
\begin{definition}\label{LinQuadDualDef}
Let $\B$ be a linear-quadratic algebra; recall that Definition~\ref{BzeroDef} associates a quadratic algebra $\B^{(0)}$ to $\B$. The quadratic dual $\B^!$ of $\B$ is defined, as an algebra, to be the usual quadratic dual of $\B^{(0)}$. Since $\B^!$ is a quadratic algebra, it has a grading by word-length. We will interpret this grading as the homological grading for a differential $\mu_1$ on $\B^!$. 

We will first define $\mu_1$ on the basis elements of $V^*$ and extend to $\B^!$ by the Leibniz rule. We will use the map $\varphi: I \to V$ defined in Definition~\ref{VarphiDef}. Dualizing $\varphi$, we get $\varphi^*: V^* \to I^*$, where $I^* := \Hom_{\Z}(I,\Z)$. 

We claim that $I^*$ is isomorphic to the degree-$2$ summand of $\B^!$. To see this, write the degree-$2$ summand of $\B^!$ as $T^2(V^*) / I^{\perp} = \Hom_{\Z}(V \otimes V, \Z) / I^{\perp}$. There is a natural map $\Xi$ from $\Hom_{\Z}(V \otimes V, \Z)$ to $I^*$ given by precomposing with the inclusion from $I$ into $V \otimes V$. The map $\Xi$ is surjective because any functional from $I$ to $\Z$ may be extended to a functional from $V \otimes V$ to $\Z$. Indeed, using the conventions of Remark~\ref{lexorder}, the $\Z$-basis $\{I_{i,j}: b_i b_j \in Q_2 \}$ for $I$ may be extended to a $\Z$-basis $\{I_{i,j}: b_i b_j \in Q_2 \} \cup Q_1$ for $V \otimes V$.

The kernel of $\Xi$, by definition, consists of those functionals on $V \otimes V$ which annihilate $I$. Thus, the kernel is the same as $I^{\perp}$. We can conclude that $\Xi$ is an isomorphism and $I^*$ is the degree-2 summand of $\B^!$.

Now, for a degree-1 element of $\B^!$, i.e. an element $v^* \in V^*$ dual to a basis element $v$ of $V$, define $\mu_1(v^*)$ to be $\varphi^*(v^*)$. This is an element of $I^*$, and thus a degree-2 element of $\B^!$.

We may extend $\mu_1$ to a map from $\B^!$ to $\B^!$, homogeneous of degree $+1$, using the Leibniz rule
\[
\mu_1(xy) = (-1)^{\deg x} \mu_1(x) y + x \mu_1(y).
\]
Note that this Leibniz rule differs from the one used in \cite{PP}, to stay consistent with our earlier sign conventions.

\end{definition}

\begin{remark}\label{LinQuadIntrinsicGrading}
Suppose $\B$ is a linear-quadratic algebra with an intrinsic grading whose grading-induced augmentation map is the given one. Suppose further that all the multiplicative generators $b_i$ of $\B$ and the explicit generators $I_{i,j}$ of $I$ from Remark~\ref{lexorder} are homogeneous with respect to the intrinsic grading, and that the map $\varphi: I \to V$ preserves intrinsic degree. For example, $H^n$ satisfies these properties: the generators $h_{\gamma}$ have intrinsic degree $1$ and the generators $h_{\alpha}$ have intrinsic degree $2$. Each term of each relation in items \ref{TwoBridgeRels} and \ref{LinQuadRels} of the proof of Proposition~\ref{HnIsLinQuad} has intrinsic degree $2$. Those in item~\ref{BridgeCircleRels} have degree $3$, and those in item~\ref{TwoCircleRels} have degree $4$. The map $\varphi$ is only nonzero on relations from item \ref{LinQuadRels}, and it sends elements of degree $2$ to elements of degree $2$.

With these assumptions, $V^*$ has a natural intrinsic grading, namely the negative of the grading on $V$ (so that the pairing of $V^*$ with $V$ is grading-preserving). The generators of $I^{\perp}$ are homogeneous with respect to this grading; this can be seen from Remark~\ref{lexorder2}. Thus, $\B^!$ acquires an intrinsic grading. Since $\varphi: I \to V$ preserves intrinsic grading, so does $\varphi^*: V^* \to I^*$, and hence the differential $\mu_1$ on $\B^!$ preserves intrinsic grading. The intrinsic grading on $\B^!$ is different from the homological grading, which $\mu_1$ increases by one. In summary, $\B^!$ may be viewed as a differential bigraded algebra with a $(0,+1)$ differential.
\end{remark}

\subsection{The dual of Khovanov's arc algebra}\label{DualOfHnSect}

Proposition~\ref{HnIsLinQuad}, Definition~\ref{LinQuadDualDef}, and Remark~\ref{LinQuadIntrinsicGrading} together give us a differential bigraded algebra $(H^n)^!$, which we will call the dual of $H^n$.
\begin{example}
When $n = 1$, $H^n =  H^1$ is the algebra $\Z[x]/x^2$ over the idempotent ring $\I_1 = \Z$. The generator $x$ has intrinsic degree $2$. Thus, the dual $(H^1)^!$ is $\Z[x^*]$, where $x^*$ has bidegree $(-2,1)$. The differential on $(H^n)^!$ is zero, and $(H^n)^!$ is not finitely generated over $\Z$.
\end{example}
In general, $(H^n)^!$ is never finitely generated over $\Z$, since arbitrary powers of any generator $h^*_{\alpha}$ will be nonzero in $(H^n)^!$.

\subsection{Type DD bimodules}\label{TypeDDBimodSect}

We may relate the duality discussed in Section~\ref{QLQDualitySection} with the Type DD bimodules encountered in bordered Heegaard Floer homology (see Lipshitz-Ozsv{\'a}th-Thurston \cite{LOTBimodules}, especially Section 8). First, we give a definition of these bimodules over $\Z$; as in Section~\ref{BorderedAlgebraSection}, we do not cover the most general case possible.

Let $\B$ and $\B'$ be differential bigraded algebras over an idempotent ring $R = \Pi_i (\Z e_i)$. The case $\B' = \B^!$ will be important, so we will not assume that $\B'$ is finitely generated over $\Z$.
\begin{definition}\label{GeneralDDBimodDef}
A Type DD bimodule over $\B$ and $\B'$ is, first of all, a bigraded free $\Z$-module $\widehat{DD}$ with left and right actions of $R$, such that $\DD$ admits a $\Z$-basis consisting of grading-homogeneous elements with unique left and right idempotents among the $e_i$. Furthermore, $\DD$ must be equipped with an $R$-bilinear map
\[
\delta: \DD \to \B \otimes_R \DD \otimes_R (\B')^{op},
\]
of degree $(0,+1)$, such that the Type DD structure relations
\begin{align*}
&(\mu_1 \otimes \left|\id\right| \otimes \left|\id\right|) \circ \delta \\
&+ (\id \otimes \left|\id\right| \otimes \mu_1) \circ \delta \\
&+ (\mu_2 \otimes \id \otimes \mu_2) \circ \sigma \circ (\id \otimes \delta \otimes \id) \circ \delta \\
&= 0
\end{align*}
are satisfied, where $\mu_1$ and $\mu_2$ denote the differential and multiplication on $\B$ or $\B'$ as appropriate, and
\[
\sigma: \B \otimes \B \otimes \DD \otimes (\B')^{op} \otimes (\B')^{op} \to \B \otimes \B \otimes \DD \otimes (\B')^{op} \otimes (\B')^{op}
\]
is a sign flip which sends a generator $b_1 \otimes b_2 \otimes x \otimes (b'_3)^{op} \otimes (b'_4)^{op}$ to $(-1)^{(\deg_h b_2)(\deg_h b'_4)}$ times itself.
\end{definition}

\begin{remark}
The odd-seeming sign conventions reflect the fact that, while we write $\DD$ with $\B$ on the left and $(\B')^{op}$ on the right to make the notation more manageable, we really want to think of both $\B$ and $\B'$ being on the left of $\DD$ when fixing sign conventions.
\end{remark}

Of particular interest here are Type DD bimodules with $\DD = R$ as an $R$-bimodule. We will refer to these as \emph{rank-one} DD bimodules, following the notation of Chapter 8 of \cite{LOTMorphism}. For a rank-one DD bimodule, we have $\B \otimes_R \DD \otimes_R (\B')^{op} \cong \B \otimes_R (\B')^{op}$, so we may rewrite the Type DD structure relations as
\begin{align*}
&(\mu_1 \otimes \left|\id\right|) \circ \delta \\
&+ (\id  \otimes \mu_1) \circ \delta \\
&+ (\mu_2 \otimes \mu_2) \circ \sigma \circ (\id \otimes \delta \otimes \id) \circ \delta \\
&= 0,
\end{align*}
where $\sigma$ is now a map from $\B \otimes \B \otimes (\B')^{op} \otimes (\B')^{op}$ to itself.

When $\B$ is a linear-quadratic algebra with an intrinsic grading as in Remark~\ref{LinQuadIntrinsicGrading}, we can construct an associated rank-one DD bimodule over $\B$ and $\B^!$. Setting $\DD = R$, we define $\delta: R \to \B \otimes_R (\B^!)^{op}$ by
\[
\delta(e) = \sum_i b_i \otimes (b_i^*)^{op},
\]
where $e$ is one of the elementary idempotents and the sum runs over those multiplicative generators $b_i$ of $R$ which have left idempotent $e$. (These idempotent conditions will be implicit in what follows.) Note that $\delta$ has degree $(0,+1)$; it preserves the intrinsic grading, since the grading on $\B^!$ was defined to be the negative of that on $\B$, and it increases the homological grading by $1$, since $b_i$ has zero homological grading while $b_i^*$ has homological degree $1$.

\begin{proposition}\label{BBbangDD}
The map $\delta$, as defined above, satisfies the Type DD structure relations.
\end{proposition}

\begin{proof}
First, let $e \in R$ be one of the elementary idempotents. Applying the term $(\mu_2 \otimes \mu_2) \circ \sigma \circ (\id \otimes \delta \otimes \id) \circ \delta$ to $e$, we get
\[
(\mu_2 \otimes \mu_2) \circ \sigma \circ (\id \otimes \delta \otimes \id) \circ \delta (e) = \sum_{i,j} b_i b_j \otimes (b^*_j)^{op} (b^*_i)^{op},
\]
where the sum runs over all pairs of multiplicative generators $b_i,b_j$ of $\B$ with compatible idempotents, such that the left idempotent of $b_i$ is $e$. Note that $\sigma = \id$ here, because the generators $b_j$ all have homological degree zero.

In the notation of Remark~\ref{lexorder}, we may split the above sum as
\begin{equation}\label{DDTermEqn1}
\sum_{b_i b_j \in Q_1} b_i b_j \otimes (b^*_j)^{op} (b^*_i)^{op} + \sum_{b_i b_j \in Q_2} b_i b_j \otimes (b^*_j)^{op} (b^*_i)^{op}.
\end{equation}

If $b_i b_j$ is in $Q_1$, then in $\B^!$, we may write $b^*_i b^*_j$ as $-\sum_{(i',j') > (i,j)} c_{i',j';i,j} b^*_{i'} b^*_{j'}$. Thus,
\begin{equation}\label{DDTermEqn2}
\sum_{b_i b_j \in Q_1} b_i b_j \otimes (b^*_j)^{op} (b^*_i)^{op} = -\sum_{b_i b_j \in Q_1} b_i b_j \otimes \sum_{b_{i'} b_{j'} \in Q_2, (i',j') > (i,j)} c_{i',j';i,j} (b^*_{j'})^{op} (b^*_{i'})^{op}.
\end{equation}

On the other hand, if $b_i b_j$ is in $Q_2$, then we may write $b_i b_j$ as 
\begin{align*}
b_i b_j &= \bigg( \sum_{b_{i'} b_{j'} \in Q_1, (i',j') < (i,j)} c_{i,j;i',j'} b_{i'} b_{j'} \bigg) \\
&- \varphi \bigg( b_i b_j - \sum_{b_{i'} b_{j'} \in Q_1, (i',j') < (i,j)} c_{i,j;i',j'} b_{i'} b_{j'} \bigg). \\
\end{align*}
The expression $\varphi \bigg( b_i b_j - \sum_{b_{i'} b_{j'} \in Q_1, (i',j') < (i,j)} c_{i,j;i',j'} b_{i'} b_{j'} \bigg)$, or $\varphi(I_{i,j})$, denotes some linear combination of the multiplicative generators $b_k$ of $\B$. Define the coefficients $C_{i,j;k}$ by
\begin{equation}\label{cijkcoeffs}
\varphi \bigg( b_i b_j - \sum_{b_{i'} b_{j'} \in Q_1, (i',j') < (i,j)} c_{i,j;i',j'} b_{i'} b_{j'} \bigg) = \sum_k C_{i,j;k} b_k.
\end{equation}
Thus,
\begin{align*}
\sum_{b_i b_j \in Q_2} b_i b_j &\otimes (b^*_j)^{op} (b^*_i)^{op} = \sum_{b_i b_j \in Q_2} \sum_{b_{i'} b_{j'} \in Q_1} c_{i,j;i',j'} b_{i'} b_{j'} \otimes (b^*_j)^{op} (b^*_i)^{op} \\
&- \sum_{b_i b_j \in Q_2}  \bigg( \sum_k C_{i,j;k} b_k \bigg) \otimes (b^*_j)^{op} (b^*_i)^{op}.\\
\end{align*}

On the right side of this equation, the first term cancels with the first term $\sum_{b_i b_j \in Q_1} b_i b_j \otimes (b^*_j)^{op} (b^*_i)^{op}$ of expression~\ref{DDTermEqn1}, by equation~\ref{DDTermEqn2}. Thus, we see that 
\[
(\mu_2 \otimes \mu_2) \circ \sigma \circ (\id \otimes \delta \otimes \id) \circ \delta (e) = - \sum_{b_i b_j \in Q_2}  \bigg( \sum_k C_{i,j;k} b_k \bigg) \otimes (b^*_j)^{op} (b^*_i)^{op}.
\]

Now we consider the terms $(\mu_1 \otimes \left|\id\right|) \circ \delta(e)$ and $(\id  \otimes \mu_1) \circ \delta(e)$. The first of these is zero, because $\B$ has no differential. The second may be written as
\[
(\id  \otimes \mu_1) \circ \delta(e) =  \sum_k b_k \otimes (\varphi^*(b^*_k))^{op}.
\]

To compute $\varphi^*(b^*_k)$ as an element of $I^*$, i.e. a homomorphism from $I$ to $\Z$, use equation~\ref{cijkcoeffs} above: this homomorphism sends the generator $I_{i,j} = b_i b_j - \sum_{b_{i'} b_{j'} \in Q_1, (i',j') < (i,j)} c_{i,j;i',j'} b_{i'} b_{j'}$ of $I$ to the coefficient $C_{i,j;k} \in \Z$.

We want to view $\varphi^*(b^*_k)$ as an element of $\B^!$ of homological degree $2$. To do this, following Definition~\ref{LinQuadDualDef}, we pick any extension of $\varphi^*(b^*_k)$ to a functional from $V \otimes_R V$ to $\Z$, or in other words an element of $V^* \otimes_R V^*$, and then consider this element modulo the ideal $I^{\perp}$. Since $\{I_{i,j} : b_i b_j \in Q_2 \} \cup Q_1$ is a $\Z$-basis for $V \otimes_R V$, we may extend $\varphi^*(b^*_k)$ to $V \otimes_R V$ by defining it to be zero on any $b_{i'} b_{j'}$ in $Q_1$. 

This extended $\varphi^*(b^*_k)$ sends $b_i b_j \in Q_2$ to $C_{i,j;k}$, since it sends $I_{i,j}$ to $C_{i,j;k}$ and sends every $b_{i'} b_{j'} \in Q_1$ to zero. Thus,
\[
\varphi^*(b^*_k) = \sum_{b_i b_j \in Q_2} C_{i,j;k} b^*_i b^*_j.
\]
We conclude that 
\[
(\id  \otimes \mu_1) \circ \delta(e) =  \sum_k b_k \otimes \sum_{b_i b_j \in Q_2} C_{i,j;k} (b^*_j)^{op} (b^*_i)^{op},
\]
canceling the remaining term of $(\mu_2 \otimes \mu_2) \circ \sigma \circ (\id \otimes \delta \otimes \id) \circ \delta (e)$. This computation verifies that the Type DD structure relations for $\delta$ are satisfied.
\end{proof}

We can also reverse the roles of $\B$ and $\B^!$: define $\delta': R \to \B^! \otimes_R (\B)^{op}$ by
\[
\delta'(e) = \sum_i b_i^* \otimes (b_i)^{op},
\]
where again the sum is over all multiplicative generators $b_i$ with left idempotent $e$.

\begin{proposition}\label{BbangBDD}
The map $\delta'$ satisfies the Type DD structure relations.
\end{proposition}

\begin{proof} The proof is similar enough to the proof of Proposition~\ref{BBbangDD} that we will omit it to save space.
\end{proof}

\begin{definition}\label{KRk1DD} The Type DD bimodules constructed in Proposition~\ref{BBbangDD} and Proposition~\ref{BbangBDD} will be denoted $^{\B}K^{(\B^!)^{op}}$ and $^{\B^!}K^{\B^{op}}$ respectively. 
\end{definition}

\begin{remark}
In Section 8 of \cite{LOTMorphism}, Lipshitz, Ozsv{\'a}th, and Thurston define a notion of Koszul duality in the language of DD bimodules: two algebras $\B$ and $\B'$ are Koszul dual if there exists a rank-one DD bimodule over $\B$ and $\B'$ which is quasi-invertible (and such that the algebra outputs of the DD operation $\delta$ lie in the kernel of the augmentation maps on $\B$ and $\B'$; this technical condition is satisfied for all the bimodules we consider). We will not define the notion of quasi-invertibility precisely here; see \cite{LOTMorphism}, although Lipshitz-Ozsv{\'a}th-Thurston use $\Z/2\Z$ coefficients.
\end{remark}

By Proposition~\ref{BBbangDD}, we get a Type DD bimodule over $H^n$ and $(H^n)^!$; Proposition~\ref{BbangBDD} gives us a Type DD bimodule over $(H^n)^!$ and $H^n$. It would be interesting to know whether these bimodules are quasi-invertible; if they were, then $(H^n)^!$ could be regarded as the Koszul dual of $H^n$ in this generalized sense. 

However, bordered Floer homology has even stronger duality properties: Theorem 13 of \cite{LOTMorphism} asserts that the bordered surface algebra $\mathcal{A}(\mathcal{Z},i)$ is Koszul dual to both $\mathcal{A}(\mathcal{Z},-i)$ and $\mathcal{A}(\mathcal{Z}_*,i)$, where $\mathcal{Z}$ is a pointed matched circle and $\mathcal{Z}_*$ is another pointed matched circle constructed from $\mathcal{Z}$. This situation contrasts with that of $H^n$, where the quadratic dual algebra is infinitely generated and thus much larger than $H^n$ itself. Below, we will see that Roberts' construction is able to avoid this issue.

\section{Khovanov's algebra and Roberts' algebra}\label{RobertsFromKhovanovSection}

In this section we begin to discuss Roberts' bordered theory for Khovanov homology from \cite{RtypeD} and \cite{RtypeA}. Readers unfamiliar with these papers should still be able to follow our constructions; only those parts claiming that our construction agrees with Roberts' will require familiarity with \cite{RtypeD} and \cite{RtypeA}. Roberts' bordered theory uses a differential bigraded algebra which is denoted $\B\Gamma_n$. This algebra is generated by some right-pointing generators $\overrightarrow{e}$ and left-pointing generators $\overleftarrow{e}$, modulo some explicitly given relations. The differential on $\B\Gamma_n$ is zero on all the right-pointing generators $\overrightarrow{e}$.

We start by defining an algebra $\B_R(H^n)$ using the structure of $H^n$, with its additive basis $\beta = \{(W(a)b, \sigma)\}$ and set of multiplicative generators $\beta_{mult} = \{h_{\gamma}, h_{\alpha}\}$. In Proposition~\ref{KhovRobertsAgreeRSAlg}, we show that $\B_R(H^n)$ is isomorphic to the subalgebra $\B_R\Gamma_n$ of $\B\Gamma_n$ generated by the right-pointing elements.

The subalgebra $\B_R\Gamma_n$ is a linear-quadratic algebra. Its quadratic dual, as defined in Section~\ref{QLQDualitySection}, is closely related to the subalgebra $\B_L\Gamma_n$ of $\B\Gamma_n$ generated by the left-pointing elements. More precisely, in Proposition~\ref{KhovRobertsLSQuotient}, we show that $\B_L \Gamma_n$ is a quotient of $m(\B_R(H^n)^!)$, the mirroring of the quadratic dual of $\B_R (H^n) \cong \B_R \Gamma_n$, by a few explicitly given extra relations. 

Finally, in Section~\ref{FullAlgebraSection}, we take a suitably defined product of $\B_R (H_n)$ and $m(\B_R(H^n))^!$, obtaining an algebra whose quotient by the same extra relations as above is $\B\Gamma_n$.

\subsection{Right side of Roberts' algebra}\label{RightSideSection}

As in Section~\ref{KhovTypeDSect}, let $\beta$ denote the $\Z$-basis $\{(W(a)b, \sigma)\}$ of $H^n$. As at the beginning of Section~\ref{HnAsBorderedSection}, let $\I_n$ denote the idempotent ring of $H^n$. The space $\Hom_{\I_n}(H^n,H^n)$ of left $\I_n$-module maps from $H^n$ to itself is a free $\Z$-module. A $\Z$-basis for $\Hom_{\I_n}(H^n,H^n)$ has generators $e(h_1,h_2)$ for each pair $h_1 \in \beta, h_2 \in \beta$ such that $h_1$ and $h_2$ have the same left idempotent. Here, $e(h_1,h_2)$ is the homomorphism that sends $h_1$ to $h_2$ and sends all other basis elements in $\beta$ to zero. 

Note that $\Hom_{\I_n}(H^n,H^n)$ has the structure of a ring, with multiplication given by composition. We will define a grading on $\Hom_{\I_n}(H^n,H^n)$ which differs from the usual one by a factor of $-\frac{1}{2}$:
\begin{definition}\label{WeirdHomSpaceGradingDef}
Let $e(h_1,h_2)$ be a generator of $\Hom_{\I_n}(H^n,H^n)$. The degree of $e(h_1,h_2)$ is defined to be $\frac{1}{2}(\deg h_1 - \deg h_2)$.
\end{definition}

\begin{remark}
This choice of grading has the advantage that it agrees with Roberts' choice, but it can also be justified on its own grounds. The factor of $-1$ comes from the fact that Khovanov, in \cite{KhovFunctor}, replaces the usual $q$-grading by its negative, to make $H^n$ positively rather than negatively graded. We will see below (in the proof of Proposition~\ref{TensorXBoxAgree} in particular) why the factor of $\frac{1}{2}$ is reasonable. Note that while this grading is now a $\frac{1}{2}\Z$-grading rather than a $\Z$-grading, it will always function as an intrinsic grading rather than a homological grading. Thus, it will have no effect on signs, and we are free to use a $\frac{1}{2}\Z$-grading if desired.
\end{remark}

The degree-0 subring of $\Hom_{\I_n}(H^n,H^n)$ is a direct product of copies of $\Z$, one for each basis element $\beta$ of $H^n$. We will denote this idempotent ring by $\I_{\beta}$; we may view $\Hom_{\I_n}(H^n,H^n)$ as an algebra over $\I_{\beta}$. Note that $\I_{\beta}$ is also the idempotent ring of $\B\Gamma_n$, and hence of $\B_R\Gamma_n$ as well.

\begin{definition}
Let $\B_R(H^n)$ denote the smallest subalgebra of $\Hom_{\I_n}(H^n,H^n)$ containing $\I_{\beta}$ and containing every $e(h_1,h_2)$ such that $h_2$ occurs as a nonzero term in the $\beta$-expansion of $h_1 \cdot h$, for some $h$ in the set of multiplicative generators $\beta_{mult}$. 
\end{definition}

The algebra $\B_R(H^n)$ has an intrinsic grading inherited from the grading on $\Hom_{\I_n}(H^n,H^n)$ defined in Definition~\ref{WeirdHomSpaceGradingDef}. As with $\Hom_{\I_n}(H^n,H^n)$, the degree-0 summand of $\B_R(H^n)$ is the idempotent ring $\I_{\beta}$. The multiplicative generators $e(h_1,h_2)$ such that $h_2$ occurs as a nonzero term in the basis expansion of $h_1 \cdot h_{\gamma}$, for some $\gamma$, have degree $-\frac{1}{2}$, since $h_{\gamma}$ has intrinsic degree $1$. Those such that $h_2$ occurs as a nonzero term in the expansion of some $h_1 \cdot h_{\alpha}$ have degree $-1$, since $h_{\alpha}$ has degree $2$.

A natural set of multiplicative generators for $\B_R(H^n)$ is given in its definition, namely the elements $e(h_1,h_2)$ such that $h_2$ occurs as a nonzero term in the $\beta$-expansion of $h_1 \cdot h$, for some $h$ in the set of multiplicative generators $\beta_{mult}$. If $h = h_{\gamma}$, the corresponding element of $\B_R(H^n)$ will be denoted $b_{\gamma;h_1,h_2}$. If $h = h_{\alpha}$, the corresponding element of $\B_R(H^n)$ will be denoted $b_{C;h_1,h_2}$, where if $h_1 = (W(a)b,\sigma)$, then $C$ is the circle in $W(a)b$ containing $\alpha$. Note that for a fixed $h_1$, all arcs $\alpha'$ which lie on the same circle $C$ as $\alpha$ in $W(a)b$ yield the same generator $b_{C;h_1,h_2}$ of $\B_R(H^n)$.

There are no linear relations among the generators $b_{\gamma;h_1,h_2}$ and $b_{C;h_1,h_2}$. The generators are homogeneous with intrinsic degree $-\frac{1}{2}$ or $-1$, so they are in the kernel of the augmentation map on $\B_R(H^n)$ (which is the projection onto the degree-0 summand). The left idempotent of each generator of the form $b_{\gamma;h_1,h_2}$ and $b_{C;h_1,h_2}$ is $h_1$; the right idempotent is $h_2$.

We will show that $\B_R(H^n)$, with this set of generators, is a linear-quadratic algebra. The proof will closely follow that of Proposition~\ref{HnIsLinQuad}.

Let $V$ be the free $\Z$-module spanned by the generators of $\B_R(H^n)$; as discussed above, the idempotent ring $\I_{\beta}$ has left and right actions on $V$. We may write $\B_R(H^n) = T(V) / J$ for some ideal $J$ of $T(V)$. Let $J_2 := J \cap (T^1(V) \oplus T^2(V))$. We identify a set of generators for $J_2$:

\begin{proposition}\label{BrHnRels}
The ideal $J_2$ of linear-quadratic relations of $\B_R(H^n)$ is generated by the following relations:
\begin{enumerate}

\item\label{BrHnRel1}
Suppose $\gamma$ and $\eta$ are two bridges which can be drawn without intersection on the same crossingless matching $b$; let $\eta'$ denote the bridge on $b_{\gamma}$ corresponding to $\eta$, where $b_{\gamma}$ is the crossingless matching resulting from surgery on $\gamma$. Define $\gamma'$ similarly. For any choice of $(a,\sigma)$, let $h_1 = (W(a)b,\sigma) \in \beta$. If the generators $b_{\gamma;h_1,h_2}$ and $b_{\eta';h_2,h_3}$ exist in $\B_R(H^n)$ for some $h_2,h_3$ in $\beta$, we get a relation
\[
b_{\gamma;h_1,h_2} b_{\eta'; h_2,h_3} - b_{\eta;h_1,h'_2} b_{\gamma'; h'_2,h_3}
\]
in $J_2$, where $h'_2$ is any appropriately chosen element of $\beta$.

\item\label{BrHnRel2}
Suppose $\gamma$ is a bridge on a crossingless matching $b$. For any choice of $(a,\sigma)$, let $h_1 = (W(a)b,\sigma) \in \beta$. Let $C$ be any circle in $W(a)b$. Let $C'$ be any circle in $W(a)b_{\gamma}$ which corresponds to $C$ under surgery on $\gamma$; if the endpoints of $\gamma$ do not both lie on $C$, then $C'$ is unique, and otherwise there are two choices for $C'$. If the generators $b_{\gamma;h_1,h_2}$ and $b_{C';h_2,h_3}$ exist in $\B_R(H^n)$ for some $h_2,h_3$ in $\beta$, we get a relation
\[
b_{\gamma;h_1,h_2} b_{C'; h_2,h_3} - b_{C;h_1,h'_2} b_{\gamma; h'_2,h_3}
\]
in $J_2$, where $h'_2$ is the appropriately chosen element of $\beta$; note that $h'_2$ is uniquely determined by $C$ and $h_1$.

\item\label{BrHnRel3}
For any choice of $(a,b,\sigma)$, let $h_1 = (W(a)b,\sigma) \in \beta$. Let $C_1$ and $C_2$ be two circles in $W(a)b$. If the generators $b_{C_1;h_1,h_2}$ and $b_{C_2,h_2,h_3}$ exist in $\B_R(H^n)$ for some $h_2, h_3$ in $\beta$, we get a relation
\[
b_{C_1;h_1,h_2} b_{C_2;h_2,h_3} - b_{C_2;h_1;h'_2} b_{C_1;h'_2,h_3}
\]
in $J_2$, where $h'_2$ is the appropriately chosen element of $\beta$. As above, $h'_2$ is uniquely determined by $C_2$ and $h_1$.

\item\label{BrHnRel4}
Finally, suppose $\gamma$ is any bridge on a crossingless matching $b$. For any choice of $(a,\sigma)$, let $h_1 = (W(a)b,\sigma) \in \beta$. If the generators $b_{\gamma;h_1,h_2}$ and $b_{\gamma^{\dagger};h_2,h_3}$ exist in $\B_R(H^n)$ for some $h_2,h_3$ in $\beta$, then $h_3$ differs from $h_1$ by switching the sign of one circle of $W(a)b$ from plus to minus. Let $C$ denote this circle. We get a relation
\[
b_{\gamma;h_1,h_2} b_{\gamma^{\dagger};h_2,h_3} - b_{C;h_1,h_3}
\]
in $J_2$.

\end{enumerate}

\end{proposition}

\begin{proof}
Since $\B_R(H^n)$ is an intrinsically graded algebra, if we have a relation in $J_2$, then each of its grading-homogeneous parts must also be in $J_2$. Thus, we may analyze $J_2$ one degree at a time. Since the generators of $\B_R(H^n)$ have intrinsic degree $-\frac{1}{2}$ or $-1$, and we are trying to identify the linear-quadratic relations among them, we may assume these relations have intrinsic degree $-1$, $-\frac{3}{2}$, or $-2$. The case of intrinsic degree $-\frac{1}{2}$ is excluded since any such relation would be a linear dependency among the generators of $\B_R(H^n)$.

The relations of intrinsic degree $-1$ may be sums of quadratic monomials in the degree $-\frac{1}{2}$ generators $b_{\gamma;h_1,h_2}$ of $\B_R(H^n)$ and linear monomials in the degree $-1$ generators $b_{C;h_1,h_2}$. Analyzing the possible cases, we get the relations of items \ref{BrHnRel1} and \ref{BrHnRel4} above.

The relations of intrinsic degree $-\frac{3}{2}$ are sums of quadratic monomials, each involving one degree $-\frac{1}{2}$ generator $b_{\gamma;h_1,h_2}$ and one degree $-1$ generator $b_{C;h_1,h_2}$. These relations are generated by the relations of item \ref{BrHnRel2} above.

Finally, the relations of degree $-2$ are sums of quadratic monomials in the degree $-2$ generators $b_{C;h_1,h_2}$. They are generated by the relations of item \ref{BrHnRel3} above.
\end{proof}

\begin{proposition}\label{BrHnLinQuad}
With $J$ and $J_2$ defined as above, we have
\[
J = T(V) \cdot J_2 \cdot T(V).
\]
Thus, $\B_R(H^n)$ is a linear-quadratic algebra.
\end{proposition}

\begin{proof}

We want to show that $J \subset T(V) \cdot J_2 \cdot T(V)$. As in Proposition~\ref{HnIsLinQuad} above, it suffices to show that for a general element $r$ of $J$, one may successively add to $r$ elements of the ideal generated by the relation elements listed in items \ref{BrHnRel1} - \ref{BrHnRel4} of Proposition~\ref{BrHnRels}, until one obtains zero.

Let $r$ be an arbitrary element of $J$. We may assume without loss of generality that $r$ has a unique left idempotent and right idempotent. Since $J$ is an ideal of the tensor algebra $T(V)$, $r$ may be written as a linear combination of monomials in the generators $b_{\gamma;h_1,h_2}$ and $b_{C;h_1,h_2}$. Let
\[
r = \sum_i n_i (b_{i,1} \cdots b_{i,l_i}),
\]
where $n_i \in \Z$ and each $b_{i,j}$ is one of the generators $b_{\gamma;h_1,h_2}$ or $b_{C;h_1,h_2}$.

Consider one of the monomial summands $m_i = b_{i,1} \cdots b_{i,l_i}$ of $r$. After adding elements of $T(V) \cdot J_2 \cdot T(V)$ to this monomial, we may assume that all the $b_{\gamma;h_1,h_2}$ generators among the $b_{i,j}$ come before (i.e. with lower $j$ than) the $b_{C;h_1,h_2}$ generators. The necessary relation elements come from item \ref{BrHnRel2} of Proposition~\ref{BrHnRels}. Let $m_i'$ denote the monomial obtained by modifying $m_i$ in this way.

Write $m_i'$ as $m_{\gamma,i} \cdot m_{C,i}$, where $m_{\gamma,i}$ is a product of $b_{\gamma;h_1,h_2}$ generators and $m_{C,i}$ is a product of $b_{C;h_1,h_2}$ generators. Let $h \in \beta$ be the left idempotent of $m_{\gamma,i}$, and let $h'_i \in \beta$ be the right idempotent of $m_{\gamma,i}$. Note that $h$ does not depend on $i$, since $h$ is the left idempotent of our original relation term $r$.

Viewing $h$ and $h'_i$ as elements of $H^n$, let $e \in \I_n$ denote the right idempotent of $h$. Let $e' \in \I_n$ denote the right idempotent of $h'_i$, which does not depend on $i$ since the monomial $m_{C,i}$ is a product of $b_C$ generators. As in Proposition~\ref{HnIsLinQuad}, $e$ and $e'$ are vertices of $G_n$, the undirected Hasse diagram of $NC_n$. To the monomial $m_{\gamma,i}$, we can associate a path $p(m_{\gamma,i})$ from $e$ to $e'$ in $G_n$.

We claim that we may further modify $m_i'$ such that $p(m_{\gamma,i})$ is a minimal-length path between $e$ and $e'$ as vertices of $G_n$. Indeed, suppose $m_{\gamma,i}$ corresponds to a path of non-minimal length between $e$ and $e'$. Write $m_{\gamma,i} = b_{\gamma_1;h_1,h_2} \cdots b_{\gamma_k;h_k,h_{k+1}}$. Then there exists a minimal index $2 \leq j \leq k$ such that $b_{\gamma_1;h_1,h_2} \cdots b_{\gamma_{j-1};h_{j-1},h_j}$ corresponds to a path $\phi$ of minimal length in $G_n$ but $b_{\gamma_1;h_1,h_2} \cdots b_{\gamma_j;h_j,h_{j+1}}$ does not.

Let $e_R(h_j) \in \I_n$ denote the right idempotent of $h_j$. Then $e_R(h_j)$ is a vertex of $G_n$, and the distance in $G_n$ between $e$ and $e_R(h_j)$ is $j-1$. However, the distance between $e$ and $e_R(h_{j+1})$ is $j-2$ rather than $j$; the argument is the same as in the proof of Proposition~\ref{HnIsLinQuad}. Thus, there exists a path $\tilde{\psi}$ in $G_n$, of length $j-2$, from $e$ to $e_R(h_{j+1})$. Appending $e_R(h_j)$ to the end of the path $\tilde{\psi}$, we get a path $\psi$ in $G_n$, of length $j-1$, between $e$ and $e_R(h_j)$. By assumption, $\psi$ is a minimal-length path.

We now have two minimal-length paths $\phi$ and $\psi$ between $e$ and $e_R(h_j)$. The path $\phi$ corresponds to the monomial $b_{\gamma_1;h_1,h_2} \cdots b_{\gamma_{j-1};h_{j-1},h_j}$. The path $\psi$ corresponds to a monomial $b_{\gamma'_1;h_1,h'_2} \cdots b_{\gamma'_{j-2};h'_{j-2},h'_{j-1}} \cdot b_{\gamma_j^{\dagger};h'_{j-1},h_j}$, and we have $e_R(h'_{j-1}) = e_R(h_{j+1})$. 

By Lemma~\ref{NonCrossPartLemma}, we may modify $\phi$ one vertex at a time to obtain $\psi$. Such modifications can be mirrored on the level of monomials by adding relation terms obtained from item \ref{BrHnRel1} of Proposition~\ref{BrHnRels}. Thus, we may modify $m_{\gamma,i}$, which equals 
\[
b_{\gamma_1;h_1,h_2} \cdots b_{\gamma_{j-1};h_{j-1},h_j} \cdot b_{\gamma_j;h_j,h_{j+1}} \cdots b_{\gamma_k;h_k,h_{k+1}},
\]
by adding terms in $T(V) \cdot J_2 \cdot T(V)$ to obtain 
\[
b_{\gamma'_1;h_1,h'_2} \cdots b_{\gamma'_{j-2};h'_{j-2},h'_{j-1}} \cdot b_{\gamma_j^{\dagger};h'_{j-1},h_j} \cdot b_{\gamma_j;h_j,h_{j+1}} \cdots b_{\gamma_k;h_k,h_{k+1}}.
\]
Inside this monomial is $b_{\gamma_j^{\dagger};h'_{j-1},h_j} \cdot b_{\gamma_j;h_j,h_{j+1}}$, which may be replaced with a $b_{C;h'_{j-1},h_{j+1}}$ term using the relation terms in item \ref{BrHnRel4} of Proposition~\ref{BrHnRels}. As before, this $b_C$ term may be commuted to the right side of $m_i'$.

After this modification, we have strictly reduced the length of $m_{\gamma,i}$ in the factorization of $m_i'$ as $m_{\gamma,i} \cdot m_{C,i}$. If the new $m_{\gamma,i}$ still does not represent a minimal-length path in $G_n$, we can repeat the same procedure, and eventually it will terminate.

At this point, we have shown that we can modify our original $r = \sum_i n_i(m_i)$ by adding terms in $T(V) \cdot J_2 \cdot T(V)$, until each $m_i$ is a monomial factorizable as $m_{\gamma,i} \cdot m_{C,i}$ with $m_{\gamma,i}$ representing a minimum-length path in $G_n$. The starting and ending vertices of all these paths are the same. Thus, by Lemma~\ref{NonCrossPartLemma}, we may do further modifications until all of the $m_{\gamma,i}$ are the same monomial $m_{\gamma}$, and we have 
\[
r = m_{\gamma} \sum_i n_i(m_{C,i}) \textrm{ modulo } T(V) \cdot J_2 \cdot T(V).
\]
Since $r$ was assumed to have unique left and right idempotents in $\I_{\beta}$, the set of circles $C$ involved in each term $m_{C,i}$ of the above expression must be the same. Thus, using relations from item \ref{BrHnRel3} of Proposition~\ref{BrHnRels}, we may rewrite each $m_{C,i}$ as the same monomial $m_C$. Then
\[
r = N \cdot m_{\gamma} m_C \textrm{ modulo } T(V) \cdot J_2 \cdot T(V),
\]
where $N = \sum n_i$.

Finally, we use the fact that $r \in J$, or in other words that $r = 0$ as an element of $\B_R(H^n)$. This condition implies that $N \cdot m_{\gamma} m_C$ must also be in $J$, since it differs from $r$ by an element of $T(V) \cdot J_2 \cdot T(V)$ which is contained in $J$.

Note that $\B_R(H^n)$ is a subalgebra of $\Hom_{\I_n}(H^n,H^n)$; the element $m_{\gamma} m_C$ may be identified with the left $R$-linear map from $H^n$ to $H^n$ which sends $e$ to $e'$ and sends all other elements of $\beta$ to zero, where $e$ and $e'$ here are the left and right idempotents of $m_{\gamma} m_C$. If $N \cdot m_{\gamma} m_C$ is zero in $\B_R(H^n)$, then it is zero in $\Hom_{\I_n}(H^n,H^n)$, implying that $N$ must be zero. 

 In other words, starting with $r \in J$ above, we have shown that $r = 0$ modulo $T(V) \cdot J_2 \cdot T(V)$. Hence $J \subset T(V) \cdot J_2 \cdot T(V)$, and so $B_R(H^n)$ is a linear-quadratic algebra.
\end{proof}

Now we can see that $\B_R(H^n)$ is isomorphic to $\B_R \Gamma_n$. First, we define the latter algebra more precisely:
\begin{definition}\label{BrGammanDef} Let $\B_R \Gamma_n$ be the subalgebra of $B\Gamma_n$ spanned over $\I_{\beta}$ by those generators $\overrightarrow{e}$ with right pointing arrows ($B\Gamma_n$ also has some generators $\overleftarrow{e}$ with left pointing arrows). The subalgebra $\B_R \Gamma_n$ has no differential. It inherits a bigrading from $B\Gamma_n$; the homological grading is identically zero on $\B_R \Gamma_n$.

\end{definition}

\begin{proposition}\label{KhovRobertsAgreeRSAlg} $\B_R(H^n) \cong \B_R \Gamma_n$ as bigraded algebras over $\I_{\beta}$.
\end{proposition}

\begin{proof} An examination of the subset of Roberts' algebra relations in \cite{RtypeD} which involve only right-pointing generators shows that they correspond with the relations listed in Proposition~\ref{BrHnRels} under the (bigrading-preserving) identification of the generators $b$ of $\B_R(H^n)$ with the generators $\overrightarrow{e}$ of $\B_R \Gamma_n$. Thus, this proposition follows from Proposition~\ref{BrHnLinQuad}.
\end{proof}

\subsection{Left side of Roberts' algebra} 

\begin{definition}\label{BlGammanDef} Let $\B_L \Gamma_n$ be the subalgebra of $\B\Gamma_n$ spanned over $\I_{\beta}$ by those generators $\overleftarrow{e}$ with left pointing arrows. The bigrading and differential on $\B\Gamma_n$ give us a bigrading and differential on $\B_L \Gamma_n$.

\end{definition}

We will see that $\B_L \Gamma_n$ may be identified, after a mirroring operation defined in Definition~\ref{AlgebraMirroringDef}, with a quotient of the quadratic dual $(\B_R (H^n))^!$ of $\B_R (H^n)$ by a few explicitly given extra relations. 

First, we analyze the dual algebra $(\B_R (H^n))^!$. As an algebra, it is the quadratic dual of $(\B_R (H^n))^{(0)}$. We may write $\B_R (H^n)$ as $T(V) / J$, where if $J_2 := J \cap (T^1(V) \oplus T^2(V))$ then we have $J = T(V) \cdot J_2 \cdot T(V)$. Let $I$ denote the image of $J_2$ under the projection map $T^1(V) \oplus T^2(V) \to T^2(V)$ onto the second summand. Then
\[
(\B_R (H^n))^{(0)} \cong T(V) / I,
\]
and
\[
(\B_R (H^n))^{!} \cong T(V^*) / I^{\perp}.
\]

The ideal $J_2$ is generated explicitly by the relations listed in Proposition~\ref{BrHnRels}. We may discard the linear parts of these relations to get a set of generators for $I$. These generators have a simple form: if $r$ is a generating relation in $I$, then $r$ is either a single quadratic monomial or a difference of two quadratic monomials.

Define a graph $G$ whose vertices are all quadratic monomials appearing with nonzero coefficient in some relation $r \in I$. Two monomials $v$ and $v'$ are connected by an edge in $G$ if $v - v'$ is in $I$. Then, analyzing the relations in Proposition~\ref{BrHnRels}, $G$ is a disjoint union of isolated points, line segments (two points connected by an edge and disconnected from the rest of $G$), triangles (three points, all connected, and disconnected from the rest of $G$), and tetrahedra (four points, all connected, and disconnected from the rest of $G$).

Indeed, the isolated points in $G$ are quadratic monomials of the form 
\[
b_{\gamma;h_1,h_2} b_{\gamma^{\dagger},h_2,h_3}.
\]
Some of the line segments come from relation terms 
\[
b_{\gamma;h_1,h_2} b_{\eta';h_2,h_3} - b_{\eta;h_1,h'_2} b_{\gamma';h'_2,h_3},
\]
where $\gamma$ and $\eta$ are two right bridges which can be drawn on the same crossingless matching without intersection, such that $\eta \in B_d(L,\gamma)$ in the notation of Proposition 11 of Roberts~\cite{RtypeD}. Roberts' $L$ corresponds to our $W(a)b$. 

Other line segments come from the same relation terms when $\eta \in B_o(L,\gamma)$, in every case except when $\gamma$ splits a plus-labeled circle and $\eta'$ joins the two newly formed circles into a new minus-labeled circle. For notations like $\B_o(L,\gamma)$ and $\B_d(L,\gamma)$, see Proposition 11 of Roberts \cite{RtypeD}. 

Finally, line segments also come from terms $b_{\gamma;h_1,h_2} b_{C;h_2,h_3} - b_{C;h_1,h'_2} b_{\gamma;h'_2,h_3}$ when the circle $C$ is disjoint from the support of $\gamma$, and from relation terms $b_{C_1;h_1,h_2} b_{C_2;h_2,h_3} - b_{C_2;h_1;h'_2} b_{C_1;h'_2,h_3}$ where $C_1$ and $C_2$ are two distinct circles labeled $+$ in $h_1$.

Some triangles in $G$ connect triples 
\[
\{b_{\gamma;h_1,h_2} b_{C;h_2,h_3}, b_{C_1;h_1,h'_2} b_{\gamma;h'_2,h_3}, b_{C_2;h_1,h''_2} b_{\gamma;h''_2,h_3} \}
\]
and 
\[
\{b_{C;h_1,h_2} b_{\gamma;h_2,h_3}, b_{\gamma;h_1,h'_2} b_{C_1;h'_2,h_3}, b_{\gamma;h_1,h''_2} b_{C_2;h''_2,h_3} \}
\]
when the circle $C$ is not disjoint from the support of $\gamma$. The rest of the triangles connect triples 
\[
\{b_{\gamma_1;h_1,h_2} b_{*;h_2,h_3}, b_{\gamma_2;h_1,h'_2} b_{*;h'_2,h_3}, b_{\gamma_3;h_1,h''_2} b_{*;h''_2,h_3} \},
\]
when $\gamma_i \in B_s(L,\gamma_j)$ for $1 \leq i,j \leq 3$.

Finally, the tetrahedra in $G$ arise when we have two bridges $\gamma$ and $\eta$, with $\eta \in B_o(L,\gamma)$, such that $\gamma$ splits a plus-labeled circle and $\eta'$ joins the newly-formed circles into a minus-labeled circle. In such cases, we have four quadratic monomials which are all equal modulo the relation terms in $I$. These can be written as $b_{\gamma;h_1,h_2} b_{\eta';h_2,h_3}$, $b_{\gamma;h_1,h'_2} b_{\eta';h'_2,h_3}$, $b_{\eta;h_1,h''_2} b_{\gamma';h''_2,h_3}$, and $b_{\eta;h_1,h'''_2} b_{\gamma';h'''_2,h_3}$. 

Order the set of generators $b_{\gamma;h_1,h_2}$ and $b_{C;h_1,h_2}$ of $\B_R(H^n)$ such that the generators $b_{\gamma;h_1,h_2}$ come before the generators $b_{C;h_1,h_2}$ in the ordering. Using $G$, the generators of $I$ may be summarized as follows: for every connected component of $G$, there exists a minimal vertex $v$. For all other vertices $v'$ in the same component of $v$, there exists a relation $v' - v$ in $I$, and if $v$ is a singleton, then $v$ is also a relation in $I$. These relations are a set of generators for $I$ as in Remark~\ref{lexorder}.

We may use the reasoning of Remark~\ref{lexorder2} to identify a set of generators for $I^{\perp}$. For any quadratic monomial in the generators $b$ which does not appear as a vertex of $G$, the corresponding monomial in the generators $b^*$ is an element of $I^*$. Isolated points of $G$ do not give generators of $I^*$. For every line segment in $G$ between vertices $v$ and $v'$, let $v^*$ and $(v')^*$ denote the corresponding monomials in the generators $b^*$. Then $v^* + (v')^*$ is an element of $I^*$. For every triangle in $G$ with a minimal vertex $v$ and two non-minimal vertices $v'$ and $v''$, let $v^*, (v')^*,$ and $(v'')^*$ denote the corresponding monomials in the generators $b^*$. Then
\[
v^* + (v')^* + (v'')^*
\]
is an element of $I^{\perp}$. Finally, for every tetrahedron in $G$ with a minimal vertex $v$ and three non-minimal vertices $v'$, $v''$, and $v'''$, let $v^*, (v')^*,$, $(v'')^*$, and $(v''')^*$ denote the corresponding monomials in the generators $b^*$. Then
\[
v^* + (v')^* + (v'')^* + (v''')^*
\]
is an element of $I^{\perp}$. The above-listed elements generate $I^{\perp}$.

We may also compute the action of the map $\varphi$ on the generators of $I$, using the relations from \ref{BrHnRel4} of Proposition~\ref{BrHnRels}. For every generator of $I$ of the form $b_{\gamma;h_1,h_2} b_{\gamma^{\dagger};h_2,h_3} - b_{C;h_1,h_3}$, we have
\[
\varphi(b_{\gamma;h_1,h_2} b_{\gamma^{\dagger};h_2,h_3}) = - b_{C;h_1,h_3}.
\]
The map $\varphi$ sends all other generators of $I$ to zero. Thus, dualizing $\varphi$, we have
\[
\varphi^*(b^*_{C;h_1,h_3}) = - \sum_i b^*_{\gamma_i;h_1,h_{2,i}} b^*_{\gamma_i^{\dagger};h_{2,i},h_3},
\]
where the sum runs over all bridges $\gamma_i$ on the right crossingless matching of $h_1$ which have an endpoint on $C$, as well as all compatible $h_{2,i}$.

Finally, $(\B_R(H^n))^!$ is bigraded; the generators $b^*_{\gamma;h_1,h_2}$ have degree $(\frac{1}{2},1)$ since $b_{\gamma;h_1,h_2}$ has degree $(-\frac{1}{2},0)$, and the generators $b^*_{C;h_1,h_2}$ have degree $(1,1)$ since $b_{C;h_1,h_2}$ has degree $(-1,0)$. Here, the first index denotes the intrinsic degree, and the second index denotes the homological degree (this is the reverse of Roberts' convention). The generators of $\B_R(H^n)$ are all placed in homological degree $0$.

\begin{remark}\label{MagicFGRemark} While the quadratic dual of an algebra which is finitely generated over $\Z$ (like $\B_R \Gamma_n$) may in general be infinitely generated over $\Z$, the algebra $(\B_R(H^n))^!$ is finitely generated over $\Z$. In fact, the relations on the algebra are irrelevant for this property: $T(V^*)$ is already finitely generated over $\Z$, since the structure of the idempotents only allows monomials of a certain length in the generators of $V^*$ to be nonzero.
\end{remark}

We can now relate $(\B_R(H^n))^!$ with $\B_L \Gamma_n$. To do this, we need to define a mirroring operation for modules and bimodules over the idempotent ring $\I_{\beta}$ of $(\B_R(H^n))^!$ and $\B_L \Gamma_n$:

\begin{definition}\label{IBetaGeneralMirrorDef}
Let $X$ be any module or bimodule over the idempotent ring $\I_{\beta}$. The mirror of $X$, denoted $m(X)$, is the module or bimodule whose actions of $\I_{\beta}$ are the actions on $X$, precomposed with the map from $\I_{\beta}$ to $\I_{\beta}$ which mirrors each elementary idempotent across the line $\{0\} \times \R$. In other words, for a left action of $\I_{\beta}$ on $X$, suppose $x \in X$ and let $m(x)$ denotes the corresponding element of $m(X)$. Let $h = (W(a)b, \sigma) \in \I_{\beta}$; then
\[
h \cdot m(x) := m(h) \cdot x,
\]
where $m(h)$ is $(W(b)a,m(\sigma))$ and $m(\sigma)$ is the same labeling of circles as $\sigma$, mirrored across $\{0\} \times \R$. Right actions of $\I_{\beta}$ on $m(X)$ works similarly. It is immediate from the definitions that $m(m(X)) = X$.
\end{definition}

We will have definitions related to Definition~\ref{IBetaGeneralMirrorDef} for $\I_{\beta}$-modules and bimodules with more structure. Here, we are concerned with algebras:

\begin{definition}\label{AlgebraMirroringDef} Let $\B$ be a differential bigraded algebra over the idempotent ring $\I_{\beta}$. The mirror of $\B$, denoted $m(\B)$, is the same differential bigraded ring as $\B$. As an algebra, the left and right actions of $\I_{\beta}$ are mirrored as in Definition~\ref{IBetaGeneralMirrorDef}. The map from $\B$ to $m(\B)$ sending $b \in \B$ to $m(b) \in m(\B)$ is an isomorphism of rings (but not of algebras); its inverse is the analogously-defined map from $m(\B)$ to $m(m(\B)) = \B$. To avoid confusion with other uses of the letter $m$, we will refer to both of these mirroring maps as $mirr$. 

\end{definition}

\begin{remark} The mirroring operation for algebras commutes with quadratic duality: if $\B$ is a linear-quadratic algebra over $\I_{\beta}$, then
\[
m(\B^!) = (m(\B))^!.
\]
Thus, we can write either of these algebras as $m(\B)^!$. Mirroring also commutes with taking the opposite algebra: we have $m(\B^{op}) = (m(\B))^{op}$, so we can write either of these algebras as $m(\B)^{op}$.
\end{remark}

\begin{proposition}\label{KhovRobertsLSQuotient} 
$\B_L \Gamma_n$ is isomorphic to the quotient of $m(\B_R(H^n))^!$ by the following extra relations. Let the graph $G$ be defined as above; for each tetrahedron in $G$, the only relation in $m(\B_R(H^n))^!$ involving the vertices of the tetrahedron is that the sum of all its vertices is zero. The algebra $\B_L \Gamma_n$ imposes extra relations among the vertices of each tetrahedron. Recall that tetrahedra in $G$ arise when we have two bridges $\gamma$ and $\eta$, with $\eta \in B_o(L,\gamma)$, such that $\gamma$ splits a plus-labeled circle and $\eta'$ joins the newly-formed circles into a minus-labeled circle. The vertices of the corresponding tetrahedron are, following the discussion above:
\begin{itemize}
\item $a := m(b^*_{\gamma;m(h_1),m(h_2)}) m(b^*_{\eta';m(h_2),m(h_3)})$,
\item $b := m(b^*_{\gamma;m(h_1),m(h'_2)}) m(b^*_{\eta';m(h'_2),m(h_3)})$, 
\item $c := m(b^*_{\eta;m(h_1),m(h''_2)}) m(b^*_{\gamma';m(h''_2),m(h_3)})$, 
\item $d := m(b^*_{\eta;m(h_1),m(h'''_2)}) m(b^*_{\gamma';m(h'''_2),m(h_3)})$.
\end{itemize}
Whereas the algebra $m(\B_R(H^n))^!$ imposes only the relation $a + b + c + d = 0$, the algebra $\B_L \Gamma_n$ imposes the relations
\begin{itemize}
\item $a + c = 0$;
\item $a + d = 0$;
\item $b + c = 0$; and
\item $b + d = 0$ (this relation also follows from the previous three).
\end{itemize}
From these relations, $a + b + c + d = 0$ may be deduced, as well as relations for the two remaining edges of the tetrahedron:
\begin{itemize}
\item $a - b = 0$ and
\item $c - d  = 0$.
\end{itemize}
\end{proposition}

\begin{proof} 

Consider the map from $m(\B_R(H^n))^!$ to $\B_L \Gamma_n$ sending  $m(b^*_{\gamma;m(h_1),m(h_2)})$ to $\overleftarrow{e}_{\gamma;h_1,h_2}$, and sending $m(b^*_{C;m(h_1),m(h_2)})$ to $\overleftarrow{e}_{C;h_1,h_2}$. By examining the subset of Roberts' relations from \cite{RtypeD} involving only left-pointing generators, and comparing with the relations for $m(\B_R(H^n))^!$ above, we see that this is a well-defined surjective bigrading-preserving map whose kernel is generated by the extra relations listed in the statement of this proposition. These extra ``anticommutation'' relations can be found in Roberts' algebra as a subset of the relations (20), case (2), on page 17 of \cite{RtypeD}.

After mirroring, the formula above for $\varphi^*$ agrees with Roberts' definition, in Proposition 25 of \cite{RtypeD}, of the differential on $\B \Gamma_n$ (or equivalently on $\B_L \Gamma_n$, since the differential of any right-pointing generator $\overrightarrow{e}$ of $\B \Gamma_n$ is zero). Since both the differential on $m(\B_R(H^n))^!$ and the differential on $\B_L \Gamma_n$ are defined by the same formula on the degree $1$ generators and extended formally to the full algebras by the Leibniz rule, we can conclude that the differential on $\B_L \Gamma_n$ agrees with the differential on $m(\B_R(H^n))^!$ after quotienting the latter algebra by the extra relations.

\end{proof}

\subsection{The full algebra}\label{FullAlgebraSection}
In the following, $\B$ will denote $\B_R(H^n) \cong \B_R \Gamma_n$ unless otherwise specified.

By Proposition~\ref{BBbangDD} and Proposition~\ref{BbangBDD}, we have rank-one Type DD bimodules which we may refer to as ${^{\B}}K^{(\B^!)^{op}}$ and ${^{\B^!}}K^{\B^{op}}$. Like in Definition~\ref{AlgebraMirroringDef} above, we can extend the mirroring operation of Definition~\ref{IBetaGeneralMirrorDef} to these bimodules:

\begin{definition}\label{DDMirroringDef} Let $\B_1$ and $\B_2$ be differential bigraded algebras over the idempotent ring $\I_{\beta}$, and let $(K,\delta)$ be a Type DD bimodule over $\B_1$ and $\B_2$. The mirrored DD bimodule $(m(K),\delta')$ is defined as follows: as an $(\I_{\beta},\I_{\beta})$-bimodule, $m(K)$ is the mirror of $K$ as defined in Definition~\ref{IBetaGeneralMirrorDef}. As in Definition~\ref{AlgebraMirroringDef}, denote the natural map from $K$ to $m(K)$ or $m(K)$ to $K$ by $mirr$. The DD operation on $m(K)$ is 
\[
\delta' = m(K) \xrightarrow{mirr} K \xrightarrow{\delta} \B_1 \otimes K \otimes (\B_2)^{op} \xrightarrow{mirr \otimes mirr \otimes mirr} m(\B_1) \otimes m(K) \otimes m(\B_2)^{op}.
\]
\end{definition}

Applying Definition~\ref{DDMirroringDef} to the bimodules ${^{\B}}K^{(\B^!)^{op}}$ and ${^{\B^!}}K^{\B^{op}}$, we get rank-one DD bimodules which we will denote ${^{m(\B)}}K^{m(\B^!)^{op}}$ and ${^{m(\B)^!}}K^{m(\B)^{op}}$. 

We will focus on the DD bimodules ${^{\B}}K^{(\B^!)^{op}}$ and ${^{m(\B)^!}}K^{m(\B)^{op}}$. Let $\delta_1$ and $\delta_2$ denote the corresponding maps $\delta_1: \I_{\beta} \to \B \otimes_{\I_{\beta}} (\B^!)^{op}$ and $\delta_2: \I_{\beta} \to m(\B)^! \otimes_{\I_{\beta}} m(\B)^{op}$. 

Our goal will be to identify $\B \Gamma_n$ with some type of product algebra $\B \astrosun m(\B)^!$. The set of multiplicative generators of $\B \astrosun m(\B)^!$ should be the union of the generator sets of $\B$ and $m(\B)^!$; there will be inclusion maps from $\B$ and $m(\B)^!$ into $\B \astrosun m(\B)^!$. Similarly, there will be inclusion maps from $m(\B)$ and $\B^! = m(m(\B))^!$ into $m(\B \astrosun m(\B)^!)$, and thus maps from $m(\B)^{op}$ and $(\B^!)^{op}$ into $m(\B \astrosun m(\B)^!)^{op}$.

The algebra $\B \astrosun m(\B)^!$ will be defined such that, when $\delta_1$ and $\delta_2$ are postcomposed with the above inclusion maps, their sum
\[
\delta_1 + \delta_2: \I_{\beta} \to (\B \astrosun m(\B)^!) \otimes_{\I_{\beta}} (m(\B \astrosun m(\B)^!))^{op}
\]
satisfies the Type DD structure relations.

Let $V_{\B}$ (respectively $V_{m(\B)^!})$ denote the free $\Z$-module spanned by the multiplicative generators of $\B$ (respectively $m(\B)^!$). Then $V_{\B}$ and $V_{m(\B)^!}$ have left and right actions of $\I_{\beta}$, and we may write $\B = T(V_{\B}) / J_{\B}$ and $m(\B)^!$ as $T(V_{m(\B)^!}) / J_{m(\B)^!}$. 

Define $V_{full}$, as a bigraded free $\Z$-module, to be $V_{\B} \oplus V_{m(\B)^!}$. The actions of $\I_{\beta}$ on the summands of $V_{full}$ give $V_{full}$ an $\I$-bimodule structure. 

We will define $\B \astrosun m(\B)^!$ to be $T(V_{full}) / J_{full}$, for some ideal $J_{full}$ of $T(V_{full})$. We will define $J_{full}$ with an explicit set of linear-quadratic generators, which will agree with Roberts' relations involving both left-pointing and right-pointing elements of $\B \Gamma_n$.

We can start by analyzing $T^1(V_{full}) \oplus T^2(V_{full})$, which is equal to 
\begin{align*}
&(V_{\B} \oplus V_{m(\B)^!}) \oplus ((V_{\B} \oplus V_{m(\B)^!}) \otimes (V_{\B} \oplus V_{m(\B)^!})) \\
&= V_{\B} \oplus V_{m(\B)^!} \oplus (V_{\B} \otimes V_{\B}) \oplus (V_{\B} \otimes V_{m(\B)^!}) \oplus (V_{m(\B)^!} \otimes V_{\B}) \oplus (V_{m(\B)^!} \otimes V_{m(\B)^!}).
\end{align*}
Thus, $T^1(V_{full}) \oplus T^2(V_{full})$ is the direct sum of $T^1(V_{\B}) \oplus T^2(V_{\B})$, $T^1(V_{m(\B)^!}) \oplus T^2(V_{m(\B)^!})$, and two more summands $V_{\B} \otimes V_{m(\B)^!}$ and $V_{m(\B)^!} \otimes V_{\B}$.

The ideal $J_{full}$ will be generated multiplicatively by $J_{\B} \cap (T^1(V_{\B}) \oplus T^2(V_{\B}))$, $J_{m(\B)^!} \cap (T^1(V_{m(\B)^!}) \oplus T^2(V_{m(\B)^!}))$, and some extra relations $J_{extra} \subset (V_{\B} \otimes V_{m(\B)^!}) \oplus (V_{m(\B)^!} \otimes V_{\B})$.

\begin{definition}\label{JExtraDef}
$J_{extra} \subset (V_{\B} \otimes V_{m(\B)^!}) \oplus (V_{m(\B)^!} \otimes V_{\B})$ is defined additively by the following relations:
\begin{enumerate}

\item\label{JExtraRels1} For two bridge generators $b_{\gamma;h_1,h_2}$ and $m(b^*_{\eta';m(h_2),m(h_3)})$, the commutation relation
\[
b_{\gamma;h_1,h_2} m(b^*_{\eta';m(h_2),m(h_3)}) - m(b^*_{\eta;m(h_1),m(h'_2)}) b_{\gamma';h'_2,h_3}
\]
is in $J_{extra}$ for any valid choice of $h'_2 \in \beta$, following Roberts \cite{RtypeD}. The bridges $\gamma'$ and $\eta$ are uniquely determined.

\item\label{JExtraRels2} Any generator $b_{C;h_1,h_2}$ of $\B = \B_R(H^n)$ for $n > 1$ can be written as a product of bridge generators $b_{\gamma} b_{\gamma^{\dagger}}$. Thus, by the commutation relations above, $b_{C;h_1,h_2}$ must also commute with bridge generators $m(b^*_{\eta;m(h_2),m(h_3)})$: the relation
\[
b_{C;h_1,h_2} m(b^*_{\eta;m(h_2),m(h_3)}) - m(b^*_{\eta;m(h_1),m(h'_2)}) b_{C;h'_2,h_3}
\]
must be in $J_{extra}$ for any valid choice of $h'_2 \in \beta$.

\item\label{JExtraRels3} For a bridge generator $b_{\gamma;h_1,h_2}$ and a decoration generator $m(b^*_{C;m(h_2),m(h_3)})$ in which the circle $C$ are disjoint from the circles involved in surgery on $\gamma$, we also impose commutation relations:
\[
b_{\gamma;h_1,h_2} m(b^*_{C;m(h_2),m(h_3)}) - m(b^*_{C;m(h_1),m(h'_2)}) b_{\gamma;h'_2,h_3}
\]
must be in $J_{extra}$, for the uniquely determined choice of $h'_2 \in \beta$.

\item\label{JExtraRels4} For two disjoint circles $C$ and $C'$, we again have commutation relations:
\[
b_{C;h_1,h_2} m(b^*_{C';m(h_2),m(h_3)}) - m(b^*_{C';m(h_1),m(h'_2)}) b_{C;h'_2,h_3}
\]
must be in $J_{extra}$, for the uniquely determined choice of $h'_2 \in \beta$.

\item\label{JExtraRels5} 
Finally, when $\gamma$ is a bridge and $C$ is one of the circles involved in surgery on $\gamma$:
\begin{align*}
b_{\gamma;h_1,h_2} m(b^*_{C;m(h_2),m(h_3)}) &- m(b^*_{C';m(h_1),m(h'_2)}) b_{\gamma;h'_2,h_3} \\
&- m(b^*_{C'';m(h_1),m(h''_2)}) b_{\gamma;h''_2,h_3}
\end{align*}
is in $J_{extra}$ when $\gamma$ joins $C'$ and $C''$ to form $C$, and 
\begin{align*}
&b_{\gamma;h_1,h'_2} m(b^*_{C';m(h'_2),m(h_3)}) + b_{\gamma;h_1,h''_2} m(b^*_{C'';m(h''_2),m(h_3)}) \\
&- m(b^*_{C;m(h_1),m(h_2)}) b_{\gamma;h_2,h_3}
\end{align*}
is in $J_{extra}$ when $\gamma$ splits $C$ to form $C'$ and $C''$.
\end{enumerate}
\end{definition}

\begin{definition} The ideal $J_{full}$ is defined by
\begin{align*}
J_{full} := T(V_{full}) \cdot &((J_{\B} \cap (T^1(V_{\B}) \oplus T^2(V_{\B}))) \\
&\oplus (J_{m(\B)^!} \cap (T^1(V_{m(\B)^!}) \oplus T^2(V_{m(\B)^!}))) \\
&\oplus J_{extra}) \cdot T(V_{full}).
\end{align*}
The differential bigraded algebra $\B \astrosun m(\B)^!$ is defined by
\[
\B \astrosun m(\B)^! := T(V_{full}) / J_{full},
\]
with a differential induced from the differential on $m(\B)^!$. The differential of any generator of $\B$ is declared to be zero.
\end{definition}

The relations in $J_{extra}$ are modeled on Roberts' relations for $\B \Gamma_n$ in \cite{RtypeD} involving quadratic monomials with one left-pointing and one right-pointing generator.  Thus,

\begin{corollary}\label{FullAlgQuotientCorr} $\B \Gamma_n$, as a differential bigraded algebra, is the quotient of the algebra $\B \astrosun m(\B)^!$ by the same additional relations as specified in Proposition~\ref{KhovRobertsLSQuotient}. These relations involve only quadratic monomials with two generators of $m(\B)^!$.
\end{corollary}

From the definition of the product algebra $\B \astrosun m(\B)^!$, there are natural inclusion maps of $\B$ and $m(\B)^!$ into the product. Also, by Definition~\ref{AlgebraMirroringDef}, we have a mirror algebra $m(\B \astrosun m(\B)^!)$; both $m(\B)$ and $\B^! = m(m(\B)^!)$ have natural inclusion maps into $m(\B \astrosun m(\B)^!)$.

We may view the Type DD map 
\[
\delta_1: \I_{\beta} \to \B \otimes_{\I_{\beta}} (\B^!)^{op} = \B \otimes_{\I_{\beta}} m(m(\B)^!)^{op}
\] 
as a map from $\I_{\beta}$ to $(\B \astrosun m(\B)^!) \otimes_{\I_{\beta}} (m(\B \astrosun m(\B)^!))^{op}$, using the inclusion maps from $\B$ into $\B \astrosun m(\B)^!$ and from $m(m(\B)^!)^{op}$ into $(m(\B \astrosun m(\B)^!))^{op}$. Similarly, we may view
\[
\delta_2: \I_{\beta} \to m(\B)^! \otimes (m(\B))^{op}
\] 
as a map from $\I_{\beta}$ to $(\B \astrosun m(\B)^!) \otimes_{\I_{\beta}} (m(\B \astrosun m(\B)^!))^{op}$ using the inclusion maps from $m(\B)^!$ into $\B \astrosun m(\B)^!$ and from $m(\B)^{op}$ into $(m(\B \astrosun m(\B)^!))^{op}$.

\begin{proposition}\label{FullAlgDDBimodProp} The map $\delta_1 + \delta_2: \I_{\beta} \to (\B \astrosun m(\B)^!) \otimes_{\I_{\beta}} (m(\B \astrosun m(\B)^!))^{op}$ satisfies the Type DD structure relations.
\end{proposition}

\begin{proof} Many of the Type DD structure terms cancel since $\delta_1$ and $\delta_2$ individually satisfy the Type DD relations. In particular, all terms of type $(\mu_1 \otimes \left|\id\right|) \circ \delta$ and $(\id  \otimes \mu_1) \circ \delta$ are accounted for, and we are left with terms of type $(\mu_2 \otimes \mu_2) \circ \sigma \circ (\id \otimes \delta \otimes \id) \circ \delta$. 

Those terms which do not cancel as part of relations for $\delta_1$ or $\delta_2$ are
\begin{equation}\label{FirstTypeDDTerms}
\begin{aligned}
&-b_i \cdot m(b^*_j) \otimes m(b_j)^{op} \cdot (b_i^*)^{op} \\
&= -b_i \cdot m(b^*_j) \otimes m(m(b^*_i) \cdot b_j)^{op},
\end{aligned}
\end{equation}
referred to as terms of type \ref{FirstTypeDDTerms}, as well as 
\begin{equation}\label{SecondTypeDDTerms}
\begin{aligned}
&m(b^*_j) \cdot b_i \otimes (b^*_i)^{op} \cdot m(b_j)^{op} \\
&= m(b^*_j) \cdot b_i \otimes m(b_j \cdot m(b^*_i))^{op},
\end{aligned}
\end{equation}
referred to as terms of type \ref{SecondTypeDDTerms}, where $b_i$ and $b^*_j$ run over all generators of $\B$ and $\B^!$, respectively, with compatible idempotents. Note that the negative signs in the terms of type \ref{FirstTypeDDTerms} come from the sign-flip operator $\sigma$. (The numbering \ref{FirstTypeDDTerms} and \ref{SecondTypeDDTerms} refers to the above equation numbers.)

The commutation relations among the relations defining $\B \astrosun m(\B)^!$ ensure that all the above terms cancel, except for potentially two sets of terms.

The first set $S_1$ of terms includes those terms $-b_i \cdot m(b^*_j) \otimes m(m(b^*_i) \cdot b_j)^{op}$ of Type~\ref{FirstTypeDDTerms} in which both $b_i$ and $b_j$ are among the generators $b_{\gamma}$, and such that the product $b_i \cdot m(b^*_j)$ corresponds to splitting a circle $C$ on the right into two circles $C_1$ and $C_2$ and then joining $C_1$ to $C_2$ again on the left to produce a new circle $C_3$. Similarly, it includes those terms $m(b^*_j) \cdot b_i \otimes m(b_j \cdot m(b^*_i))^{op}$ of Type~\ref{SecondTypeDDTerms} with $b_i$ and $b_j$ among the generators $b_{\gamma}$ such that $m(b^*_j) \cdot b_i$ corresponds to splitting a circle $C$ on the left into two circles $C_1$ and $C_2$ and then joining $C_1$ and $C_2$ again on the right to produce a new circle $C_3$.

The second set $S_2$ consists of those terms of Type \ref{FirstTypeDDTerms} or \ref{SecondTypeDDTerms} in which $b_i$ is one of the generators $b_{\gamma}$ and $b_j$ is one of the generators $b_C$, where $C$ is one of the circles involved in surgery on $\gamma$. 

For all terms except those in $S_1$ and $S_2$, commutation relations may be applied to the Type~\ref{FirstTypeDDTerms} term $-b_i \cdot m(b^*_j) \otimes m(m(b^*_i) \cdot b_j)^{op}$ uniquely to cancel with a unique corresponding Type~\ref{SecondTypeDDTerms} term $m(b^*_{j'}) \cdot b_{i'} \otimes m(b_{j'} \cdot m(b^*_{i'}))^{op}$.

First, we show that the terms in $S_1$ sum to zero. If the Type~\ref{FirstTypeDDTerms} term $-b_i \cdot m(b^*_j) \otimes m(m(b^*_i) \cdot b_j)^{op}$ is in $S_1$, then there are two terms $m(b^*_{j'}) \cdot b_{i'} \otimes m(b_{j'} \cdot m(b^*_{i'}))^{op}$ and $m(b^*_{j''}) \cdot b_{i''} \otimes m(b_{j''} \cdot m(b^*_{i''}))^{op}$ of Type~\ref{SecondTypeDDTerms} in $S_1$ which have the same left and right idempotents as $-b_i \cdot m(b^*_j) \otimes m(m(b^*_i) \cdot b_j)^{op}$. By the commutation relations, both these terms are equal to $b_i \cdot m(b^*_j) \otimes m(m(b^*_i) \cdot b_j)^{op}$.

Furthermore, there is one other Type~\ref{FirstTypeDDTerms} term $-b_{i'''} \cdot m(b^*_{j'''}) \otimes m(m(b^*_{i'''}) \cdot b_{j'''})^{op}$ of $S_1$ which is equal to both $-m(b^*_{j'}) \cdot b_{i'} \otimes m(b_{j'} \cdot m(b^*_{i'}))^{op}$ and $-m(b^*_{j''}) \cdot b_{i''} \otimes m(b_{j''} \cdot m(b^*_{i''}))^{op}$ by the commutation relations. Hence it is equal to $-b_i \cdot m(b^*_j) \otimes m(m(b^*_i) \cdot b_j)^{op}$ as well. These four terms are the only terms in $S_1$ with the same idempotents as $-b_i \cdot m(b^*_j) \otimes m(m(b^*_i) \cdot b_j)^{op}$, and their sum is zero. Thus, the terms in $S_1$ sum to zero.

Now we show that the terms in $S_2$ sum to zero. If $b_i$ is a $b_{\gamma}$ generator and $b_j$ is a $b_C$ generator with $C$ involved in surgery on $\gamma$, then suppose first that $\gamma$ joins two circles $C_1$ and $C_2$ to produce $C$. By item~\ref{JExtraRels5} of Definition~\ref{JExtraDef}, we have $b_i \cdot m(b^*_j) = m(b'^*_j) \cdot b'_i + m(b''^*_j) \cdot b''_i$, where $b'_j$ and $b''_j$ are the generators $b_{C_1}$ and $b_{C_2}$, and $b'_i$ and $b''_i$ are the appropriate $b_{\gamma}$ generators.

Thus, if $-b_i \cdot m(b^*_j) \otimes m(m(b^*_i) \cdot b_j)^{op}$ is the corresponding term of type \ref{FirstTypeDDTerms}, we have
\begin{align*} 
&-b_i \cdot m(b^*_j) \otimes m(m(b^*_i) \cdot b_j)^{op} \\
&= -m(b'^*_j) \cdot b'_i \otimes m(m(b^*_i) \cdot b_j)^{op} - m(b''^*_j) \cdot b''_i \otimes m(m(b^*_i) \cdot b_j)^{op} \\
&= -m(b'^*_j) \cdot b'_i \otimes m(b'_j \cdot m(b'^*_i))^{op} - m(b''^*_j) \cdot b''_i \otimes m(b''_j \cdot m(b''^*_i))^{op}, \\
\end{align*}
where in the last step we use commutation relations from item~\ref{JExtraRels2} of Definition~\ref{JExtraDef}. The two resulting terms cancel the two relevant terms of type \ref{SecondTypeDDTerms}. The case when $\gamma$ splits a circle, rather than joining two circles, is analogous, and so the terms in $S_2$ sum to zero.

Hence all the relation terms cancel, and $(\delta_1 + \delta_2)$ satisfies the Type DD structure relations.
\end{proof}

A general property of Type D structures and Type DD bimodules over an algebra $\B$ is they give induced Type D or DD structures over any quotient of $\B$:
\begin{proposition}\label{DandDDQuotient}
Let $\B$ be a differential bigraded algebra; let $J$ be a bigrading-homogeneous ideal of $\B$ which is preserved by the differential on $\B$. Let $\pi: \B \to \B / J$ denote the quotient projection map. Let $(\D,\delta)$ be a Type D structure over $\B$; then $\D$ descends to a Type $\D$ structure over $\B / J$, with structure operation
\[
\D \xrightarrow{\delta} \B \otimes \D \xrightarrow{\pi \otimes \id} (\B / J) \otimes \D.
\]

Similarly, if $\B'$ and $J'$ are another algebra and ideal satisfying the same conditions as $\B$ and $J$, and $(\DD, \delta)$ is a Type DD bimodule over $\B$ and $\B'$, then $\DD$ descends to a Type DD bimodule over $\B / J$ and $\B' / J'$, with structure operation
\[
\DD \xrightarrow{\delta} \B \otimes \DD \otimes (\B')^{op} \xrightarrow{\pi \otimes \id \otimes (\pi')^{op}} (\B / J) \otimes \DD \otimes (\B' / J')^{op}.
\]
\end{proposition}

\begin{proof} This is a simple consequence of the Type D and Type DD structure relations.
\end{proof}

\begin{corollary} The map $\I_{\beta} \to (\B \Gamma_n) \otimes_{\I_{\beta}} (m(\B \Gamma_n))^{op}$ obtained by postcomposing $\delta_1 + \delta_2$ with the tensor product of the quotient projections onto $\B \Gamma_n$ and $m(\B \Gamma_n)$ satisfies the Type DD structure relations.
\end{corollary}

Thus, we have rank-one Type DD bimodules 
\[
{^{\B \astrosun m(\B)^!}}K^{m(\B \astrosun m(\B)^!)^{op}}
\]
and 
\[
{^{\B \Gamma_n}}K^{m(\B \Gamma_n)^{op}}.
\]

\begin{conjecture}\label{BiggerAlgConjecture} Either or both of the DD bimodules ${^{\B \astrosun m(\B)^!}}K^{m(\B \astrosun m(\B)^!)^{op}}$ and ${^{\B \Gamma_n}}K^{m(\B \Gamma_n)^{op}}$ are quasi-invertible. Hence, either or both of the algebras $\B \astrosun m(\B)^!$ and $\B \Gamma_n$ are Koszul dual to their mirrors, $m(\B \astrosun m(\B)^!)$ and $m(\B \Gamma_n)$, in the generalized sense of \cite{LOTMorphism}.
\end{conjecture}

A proof of the above conjecture would provide a nice parallel between Roberts' theory and bordered Floer homology. In bordered Floer homology, the rank-one DD bimodule corresponding to the identity cobordism of a parametrized surface has a quasi-inverse, namely the Type AA bimodule associated to this cobordism.

\section{Khovanov's modules and Roberts' modules}\label{ModuleSection}

In this section, we relate Roberts' Type D and Type A structures from \cite{RtypeD} and \cite{RtypeA} to the Type D and Type A structures over $H^n$ from Section~\ref{HnAsBorderedSection}, or equivalently to Khovanov's $[T]^{Kh}$ which contains the same information. In Section~\ref{RobertsTypeAStrSect}, we show that given a chain complex of projective graded right $H^n$-modules satisfying a certain algebraic condition, we may construct a differential bigraded right module over $\B \astrosun m(\B)^!$. Applied to Khovanov's tangle complex $[T]^{Kh}$, this module descends to a module over $\B \Gamma_n$; in other words, the relations of Proposition~\ref{KhovRobertsLSQuotient} act as zero on the $\B \astrosun m(\B)^!$-module. The resulting $\B \Gamma_n$-module agrees with Roberts' Type A structure. 

In Section~\ref{RobertsTypeDStrSect}, given a chain complex of projective graded left $H^n$-modules satisfying the same algebraic condition as in Section~\ref{RobertsTypeAStrSect},  we construct a Type D structure over $\B \astrosun m(\B)^!$. We do this by mirroring the chain complex of left $H^n$-modules to get a complex of right modules, taking the associated Type A structure over $\B \astrosun m(\B)^!$, tensoring with the DD bimodule ${^{\B \astrosun m(\B)^!}}K^{m(\B \astrosun m(\B)^!)^{op}}$ from the end of Section~\ref{FullAlgebraSection} to get a Type D structure over $m(\B \astrosun m(\B)^!)$, and finally mirroring this Type D structure again to get a Type D structure over $\B \astrosun m(\B)^!$. We may quotient the algebra outputs of this Type D structure by the relations from Proposition~\ref{KhovRobertsLSQuotient} to get a Type D structure over $\B \Gamma_n$, which agrees with the one constructed by Roberts when one starts with the complex $[T]^{Kh}$.

Given two chain complexes, one of projective graded left $H^n$-modules and one of projective graded right $H^n$-modules, their tensor product over $H^n$ is a chain complex with an additional grading, or equivalently a differential bigraded $\Z$-module. In Section~\ref{RobertsPairingSect}, we show that this tensor product agrees with the box tensor product of the Type D and Type A structures over $\B \astrosun m(\B)^!$ associated to the two complexes in Sections~\ref{RobertsTypeDStrSect} and \ref{RobertsTypeAStrSect}. If the Type A structure is an ordinary right $\B \astrosun m(\B)^!$-module which descends to a $\B \Gamma_n$-module, then the box tensor products over $\B \astrosun m(\B)^!$ and $\B \Gamma_n$ agree. Applying these constructions to Khovanov's chain complexes $[T]^{Kh}$ of $H^n$-modules, we get an alternate proof that the pairing of Roberts' Type D and Type A structures recovers the original Khovanov complex.

In Section~\ref{TypeAEquivSect}, we show that chain homotopy equivalences of complexes of $H^n$-modules give $\mathcal{A}_{\infty}$ homotopy equivalences of the corresponding Type A structures over $\B \astrosun m(\B)^!$. Reidemeister moves on tangle diagrams yield chain homotopy equivalences of complexes of $H^n$-modules, as shown by Khovanov in \cite{KhovFunctor}. The $\mathcal{A}_{\infty}$ homotopy equivalences associated to these Reidemeister-move homotopy equivalences descend to $\mathcal{A}_{\infty}$ homotopy equivalences of Type A structures over $\B \Gamma_n$. This reasoning yields an alternate proof that Roberts' Type A structures are tangle invariants up to $\A_{\infty}$ homotopy equivalence. 

In Section~\ref{TypeDEquivSect}, we do the same for the Type D structures over $\B \astrosun m(\B)^!$. All homotopy equivalences of Type D structures over $\B \astrosun m(\B)^!$ descend to homotopy equivalences of Type D structures over the quotient $\B \Gamma_n$. Thus, we obtain an alternate proof that Roberts' Type D structures are tangle invariants up to homotopy equivalence.

\subsection{Preliminaries}\label{ModulePrelims}
Let $M$ be a differential bigraded projective right $H^n$-module, or equivalently a chain complex of projective graded $H^n$-modules by Proposition~\ref{DgModIsChainCx} or a right Type D structure over $H^n$ by the appropriate analogue of Proposition~\ref{TypeDIsChainCx}. Recall that in accordance with Convention~\ref{FGConvention}, such an $M$ is assumed to be finitely generated over $\Z$. Let $\{x_i: i \in S\}$ be bigrading-homogeneous elements of $M$, where $S$ is some finite index set, such that 
\[
M \cong \oplus_i x_i H^n
\]
as right $H^n$-modules, and each summand $x_i H^n$ is isomorphic to $e H^n$ for some elementary idempotent $e$ of $H^n$ via an isomorphism sending $x_i$ to $e$. The idempotent $e$ associated to $x_i$ will be denoted $e(x_i)$. 

We will use notation from Section~\ref{KhovTypeDSect}: $\beta$ will denote the usual $\Z$-basis of $H^n$, and $\beta_{mult}$ will denote the subset of $\beta$ consisting of the multiplicative generators $h_{\gamma}$ and $h_{\alpha}$ of $H^n$. We will further subdivide $\beta_{mult}$ into $\beta_{\gamma}$, consisting of $h_{\gamma}$ generators, and $\beta_{\alpha}$, consisting of $h_{\alpha}$ generators.

The module $M$ has a $\Z$-basis given by
\[
\{ x_i \cdot h: i \in S, h \in \beta, e_L(h) = e(x_i) \}.
\]
We define integer coefficients $c_{i,j}$ and $\tilde{c}_{i,j;h'}$ for $i,j \in S$ and $h' \in \beta$ by expanding out the differential of each $x_i$:
\begin{align*}
d(x_i) &= \sum_{j \in S, \deg x_j = \deg x_i + (0,1)} c_{i,j} x_j \\
&+ \sum_{j \in S, h' \in \beta, \deg h' \neq 0, \deg x_j + \deg h' = \deg x_i + (0,1)}  \tilde{c}_{i,j;h'} x_j \cdot h'.
\end{align*}

The algebraic condition we impose on $M$ is the following:
\begin{definition}\label{CModuleAlgCondition}
$M$ satisfies the algebraic condition $C_{module}$ if $\tilde{c}_{i,j;h'} = 0$ unless $h' \in \beta_{mult}$. 
\end{definition}
This condition is satisfied for Khovanov's tangle complexes $[T]^{Kh}$. For any $M$ satisfying $C_{module}$, we can write the above sum as
\begin{align*}
d(x_i) &= \sum_{j \in S, \deg x_j = \deg x_i + (0,1)} c_{i,j} x_j \\
&+ \sum_{j \in S, h' \in \beta_{\gamma}, \deg x_j = \deg x_i + (-1,1)}  \tilde{c}_{i,j;h'} x_j \cdot h' \\ 
&+ \sum_{j \in S, h' \in \beta_{\alpha}, \deg x_j = \deg x_i + (-2,1)}  \tilde{c}_{i,j;h'} x_j \cdot h'.
\end{align*}
This is an expansion of $d(x_i)$ in the $\Z$-basis of $M$.

Thus, if $x_i \cdot h$ is a basis element of $M$, we have
\begin{align*}
d(x_i \cdot h) &= \sum_{j \in S, \deg x_j = \deg x_i + (0,1)} c_{i,j} x_j \cdot h \\
&+ \sum_{j \in S, h' \in \beta_{\gamma}, \deg x_j = \deg x_i + (-1,1)}  \tilde{c}_{i,j;h'} x_j \cdot h'h \\ 
&+ \sum_{j \in S, h' \in \beta_{\alpha}, \deg x_j = \deg x_i + (-2,1)}  \tilde{c}_{i,j;h'} x_j \cdot h'h.
\end{align*}
However, this is not necessarily a basis expansion of $d(x_i \cdot h)$, because the elements $h'h \in H^n$ are not necessarily elements of the basis $\beta$. Instead, we may define integer coefficients $\tilde{\tilde{c}}$ by
\[
h'h = \sum_{h'' \in \beta} \tilde{\tilde{c}}_{h'h;h''} h'',
\]
and thus
\begin{align*}
d(x_i \cdot h) &= \sum_{j \in S, \deg x_j = \deg x_i + (0,1)} c_{i,j} x_j \cdot h \\
&+ \sum_{j \in S, h' \in \beta_{\gamma}, h'' \in \beta, \deg x_j = \deg x_i + (-1,1)}  \tilde{c}_{i,j;h'} \tilde{\tilde{c}}_{h'h;h''} x_j \cdot h'' \\ 
&+ \sum_{j \in S, h' \in \beta_{\alpha}, h'' \in \beta, \deg x_j = \deg x_i + (-2,1)}  \tilde{c}_{i,j;h'} \tilde{\tilde{c}}_{h'h;h''} x_j \cdot h''.
\end{align*}
This is a basis expansion of $d(x_i \cdot h)$ in the $\Z$-basis of $M$.

\begin{proposition}\label{HnProjModStructure} The equation $d^2 = 0$ on $M$ gives rise to the following five sets of equations involving the coefficients $c_{i,j}$, $\tilde{c}_{i,j,h'}$, and $\tilde{\tilde{c}}_{h'h;h''}$:
\begin{enumerate}
\item\label{Deg0CEqns} For all $x_i$ and $x_k$ with $\deg x_k = \deg x_i + (0,2)$, we have
\[
\sum_j c_{i,j} c_{j,k} = 0.
\]

\item\label{Deg1CEqns} For all $x_i \cdot h$ and $x_k \cdot h''$ with $\deg x_k = \deg x_i + (-1,2)$, we have
\[
\sum_{j,h' \in \beta_{\gamma}} (\tilde{c}_{i,j;h'} \tilde{\tilde{c}}_{h'h;h''} c_{j,k} + c_{i,j} \tilde{c}_{j,k;h'} \tilde{\tilde{c}}_{h'h;h''}) = 0.
\]

\item\label{Deg2CEqns} For all $x_i \cdot h$ and $x_k \cdot h''$ with $\deg x_k = \deg x_i + (-2,2)$, we have
\begin{align*}
&\sum_{j,h' \in \beta_{\alpha}} (\tilde{c}_{i,j;h'} \tilde{\tilde{c}}_{h'h;h''} c_{j,k} + c_{i,j} \tilde{c}_{j,k;h'} \tilde{\tilde{c}}_{h'h;h''}) \\
&+ \sum_{j, h' \in \beta_{\gamma}, h''' \in \beta_{\gamma}, h'''' \in \beta} \tilde{c}_{i,j;h'} \tilde{\tilde{c}}_{h'h;h''''} \tilde{c}_{j,k;h'''} \tilde{\tilde{c}}_{h'''h'''';h''} \\
&= 0.
\end{align*}

\item\label{Deg3CEqns} For all $x_i \cdot h$ and $x_k \cdot h''$ with $\deg x_k = \deg x_i + (-3,2)$, we have
\begin{align*}
&\sum_{j, h' \in \beta_{\gamma}, h''' \in \beta_{\alpha}, h'''' \in \beta} \tilde{c}_{i,j;h'} \tilde{\tilde{c}}_{h'h;h''''} \tilde{c}_{j,k;h'''} \tilde{\tilde{c}}_{h'''h'''';h''} \\
&+ \sum_{j, h' \in \beta_{\alpha}, h''' \in \beta_{\gamma}, h'''' \in \beta} \tilde{c}_{i,j;h'} \tilde{\tilde{c}}_{h'h;h''''} \tilde{c}_{j,k;h'''} \tilde{\tilde{c}}_{h'''h'''';h''} \\
&= 0.
\end{align*}

\item\label{Deg4CEqns} For all $x_i \cdot h$ and $x_k \cdot h''$ with $\deg x_k = \deg x_i + (-4,2)$, we have
\begin{align*}
\sum_{j, h' \in \beta_{\alpha}, h''' \in \beta_{\alpha}, h'''' \in \beta} \tilde{c}_{i,j;h'} \tilde{\tilde{c}}_{h'h;h''''} \tilde{c}_{j,k;h'''} \tilde{\tilde{c}}_{h'''h'''';h''} = 0.
\end{align*}

\end{enumerate}
\end{proposition}

\begin{proof} This follows from writing out $d^2(x_i \cdot h)$ as a sum of basis elements $x_k \cdot h''$, using the above basis expansion for $d(x_i \cdot h)$, and then grouping the $x_k \cdot h''$ according to the intrinsic degree of $x_k$ relative to $x_i$.
\end{proof}

\begin{example}\label{HnCCoeffsEx}
Suppose $M = [T]^{Kh}$, where $T$ is an oriented tangle diagram in $\R_{\leq 0} \times \R$. We will analyze the generators $x_i$ and coefficients $\tilde{c}_{i,j;h'}$ and $c_{i,j}$. To specify a generator $x_i$ of $[T]^{Kh}$, we first specify a resolution $\rho_i$ of all crossings of $T$; we can view $\rho_i$ as a function from the set of crossings to the two-element set $\{0,1\}$. If $T_{\rho_i}$ denotes the diagram $T$ with the crossings resolved according to $\rho_i$, then $T_{\rho_i}$ consists of a left crossingless matching of $2n$ points together with some free circles contained in $\R_{< 0} \times \R$. The remaining data needed to specify $x_i$ is a choice of $+$ (plus) or $-$ (minus) on each free circle. Then $[T]^{Kh}$ has a $\Z$-basis consisting of elements $x_i \cdot h$, where the left crossingless matching of $h$ agrees with the matching obtained from $T_{\rho_i}$ by discarding the free circles.

Let $S$ denote the set of $x_i$ specified above. The basis expansion defining $c_{i,j}$ and $\tilde{c}_{i,j;h'}$ is
\begin{align*}
d(x_i) &= \sum_{j \in S, \deg x_j = \deg x_i + (0,1)} c_{i,j} x_j \\
&+ \sum_{j \in S, h' \in \beta_{\gamma}, \deg x_j = \deg x_i + (-1,1)}  \tilde{c}_{i,j;h'} x_j \cdot h' \\ 
&+ \sum_{j \in S, h' \in \beta_{\alpha}, \deg x_j = \deg x_i + (-2,1)}  \tilde{c}_{i,j;h'} x_j \cdot h'.
\end{align*}
The coefficients $c_{i,j}$ and $\tilde{c}_{i,j;h'}$ can only be nonzero when the resolution $\rho_j$ of $x_j$ differs from the resolution $\rho_i$ of $x_i$ only at one crossing, to which $\rho_i$ assigns $0$ and $\rho_j$ assigns $1$. Let $\#_1(i,j)$ denote the number of $1$-resolutions of crossings in $x_i$ among those crossings which, in the ordering on crossings, occur earlier than the crossing being changed when going from $x_i$ to $x_j$.

Changing the crossing to get from $x_i$ to $x_j$ has several possible effects:
\begin{enumerate}
\item\label{KhCoeffs1} The crossing change could join two free circles or split a free circle. In this case, $c_{i,j}$ is $(-1)^{\#_1(i,j)}$, and all $\tilde{c}_{i,j;h'}$ are zero.

\item\label{KhCoeffs2} The crossing change could join a free circle in $x_i$, labeled $+$, with an arc of $x_i$. Alternatively, it could split a new free circle, labeled $-$ in $x_j$, off an arc of $x_i$. In both these cases, $c_{i,j}$ is $(-1)^{\#_1(i,j)}$, and all $\tilde{c}_{i,j;h'}$ are zero.

\item\label{KhCoeffs3} The crossing change could join a free circle in $x_i$, labeled $-$, with an arc $\alpha$ of $x_i$. Alternatively, it could split a new free circle, labeled $+$ in $x_j$, off an arc $\alpha$ of $x_i$. In both these cases, $c_{i,j}$ is zero, and $c_{i,j;h'}$ is only nonzero for one value of $h'$. If $a$ denotes the crossingless matching of $x_i$ (or of $x_j$), then when $h' = (W(a)a, \textrm{ minus on } \alpha)$, we have $c_{i,j;h'} = (-1)^{\#_1(i,j)}$; all other $\tilde{c}_{i,j;h'}$ are zero.

\item\label{KhCoeffs4} Finally, the crossing change could surger two arcs of $x_i$, changing the crossingless matching. Again, $c_{i,j} = 0$, and $c_{i,j;h'}$ is nonzero for a unique $h'$. Let $a_i$ and $a_j$ denote the crossingless matchings of $x_i$ and $x_j$ respectively. Then if $h' = (W(a_j)a_i, \textrm{ all }+)$, we have $c_{i,j;h'} = (-1)^{\#_1(i,j)}$; all other $\tilde{c}_{i,j;h'}$ are zero.
\end{enumerate}
Note that $[T]^{Kh}$ satisfies the condition $C_{module}$.
\end{example}

\subsection{Type A structures}\label{RobertsTypeAStrSect}
As in Section~\ref{FullAlgebraSection}, let $\B$ denote $\B_R(H^n) \cong \B_R \Gamma_n$. Recall that $\B_L \Gamma_n$ is the quotient of $m(\B)^!$ by the extra relations listed in Proposition~\ref{KhovRobertsLSQuotient}. 

Let $M$ be a differential bigraded projective right $H^n$ module as at the beginning of Section~\ref{ModulePrelims}; assume that $M$ satisfies the algebraic condition $C_{module}$ of Definition~\ref{CModuleAlgCondition}. We first define a Type A structure $\widehat{A}(M)^{m(\B)^!}$ over $m(\B)^!$. Then we formally extend $\widehat{A}(M)$ to a Type A structure $\widehat{A}(M)^{\B \astrosun m(\B)^!}$ over $\B \astrosun m(\B)^!$.

\begin{definition}\label{FirstTypeAStrDef}
As a $\Z$-module, $\widehat{A}(M)$ is defined to be $M$, with a $\Z$-basis given by $\{ x_i \cdot h: i \in S, h \in \beta, e_L(h) = e(x_i) \}$. The idempotent ring of $m(\B)^!$ is $\I_{\beta}$; let $h' \in \beta$ be an elementary idempotent of $m(\B)^!$. Multiplying $x_i \cdot h$ by $h'$ gives $x_i \cdot h$ if $h' = h$ and zero otherwise.

Suppose the generator $x_i$ has bigrading $(j,k)$ as an element of $M$, and $h$ has grading $j'$ (or bigrading $(j',0)$) as an element of $H^n$. Then, as an element of $\widehat{A}(M)$, the bigrading of $x_i \cdot h$ is defined to be
\[
\deg_{\widehat{A}(M)}(x_i \cdot h) := (-j-\frac{1}{2}j',k).
\]

The algebra $m(\B)^!$ acts on $\widehat{A}(M)$ on the right; we will use $m_2$ to denote this action (not to be confused with $m$ here, which means ``mirror''). Let $m(b^*)$ denote either $m(b^*_{\gamma;m(h),m(h'')})$ or $m(b^*_{C;m(h),m(h'')})$. We define
\[
m_2(x_i \cdot h, m(b^*)) = \sum_j \sum_{h' \in \beta_{mult}} \tilde{c}_{i,j;h'} \tilde{\tilde{c}}_{h'h;h''} x_j \cdot h''.
\]

If $m(b^*) = m(b^*_{\gamma;m(h),m(h'')})$, then $\tilde{\tilde{c}}_{h'h;h''}$ is only nonzero for one value of $h'$, namely $h' = (W(a')a,\textrm{ all } +)$, where $h = (W(a)b,\sigma)$ and $h'' = (W(a')b,\sigma')$. For this value of $h'$, $\tilde{\tilde{c}}_{h'h;h''}$ is $1$. Thus,
\begin{equation}\label{TopSimpleMultEqn}
m_2(x_i \cdot h, m(b^*_{\gamma;m(h),m(h'')})) = \sum_j \tilde{c}_{i,j;h'} x_j \cdot h''.
\end{equation}

If $m(b^*) = m(b^*_{C;m(h),m(h'')})$, then $\tilde{\tilde{c}}_{h'h;h''}$ will be nonzero for any $h'_{\alpha}$ which equals $(W(a)a,\textrm{ minus on }\alpha)$, where $h = (W(a)b,\sigma)$ and $\alpha$ is any arc of $a$ which is part of the circle $C$ in $W(a)a$. For $h'$ equal to one of the $h'_{\alpha}$, $\tilde{\tilde{c}}_{h'h;h''}$ is $1$, and for all other $h'$, $\tilde{\tilde{c}}_{h'h;h''}$ is zero. Thus,
\begin{equation}\label{BottomSimpleMultEqn}
m_2(x_i \cdot h, m(b^*_{C;m(h),m(h'')})) = \sum_j \sum_{\textrm{left arcs } \alpha \textrm{ of }C} \tilde{c}_{i,j;h'_{\alpha}} x_j \cdot h'',
\end{equation}
Note that $m_2$ is bigrading-preserving; this follows from the degree conditions on $x_i$ and $x_j$ in the basis expansion of $d(x_i \cdot h)$ given in Section~\ref{ModulePrelims} above.

We then extend $m_2$ to an action of $m(\B)^!$ on $\widehat{A}(M)$ by imposing the associativity relation 
\[
m_2 \circ (m_2 \otimes \id) = m_2 \circ (\id \otimes \mu_2),
\]
where $\mu_2$ is the algebra multiplication on $m(\B)^!$. Below we will verify that this algebra action is well-defined. Finally, $\widehat{A}(M)$ has a differential $m_1$ given by 
\[
m_1(x_i \cdot h) = \sum_j c_{i,j} x_j \cdot h.
\]
\end{definition}

\begin{proposition}\label{mBBangActionWellDef} The action of $m(\B)^!$ on $\widehat{A}(M)$ is well-defined and associative. Thus, $\widehat{A}(M)$ is a right module over $m(\B)^!$.
\end{proposition}

\begin{proof} The action is associative by definition, once we show that it is well-defined. We may write $\B^!$ as $T(V_{\B}^*) / I^{\perp}$; thus,
\[
m(\B)^! = T(m(V_{\B}^*)) / m(\I^{\perp}),
\]
where the mirrors of the $\I_{\beta}$-bimodules $V_{\B}^*$ and $\I^{\perp}$ are defined as in Definition~\ref{IBetaGeneralMirrorDef}. Now, Definition~\ref{FirstTypeAStrDef} gives us a map from $\widehat{A}(M) \otimes_{\I_{\beta}} T(m(V_{\B}^*)) \to \widehat{A}(M)$. We want to show that if $m(r^*)$ is a generator of $m(I^{\perp})$, then multiplying any $x_i \cdot h$ by $m(r^*)$ gives zero.

The generators of $m(I^{\perp})$ are quadratic in the $m(b^*)$, and they have intrinsic degree either $1$, $\frac{3}{2}$, or $2$. For those $m(r^*)$ of intrinsic degree $2$, the equations in item~\ref{Deg4CEqns} of Proposition~\ref{HnProjModStructure} above imply that $m(r^*)$ acts as zero on any $x_i \cdot h$. For those $m(r^*)$ of intrinsic degree $\frac{3}{2}$, the equations in item \ref{Deg3CEqns} of Proposition~\ref{HnProjModStructure} similarly imply that $m(r^*)$ acts as zero on $\widehat{A}(M)$.

The generators $m(r^*)$ of $m(I^{\perp})$ which have intrinsic degree $1$ are sums of either one, two, three, or four terms $m(b^*_{\gamma}) m(b^*_{\gamma'})$ with all coefficients $+1$. For a fixed $m(r^*)$, let $m(h) \in \I_{\beta}$ denote its left idempotent and let $m(h'') \in \I_{\beta}$ denote its right idempotent. The element $h''$ of $\beta$ has degree $2$ more than $h$, as elements of $H^n$ with its intrinsic grading, and $h''$ differs from $h$ by two surgeries on its left crossingless matching. In particular, the left crossingless matchings of $h$ and $h''$ are different; this follows from inspection of the generators $m(r^*)$ of intrinsic degree $1$ which actually appear in $m(I^{\perp})$. Monomials of the form $m(b^*_{\gamma}) m(b^*_{\gamma^{\dagger}})$ do not appear as terms of these generators.

For any generators of $\widehat{A}(M)$ of the form $x_i \cdot h$ and $x_k \cdot h''$, where $h$ and $h''$ are as above, with $\deg x_k = \deg x_i + (-2,2)$ as elements of $M$, the equations from \ref{Deg2CEqns} above become
\[
\sum_{j, h' \in \beta_{\gamma}, h''' \in \beta_{\gamma}, h'''' \in \beta} \tilde{c}_{i,j;h'} \tilde{\tilde{c}}_{h'h;h''''} \tilde{c}_{j,k;h'''} \tilde{\tilde{c}}_{h'''h'''';h''} = 0;
\]
the terms involving $h' \in \beta_{\alpha}$ vanish for these choices of $h$ and $h''$. These equations imply that all generators $m(r^*)$ of $m(I^{\perp})$ of intrinsic degree $1$ act as zero on $\widehat{A}(M)$. Thus, the algebra action $m_2$ of $m(\B)^!$ on $\widehat{A}(M)$ is well-defined.

\end{proof}

\begin{proposition}\label{LeibnizRuleOnAM} The differential $m_1$ on $\widehat{A}(M)$ satisfies $m_1^2 = 0$, and the Leibniz rule 
\[
m_1 \circ m_2 = m_2 \circ (m_1 \otimes \left|\id\right|) + m_2 \circ (\id \otimes \mu_1)
\]
is satisfied, where $\mu_1$ is the differential on $m(\B)^!$. Thus, $\widehat{A}(M)$ is a differential bigraded right module over $m(\B)^!$, and hence a Type A structure over $m(\B)^!$.
\end{proposition}

\begin{proof}
First, $m_1^2 = 0$ by the equations in item~\ref{Deg0CEqns} of Proposition~\ref{HnProjModStructure}. 

We want to show that the Leibniz rule is satisfied for $\widehat{A}(M)$. Since the action of $m(\B)^!$ on $\widehat{A}(M)$ is associative, and $\mu_1$ satisfies its own Leibniz rule, it suffices to show that 
\[
m_1 \circ m_2 (x_i \cdot h, m(b^*_{\gamma;m(h),m(h'')})) = - m_2 \circ (m_1 (x_i \cdot h) \otimes m(b^*_{\gamma;m(h),m(h'')}))
\]
and 
\begin{align*}
m_1 \circ m_2 (x_i \cdot h, m(b^*_{C;m(h),m(h'')})) &= - m_2 \circ (m_1 (x_i \cdot h) \otimes m(b^*_{C;m(h),m(h'')})) \\
&+ m_2 (x_i \cdot h \otimes \mu_1(m(b^*_{C;m(h),m(h'')})));
\end{align*}
note that in the relation involving $m(b^*_{\gamma})$, the $\mu_1$ term vanishes.

The first of these two equations follows from the equations in item~\ref{Deg1CEqns} of Proposition~\ref{HnProjModStructure}. For the second equation, note that for a fixed $h \in \beta$, the only $h''$ such that $m(b^*_{C;m(h),m(h'')})$ is a generator of $m(\B)^!$ are those $h'' \in \beta$ which differ only from $h$ by changing the sign of one circle $C$ from plus to minus. We have
\[
\mu_1(m(b^*_{C;m(h),m(h'')})) = - \sum_{h'''' \in \beta} m(b^*_{\gamma;m(h),m(h'''')}) m(b^*_{\gamma^{\dagger};m(h''''),m(h'')}),
\]
where the sum is implicitly over those $h''''$ such that generators $m(b^*_{\gamma;m(h),m(h'''')})$ and $m(b^*_{\gamma^{\dagger};m(h''''),m(h'')})$ exist.

For such $h$ and $h''$, consider generators $x_i \cdot h$ and $x_k \cdot h''$ such that $\deg x_k = \deg x_i + (-2,2)$; these are the only $x_k \cdot h''$ which may appear in the basis expansion of the left or right side of the second equation above. Applying the equations in item~\ref{Deg2CEqns} of Proposition~\ref{HnProjModStructure} to $x_i \cdot h$ and $x_k \cdot h''$, we see that the second equation above holds. Thus, the Leibniz rule on $\widehat{A}(M)^{m(\B)^!}$ is satisfied.
\end{proof}

Now we formally add actions of $\B$ to $\widehat{A}(M)$, to make it a Type A structure over $\B \astrosun m(\B)^!$ rather than just over $m(\B)^!$:

\begin{definition}\label{FullTypeAStructureDefn} The Type A operation $m_2$ on $\widehat{A}(M)^{\B \astrosun m(\B)^!}$ is defined as in Definition~\ref{FirstTypeAStrDef} on the generators of $m(\B)^!$. On the generators of $\B$, it is defined by
\[
m_2(x_i \cdot h, b_{\gamma;h;h''}) = x_i \cdot h''
\]
and
\[
m_2(x_i \cdot h, b_{C;h;h''}) = x_i \cdot h''.
\]
Note that these actions are bigrading-preserving. To define the action of an arbitrary element of $\B \astrosun m(\B)^!$ on $\widehat{A}(M)$, we impose associativity of the action. Below we check that this definition respects the relations on $\B \astrosun m(\B)^!$.
\end{definition}

\begin{proposition}\label{FullMTwoActionWellDef} The action $m_2$ of $\B \astrosun m(\B)^!$ on $\widehat{A}(M)$ is well-defined; with this action and the differential $m_1: \widehat{A}(M) \to \widehat{A}(M)$ from Definition~\ref{FirstTypeAStrDef}, $\widehat{A}(M)$ is a differential bigraded module (hence Type A structure) over $\B \astrosun m(\B)^!$.
\end{proposition}

\begin{proof}
First we need to check that
\[
m_2: \widehat{A}(M) \otimes_{\I_{\beta}} (\B \astrosun m(\B)^!) \to \widehat{A}(M)
\]
is well-defined. Recall that the relation ideal $J_{full}$ of $\B \astrosun m(\B)^!$ was defined to be
\begin{align*}
J_{full} := T(V_{full}) \cdot &((J_{\B} \cap (T^1(V_{\B}) \oplus T^2(V_{\B}))) \\
&\oplus (J_{m(\B)^!} \cap (T^1(V_{m(\B)^!}) \oplus T^2(V_{m(\B)^!}))) \\
&\oplus J_{extra}) \cdot T(V_{full}).
\end{align*}
By Proposition~\ref{mBBangActionWellDef}, the generators of $J_{m(\B)^!} \cap (T^1(V_{m(\B)^!}) \oplus T^2(V_{m(\B)^!}))$ act as zero on $\widehat{A}(M)$. It is immediate from the definition of the action of $\B$ on $\widehat{A}(M)$ that the generators of $J_{\B} \cap (T^1(V_{\B}) \oplus T^2(V_{\B}))$ act as zero on $\widehat{A}(M)$.

Thus, to show that $m_2$ is well-defined, it remains to show that the generators of $J_{extra}$ act as zero on $\widehat{A}(M)$. These generators are listed in items \ref{JExtraRels1} through \ref{JExtraRels5} of Definition~\ref{JExtraDef}.

\begin{itemize}
\item

Consider a relation 
\[
b_{\gamma;h_1,h_2} m(b^*_{\eta';m(h_2),m(h_3)}) - m(b^*_{\eta;m(h_1),m(h'_2)}) b_{\gamma';h'_2,h_3}
\]
from item \ref{JExtraRels1} of Definition~\ref{JExtraDef}. Write $h_1 = (W(a_1)b_1,\sigma_1)$, and let $x_i \cdot h_1$ be a generator of $\widehat{A}(M)$. Multiplying $x_i \cdot h_1$ by $b_{\gamma;h_1,h_2}$, we get $x_i \cdot h_2$ where $h_2 = (W(a_1)b_2,\sigma_2)$. Multiplying $x_i \cdot h_2$ by $m(b^*_{\eta';m(h_2),m(h_3)})$, with $h_3 = (W(a_2)b_2,\sigma_3)$, by equation~\ref{TopSimpleMultEqn} we get
\[
\sum_j \tilde{c}_{i,j;h'} x_j \cdot h_3
\]
where $h' \in \beta_{mult}$ is $(W(a_2)a_1, \textrm{ all } +)$. 

On the other hand, if we first multiply $x_i \cdot h_1$ by $m(b^*_{\eta;m(h_1),m(h'_2)})$, by equation~\ref{TopSimpleMultEqn} we get 
\[
\sum_j \tilde{c}_{i,j;h'} x_j \cdot h'_2,
\]
where $h'$ is also $(W(a_2)a_1, \textrm{ all } +)$. 

If we multiply this result by $b_{\gamma';h'_2,h_3}$, we get
\[
\sum_j \tilde{c}_{i,j;h'} x_j \cdot h_3.
\]
Thus, generators of $J_{extra}$ from item \ref{JExtraRels1} of Definition~\ref{JExtraDef} act as zero on $\widehat{A}(M)$.

\item
For relations
\[
b_{C;h_1,h_2} m(b^*_{\eta;m(h_2),m(h_3)}) - m(b^*_{\eta;m(h_1),m(h'_2)}) b_{C;h'_2,h_3}
\]
from item \ref{JExtraRels2} of Definition~\ref{JExtraDef}, the argument is essentially the same. If $h_1 = (W(a_1)b_1, \sigma_1)$ and $h'_2 = (W(a_2)b_1,\sigma'_2)$, then $h'$ is again $(W(a_2)a_1, \textrm{ all } +)$. We still use equation~\ref{TopSimpleMultEqn}.

\item
Consider a relation
\[
b_{\gamma;h_1,h_2} m(b^*_{C;m(h_2),m(h_3)}) - m(b^*_{C;m(h_1),m(h'_2)}) b_{\gamma;h'_2,h_3}
\]
from item \ref{JExtraRels3} of Definition~\ref{JExtraDef}. Let $x_i \cdot h_1$ be a generator of $\widehat{A}(M)$, and write $h_1 = (W(a_1)b_1,\sigma_1)$. Multiplying $x_i \cdot h_1$ by $b_{\gamma;h_1,h_2}$, we get $x_i \cdot h_2$ where $h_2 = (W(a_1)b_2,\sigma_2)$. Multiplying $x_i \cdot h_2$ by $m(b^*_{C;m(h_2),m(h_3)})$, with $h_3 = (W(a_1)b_2,\sigma_3)$, by equation~\ref{BottomSimpleMultEqn} we get
\[
\sum_j \sum_{\textrm{left arcs } \alpha \textrm{ of } C} \tilde{c}_{i,j;h'_{\alpha}} x_j \cdot h_3
\]
where $h'_{\alpha}$ is $(W(a_1)a_1, \textrm{ minus on } \alpha)$.

On the other hand, if we first multiply $x_i \cdot h_1$ by $m(b^*_{C;m(h_1),m(h'_2)})$, by equation~\ref{BottomSimpleMultEqn} we get 
\[
\sum_j \sum_{\textrm{left arcs } \alpha \textrm{ of } C} \tilde{c}_{i,j;h'_{\alpha}} x_j \cdot h'_2,
\]
where $h'_{\alpha}$ is again $(W(a_1)a_1, \textrm{ minus on } \alpha)$.

If we multiply this result by $b_{\gamma;h'_2,h_3}$, we get
\[
\sum_j \sum_{\textrm{left arcs } \alpha \textrm{ of } C} \tilde{c}_{i,j;h'_{\alpha}} x_j \cdot h_3.
\]
Thus, generators of $J_{extra}$ from item \ref{JExtraRels3} of Definition~\ref{JExtraDef} act as zero on $\widehat{A}(M)$.

\item
For relations
\[
b_{C;h_1,h_2} m(b^*_{C';m(h_2),m(h_3)}) - m(b^*_{C';m(h_1),m(h'_2)}) b_{C;h'_2,h_3}
\]
from item \ref{JExtraRels4} of Definition~\ref{JExtraDef}, the argument is the same as for relations from item \ref{JExtraRels3}.

\item
Finally, consider a relation
\begin{align*}
b_{\gamma;h_1,h_2} m(b^*_{C;m(h_2),m(h_3)}) &- m(b^*_{C';m(h_1),m(h'_2)}) b_{\gamma;h'_2,h_3} \\
&- m(b^*_{C'';m(h_1),m(h''_2)}) b_{\gamma;h''_2,h_3}
\end{align*}
from item \ref{JExtraRels5} of Definition~\ref{JExtraDef}. Write $h_1 = (W(a_1)b_1,\sigma_1)$, and let $x_i \cdot h_1$ be a generator of $\widehat{A}(M)$. Multiplying $x_i \cdot h_1$ by $b_{\gamma;h_1,h_2}$, we get $x_i \cdot h_2$ where $h_2 = (W(a_1)b_2,\sigma_2)$. Multiplying $x_i \cdot h_2$ by $m(b^*_{C;m(h_2),m(h_3)})$, with $h_3 = (W(a_1)b_2,\sigma_3)$, by equation~\ref{BottomSimpleMultEqn} we get
\[
\sum_j \sum_{\textrm{left arcs } \alpha \textrm{ of } C} \tilde{c}_{i,j;h'_{\alpha}} x_j \cdot h_3
\]
where $h'_{\alpha}$ is $(W(a_1)a_1, \textrm{ minus on } \alpha)$.

On the other hand, if we first multiply $x_i \cdot h_1$ by $m(b^*_{C';m(h_1),m(h'_2)})$, we get 
\[
\sum_j \sum_{\textrm{left arcs } \alpha \textrm{ of } C'} \tilde{c}_{i,j;h'_{\alpha}} x_j \cdot h'_2,
\]
where $h'_{\alpha}$ is again $(W(a_1)a_1, \textrm{ minus on } \alpha)$. Multiplying by $b_{\gamma';h'_2,h_3}$, we get
\[
\sum_j \sum_{\textrm{left arcs } \alpha \textrm{ of } C'} \tilde{c}_{i,j;h'_{\alpha}} x_j \cdot h_3.
\]
Finally, if we multiply $x_i \cdot h_1$ by $m(b^*_{C'';m(h_1),m(h''_2)})$, we get 
\[
\sum_j \sum_{\textrm{left arcs } \alpha \textrm{ of } C''} \tilde{c}_{i,j;h'_{\alpha}} x_j \cdot h''_2.
\]
Multiplying this result by $b_{\gamma;h''_2,h_3}$, we get
\[
\sum_j \sum_{\textrm{left arcs } \alpha \textrm{ of } C''} \tilde{c}_{i,j;h'_{\alpha}} x_j \cdot h_3.
\]

Now, since $\gamma$ was assumed to join the circles $C'$ and $C''$ to produce $C$ via a bridge on the right side of $\{0\} \times \R$, the set of left arcs $\alpha$ of $C$ is the disjoint union of the sets of left arcs of $C'$ and $C''$. 
Thus, the relation
\begin{align*}
b_{\gamma;h_1,h_2} m(b^*_{C;m(h_2),m(h_3)}) &- m(b^*_{C';m(h_1),m(h'_2)}) b_{\gamma;h'_2,h_3} \\
&- m(b^*_{C'';m(h_1),m(h''_2)}) b_{\gamma;h''_2,h_3}
\end{align*}
acts as zero on $\widehat{A}(M)$. For relations of the form
\begin{align*}
&b_{\gamma;h_1,h'_2} m(b^*_{C';m(h'_2),m(h_3)}) + b_{\gamma;h_1,h''_2} m(b^*_{C'';m(h''_2),m(h_3)}) \\
&- m(b^*_{C;m(h_1),m(h_2)}) b_{\gamma;h_2,h_3},
\end{align*}
where $\gamma$ splits $C$ into $C'$ and $C''$, the argument is analogous. Thus, generators of $J_{extra}$ from item \ref{JExtraRels5} of Definition~\ref{JExtraDef} act as zero on $\widehat{A}(M)$.
\end{itemize}

At this point, we have shown that the action $m_2$ of $\B \astrosun m(\B)^!$ on $\widehat{A}(M)$ is well-defined. It is associative by definition. To show that the Leibniz rule is satisfied, it suffices by associativity to check it on the generators of $\B$ and of $m(\B)^!$, and we have already done this for the generators of $m(\B)^!$ in Proposition~\ref{mBBangActionWellDef}.

Let $x_i \cdot h_1$ be a generator of $\widehat{A}(M)$ and let $b_{\gamma;h_1,h_2}$ be a generator of $\B$. Then
\begin{align*}
m_1 \circ m_2(x_i \cdot h_1, b_{\gamma;h_1,h_2}) &= m_1 (x_i \cdot h_2) \\
&= \sum_j c_{i,j} x_j \cdot h_2,
\end{align*}
while
\begin{align*}
m_2(m_1 \otimes \left|\id\right|)(x_i \cdot h_1, b_{\gamma; h_1,h_2}) &= m_2\bigg(\sum_j c_{i,j} x_j \cdot h_1, b_{\gamma; h_1,h_2}\bigg) \\
&= \sum_j c_{i,j} x_j \cdot h_2,
\end{align*}
and the $m_2 \circ (\id \otimes \mu_1)$ term is zero because $\mu_1 = 0$ on $\B$. The argument is unchanged for generators $b_{C;h_1,h_2}$. Thus, the Leibniz rule
\[
m_1 \circ m_2 = m_2 \circ (m_1 \otimes \left|\id\right|) + m_2 \circ (\id \otimes \mu_1)
\]
holds, and $\widehat{A}(M)$ is a differential bigraded right module over $\B \astrosun m(\B)^!$.

\end{proof}

Now we will consider the case $M = [T]^{Kh}$, where $T$ is a tangle diagram in $\R_{\leq 0} \times \R$. For this choice of $M$, we get a Type A structure over the quotient algebra $\B \Gamma_n$:

\begin{proposition}\label{TangleTypeADescends}
The extra relations from Proposition~\ref{KhovRobertsLSQuotient} act as zero on the $\B \astrosun m(\B)^!$-module $\widehat{A}([T]^{Kh})$ defined above in Definition~\ref{FullTypeAStructureDefn}. Thus, $\widehat{A}([T]^{Kh})$ descends to a differential bigraded right module over the quotient algebra $\B \Gamma_n$ of $\B \astrosun m(\B)^!$ by these relations.
\end{proposition}

\begin{proof}
Since the relations from Proposition~\ref{KhovRobertsLSQuotient} involve only quadratic monomials in the generators $m(b^*)$ of $m(\B)^!$, with no generators from $\B$ appearing, it suffices to show that these relations act as zero on the $m(\B)^!$-module $\widehat{A}([T]^{Kh})$ defined in Definition~\ref{FirstTypeAStrDef}.

Consider a tetrahedron in the graph $G$ of Proposition~\ref{KhovRobertsLSQuotient}, with vertices $a,b,c$, and $d$ as labeled in that proposition. We will show that the relation term $a + c$ acts as zero on $\widehat{A}([T]^{Kh})$; the proofs for the remaining extra relation terms are exactly analogous.

We may write out 
\[
a = m(b^*_{\gamma;m(h_1),m(h_2)}) m(b^*_{\eta';m(h_2),m(h_3)})
\]
and 
\[
c = m(b^*_{\eta;m(h_1),m(h''_2)}) m(b^*_{\gamma';m(h''_2),m(h_3)})
\]
as in Proposition~\ref{KhovRobertsLSQuotient}. Suppose we have two generators of $\widehat{A}([T]^{Kh})$ of the form $x_{00} \cdot h_1$ and $x_{11} \cdot h_3$. Here, $T$ may have more than two crossings, but for two designated crossings, $x_{00}$ has the zero-resolution at both and $x_{11}$ has the one-resolution at both (and $x_{00}$ and $x_{11}$ agree at all other crossings). We assume that changing $x_{00}$ to $x_{10}$ has the effect of surgery on $\gamma$, while changing $x_{00}$ to $x_{01}$ has the effect of surgery on $\eta$.

To show that $a + c$ acts as zero on $\widehat{A}([T]^{Kh})$, it suffices to show $m_2(x_{00} \cdot h_1, a + c)$ has zero coefficient on $x_{11} \cdot h_3$. We can compute $m_2(x_{00} \cdot h_1, a)$ and $m_2(x_{00} \cdot h_1, c)$ using associativity and equation~\ref{TopSimpleMultEqn}: we have
\begin{align*}
m_2(x_{00} \cdot h_1, a) &= \tilde{c}_{00,10;h'} m_2(x_{10} \cdot h_2, m(b^*_{\eta';m(h_2),m(h_3)})) \\
&= \tilde{c}_{00,10;h'} \tilde{c}_{10,11;h''} x_{11} \cdot h_3;
\end{align*}
this computation ignores terms which do not contribute to a coefficient on $x_{11} \cdot h_3$. The elements $h'$ and $h''$ of $\beta_{\gamma}$ are uniquely determined. Similarly, the coefficient of $m_2(x_{00} \cdot h_1, c)$ on $x_{11} \cdot h_3$ is $\tilde{c}_{00,01;h'''} \tilde{c}_{01,11;h''''}$, for two further uniquely-defined elements $h'''$ and $h''''$ of $\beta_{\gamma}$.

By item~\ref{KhCoeffs4} of Example~\ref{HnCCoeffsEx}, we have:
\begin{itemize}
\item $\tilde{c}_{00,10;h'} = (-1)^{\#_1(00,10)}$,
\item $\tilde{c}_{10,11;h''} = (-1)^{\#_1(10,11)}$,
\item $\tilde{c}_{00,01;h'''} = (-1)^{\#_1(00,01)}$, and
\item $\tilde{c}_{01,11;h''''} = (-1)^{\#_1(01,11)}$;
\end{itemize}
recall that $\#_1(i,j)$ denote the number of $1$-resolutions of crossings in $x_i$ among those crossings which occur earlier than the changed crossing (going from $x_i$ to $x_j$) in the ordering on crossings of $T$ (which is implicitly assumed, as usual, to be part of the choice of $T$).

Since
\[
(-1)^{\#_1(00,10)} (-1)^{\#_1(10,11)} + (-1)^{\#_1(00,01)} (-1)^{\#_1(01,11)} = 0,
\]
we can conclude that the coefficient of $m_2(x_{00} \cdot h_1, a + c)$ on $x_{11} \cdot h_3$ is zero, for all possible pairs $x_{00} \cdot h_1$ and $x_{11} \cdot h_3$. Thus, the extra relation terms of Proposition~\ref{KhovRobertsLSQuotient} act as zero on $\widehat{A}([T]^{Kh})$.

\end{proof}

\begin{proposition}\label{RobertsKhovanovTypeA}
Roberts' Type A structure from \cite{RtypeA} agrees with $\widehat{A}([T]^{Kh})$, the module over $\B \Gamma_n$ constructed in Proposition~\ref{TangleTypeADescends}.
\end{proposition}

\begin{proof}

First, $\widehat{A}([T]^{Kh})$ has the same $\Z$-basis, with the same bigradings and action of the idempotent ring $\I_{\beta}$, as Roberts' Type A structure. We can use the data of Example~\ref{HnCCoeffsEx} to check that the differentials $m_1$ agree, and that the algebra actions $m_2$ agree under the identification of $\B \astrosun m(\B)^!$ with $\B \Gamma_n$.

For the differentials, we have $m_1(x_i \cdot h) = \sum_j (-1)^{\#_1(i,j)} x_j \cdot h$, where the sum is over those $x_j$ related to $x_i$ by crossing changes from items \ref{KhCoeffs1} or \ref{KhCoeffs2} of Example~\ref{HnCCoeffsEx}. This formula also gives Roberts' differential $m_1 = d_{APS}$ as specified in Section 3.3 of \cite{RtypeA}.

It suffices to check that the algebra actions $m_2$ agree when multiplying by the generators of $\B$ and $m(\B)^!$. First, for a generator $x_i \cdot h_1$ of $\widehat{A}([T]^{Kh})$ and a generator $b_{\gamma;h_1,h_2}$ of $\B$, we have
\[
m_2(x_i \cdot h_1, b_{\gamma;h_1,h_2}) = x_i \cdot h_2,
\]
agreeing with Roberts' definition of the action of $\overrightarrow{e}_{\gamma;h_1,h_2}$ in item 5 of his definition of $m_2$, in Section 4 of \cite{RtypeA}. Similarly, our action of $b_{C;h_1,h_2}$ is the same as Roberts' action of $\overrightarrow{e}_{C;h_1,h_2}$, defined in item 2 of his definition of $m_2$.

For a generator $m(b^*_{\gamma;m(h_1),m(h_2)})$ of $m(\B)^!$, we have
\[
m_2(x_i \cdot h_1, m(b^*_{\gamma;m(h_1),m(h_2)})) = \sum_j \tilde{c}_{i,j;h'} x_j \cdot h_2,
\]
where if $h_1 = (W(a_1)b_1,\sigma_1)$ and $h_2 = (W(a_2)b_1,\sigma_2)$, then $h' = (W(a_2)a_1, \textrm{ all }+)$, and the coefficient $\tilde{c}_{i,j;h'}$ equals zero or $(-1)^{\#_1(i,j)}$ according to item \ref{KhCoeffs4} of Example~\ref{HnCCoeffsEx}. Thus,
\[
m_2(x_i \cdot h_1, m(b^*_{\gamma;m(h_1),m(h_2)})) = \sum_j (-1)^{\#_1(i,j)} x_j \cdot h_2,
\]
where the sum is over the subset of $j$ making $\tilde{c}_{i,j;h'}$ nonzero. This algebra action agrees with the action of $\overleftarrow{e}_{\gamma;h_1,h_2}$ as defined in item 4 of Roberts' definition of $m_2$.

Finally, for a generator $m(b^*_{C;m(h_1),m(h_2)})$ of $m(\B)^!$, we have
\[
m_2(x_i \cdot h_1, m(b^*_{C;m(h_1),m(h_2)})) = \sum_j \sum_{\textrm{left arcs }\alpha\textrm{ of }C} \tilde{c}_{i,j;h'_{\alpha}} x_j \cdot h_2,
\]
where if $h_1 = (W(a)b,\sigma_1)$ and $h_2 = (W(a)b,\sigma_2)$, then $h'_{\alpha} = (W(a)a, \textrm{ minus on }\alpha)$, and the coefficient $\tilde{c}_{i,j;h'}$ equals zero or $(-1)^{\#_1(i,j)}$ according to item \ref{KhCoeffs3} above. Thus,
\[
m_2(x_i \cdot h_1, m(b^*_{C;m(h_1),m(h_2)})) = \sum_j (-1)^{\#_1(i,j)} x_j \cdot h_2,
\]
where the sum is over the subset of $j$ making some $\tilde{c}_{i,j;h'_{\alpha}}$ nonzero (note that, given such $j$, the element $h'_{\alpha}$ is uniquely determined). This algebra action agrees with the action of $\overleftarrow{e}_{C;h_1,h_2}$ as defined in item 3 of Roberts' definition of $m_2$.
\end{proof}

\subsection{Type D structures}\label{RobertsTypeDStrSect}

Given a chain complex $M$ of projective graded right $H^n$-modules, in Definition~\ref{FullTypeAStructureDefn} we defined a Type A structure $\widehat{A}(M)$ over $\B \astrosun m(\B)^!$. For a tangle diagram $T$, $\widehat{A}([T]^{Kh})$ descends to a Type A structure over $\B \Gamma_n$ by Proposition~\ref{TangleTypeADescends}, which agrees with Roberts' Type A structure by Proposition~\ref{RobertsKhovanovTypeA}.

Now suppose $N$ is a chain complex of graded projective left $H^n$-modules. We will define another type of mirroring operation which, when applied to $N$, yields a complex of right modules. This is not the most general definition possible, but it will suffice for our purposes.
 
\begin{definition}\label{HnComplexMirrorDef}
Let $N$ be a chain complex of projective graded left $H^n$-modules. Viewing $N$ as a differential bigraded projective left $H^n$-module, write $N = \oplus_i H^n \cdot x_i[j_i,k_i]$, with $H^n$ acting by left multiplication. Then we may define a differential bigraded projective right $H^n$-module $m(N)$, called the mirror of $N$, as
\[
m(N) := \oplus_i m(x_i) [j_i,k_i] \cdot H^n,
\]
where the $m(x_i)$ are formal mirrors of the $x_i$, with the same idempotents as the $x_i$, and $H^n$ acts by right multiplication. Let $mirr$ denote the map from $N$ to $m(N)$ such that 
\[
mirr(h \cdot x_i) = m(x_i) \cdot m(h).
\]
The inverse of $mirr: N \to m(N)$ is $mirr: m(N) \to m(m(N)) = N$ (using a simple generalization of the above definition). If $m_1$ denotes the differential on $N$, then the differential on $m(N)$ is $mirr \circ m_1 \circ mirr$.

\end{definition}

\begin{remark}\label{TwoTypesOfMirrorRemark}
Although the geometric content of both Definition~\ref{HnComplexMirrorDef} and (variants of) Definition~\ref{IBetaGeneralMirrorDef} is just the reflection across the line $\{0\} \times \R$, the algebraic consequences of this reflection are different in Definition~\ref{HnComplexMirrorDef}. Whereas in Definition~\ref{IBetaGeneralMirrorDef}, left modules remain left modules and right modules remain right modules under mirroring, in Definition~\ref{HnComplexMirrorDef} left modules are sent to right modules and vice-versa.
\end{remark}

\begin{definition}\label{CModuleForLeftMods}
A chain complex $N$ of projective graded left $H^n$-modules satisfies the condition $C_{module}$ if and only if $m(N)$ satisfies the condition $C_{module}$ as defined in Definition~\ref{CModuleAlgCondition}.
\end{definition}
If $N$ satisfies $C_{module}$, then we can take the box tensor product of $\widehat{A}(m(N))$ and the Type DD bimodule ${^{\B \astrosun m(\B)^!}}K^{m(\B \astrosun m(\B)^!)^{op}}$ to get a (left) Type D structure over $m(\B \astrosun m(\B)^!)$. Below we define this tensor product precisely. It is a slight modification of Definition~\ref{XBoxWithSigns}; again, we will not give the definition in full generality.

\begin{definition}\label{XBoxWithDDDef} 
Let $\B$ be a differential bigraded algebra over an idempotent ring $\I$. Let $\widehat{A}$ be a differential bigraded right module over $\B$. Assume $\widehat{A}$ is free as a $\Z$-module, with a $\Z$-basis consisting of elements which are grading-homogeneous and have a unique right idempotent. 

Let $\B'$ be another differential bigraded algebra over $\I$, and let $\DD$ be a rank-one Type DD bimodule over $\B$ and $\B'$ with DD operation $\delta_{DD}: \I \to \B \otimes_{\I} \B'$. The Type D structure $\widehat{A} \boxtimes \DD$ over $\B'$, as a $\Z$-module, is 
\[
\widehat{A} \boxtimes \DD := \widehat{A} \otimes_{\I} \DD = \widehat{A} \otimes_{\I} \I = \widehat{A}.
\]
The idempotent ring $\I$ has a right action on $\widehat{A}$, which we will view instead as a left action (since $\I$ is commutative, we may view right actions as left actions and vice-versa). Since $\DD$ is a rank-one DD bimodule, the left and right actions of $\I$ on $\DD$ are the same. There is a bigrading on $\widehat{A} \boxtimes \DD$ inherited from that on $\widehat{A}$ (recall that $\DD$ is contained in bigrading $(0,0)$).

The Type D operation $\delta^{\boxtimes}: \widehat{A} \boxtimes \DD \to \B' \otimes_{\I} (\widehat{A} \boxtimes \DD)$ is defined by
\[
\delta^{\boxtimes} := m_1 + \xi \circ (m_2 \otimes \id) \circ (\id \otimes \delta_{DD}): \widehat{A} \to \B' \otimes_{\I} \widehat{A},
\]
where $m_1$ and $m_2$ are the Type A operations on $\widehat{A}$, and $\xi: \widehat{A} \otimes_{\I} (\B')^{op} \to \B' \otimes_{\I} \widehat{A}$ is defined by 
\[
\xi(x \otimes (b')^{op}) = (-1)^{(\deg_h x)(\deg_h b')} b' \otimes x.
\]
More precisely, the second summand is the composition
\[
\widehat{A} \xrightarrow{\id \otimes \delta_{DD}} \widehat{A} \otimes_{\I} \B \otimes_{\I} (\B')^{op} \xrightarrow{m_2 \otimes \id} \widehat{A} \otimes_{\I} (\B')^{op} \xrightarrow{\xi} \B' \otimes_{\I} \widehat{A}.
\]
The map $\delta^{\boxtimes}$ has bidegree $(0,+1)$.
\end{definition}

\begin{proposition}
$(\widehat{A} \boxtimes \DD, \delta^{\boxtimes})$ is a well-defined Type D structure over $\B'$.
\end{proposition}

\begin{proof}
First, since $\widehat{A}$ was assumed to have a $\Z$-basis $y_i$, with each $y_i$ grading-homogeneous and having unique idempotents, the same is true for $\widehat{A} \boxtimes \DD \cong \widehat{A}$.

To verify the Type D structure relations, we must show that 
\[
(\mu_1 \otimes \left|\id\right|) \circ \delta^{\boxtimes} + (\mu_2 \otimes \id) \circ (\id \otimes \delta^{\boxtimes}) \circ \delta^{\boxtimes} = 0.
\]
Substituting in the definition of $\delta^{\boxtimes}$ and simplifying some terms, we want to show that
\begin{equation}\label{XBoxTypeDTerms}
\begin{aligned}
&(\mu_1 \otimes \left|\id\right|) \circ \xi \circ (m_2 \otimes \id) \circ (\id \otimes \delta_{DD}) \\
&+ \xi \circ (m_2 \otimes \id) \circ (\id \otimes \delta_{DD}) \circ m_1 \\
&+ (\id \otimes m_1) \circ \xi \circ (m_2 \otimes \id) \circ (\id \otimes \delta_{DD}) \\
&+ (\mu_2 \otimes \id) \circ (\id \otimes (\xi \circ (m_2 \otimes \id) \circ (\id \otimes \delta_{DD}))) \\
& \qquad \circ (\xi \circ (m_2 \otimes \id) \circ (\id \otimes \delta_{DD})) \\
&= 0.
\end{aligned}
\end{equation}
We may rewrite the final term of \ref{XBoxTypeDTerms} as 
\begin{equation}\label{FinalTerm}
\xi \circ (m_2 \otimes \id) \circ (\id \otimes \mu_2 \otimes \mu_2) \circ (\id \otimes \sigma) \circ (\id \otimes \id \otimes \delta_{DD} \otimes \id) \circ (\id \otimes \delta_{DD}),
\end{equation}
where 
\[
\sigma: \B \otimes \B \otimes (\B')^{op} \otimes (\B')^{op} \to \B \otimes \B \otimes (\B')^{op} \otimes (\B')^{op}
\]
was defined in Definition~\ref{GeneralDDBimodDef}.

To verify that term~\ref{FinalTerm} is equal to the final term of \ref{XBoxTypeDTerms}, let $x$ be a generator of $\widehat{A}$. Write $\delta_{DD}(1) = \sum_i b_i \otimes (b'_i)^{op}$; then we have
\begin{align*}
&(\mu_2 \otimes \id) \circ (\id \otimes (\xi \circ (m_2 \otimes \id) \circ (\id \otimes \delta_{DD}))) \circ (\xi \circ (m_2 \otimes \id) \circ (\id \otimes \delta_{DD})) (x) \\
&= \sum_{i,j} (-1)^{(\deg_h b'_i)(\deg_h(xb_i)) + (\deg_h b'_j)(\deg_h (x b_i b_j))} (b'_i b'_j) \otimes (x b_i b_j).
\end{align*}
On the other hand, we have
\begin{align*}
&\xi \circ (m_2 \otimes \id) \circ (\id \otimes \mu_2 \otimes \mu_2) \circ (\id \otimes \sigma) \circ (\id \otimes \id \otimes \delta_{DD} \otimes \id) \circ (\id \otimes \delta_{DD}) (x) \\
&= \sum_{i,j} (-1)^{\deg_h b_j \deg_h (b'_i) + \deg_h(x b_i b_j) \deg_h(b'_i b'_j)} (b'_i b'_j) \otimes (x b_i b_j).
\end{align*}
A direct computation, using the additivity of $\deg_h$ under algebra multiplication, verifies that the signs in these expressions are equal.

Now, we may write term~\ref{FinalTerm} as 
\[
\xi \circ (m_2 \otimes \id) \circ (\id \otimes ((\mu_2 \otimes \mu_2) \circ \sigma \circ (\id \otimes \delta_{DD} \otimes \id) \circ \delta_{DD})).
\]
Using the Type DD bimodule relations for $\delta$, we may replace 
\[
\id \otimes ((\mu_2 \otimes \mu_2) \circ \sigma \circ (\id \otimes \delta_{DD} \otimes \id) \circ \delta_{DD})
\]
with 
\[
-(\mu_1 \otimes \left|\id\right|) \circ \delta_{DD} - (\id  \otimes \mu_1) \circ \delta_{DD}.
\]
Then term~\ref{FinalTerm} is equal to
\begin{align*}
&-\xi \circ (m_2 \otimes \id) \circ (\id \otimes ((\mu_1 \otimes \left|\id\right|) \circ \delta_{DD})) \\
& \qquad - \xi \circ (m_2 \otimes \id) \circ (\id \otimes ((\id  \otimes \mu_1) \circ \delta_{DD})) \\
&= -\xi \circ (m_2 \otimes \id) \circ (\id \otimes \mu_1 \otimes \left|\id\right|) \circ (\id \otimes \delta_{DD}) \\
& \qquad - \xi \circ (m_2 \otimes \id) \circ (\id \otimes \id \otimes \mu_1) \circ (\id \otimes \delta_{DD}).
\end{align*}

The term
\[
- \xi \circ (m_2 \otimes \id) \circ (\id \otimes \id \otimes \mu_1) \circ (\id \otimes \delta_{DD})
\]
above cancels the first term
\[
(\mu_1 \otimes \left|\id\right|) \circ \xi \circ (m_2 \otimes \id) \circ (\id \otimes \delta_{DD})
\]
of the terms in \ref{XBoxTypeDTerms}, whose sum we are trying to show is zero. The remaining terms of \ref{XBoxTypeDTerms} are:
\begin{itemize}
\item $-\xi \circ (m_2 \otimes \id) \circ (\id \otimes \mu_1 \otimes \left|\id\right|) \circ (\id \otimes \delta_{DD})$,
\item $\xi \circ (m_2 \otimes \id) \circ (\id \otimes \delta_{DD}) \circ m_1$, and
\item $(\id \otimes m_1) \circ \xi \circ (m_2 \otimes \id) \circ (\id \otimes \delta_{DD})$. 
\end{itemize}
The final of these may be written as
\[
\xi \circ (m_1 \otimes \left|\id\right|) \circ (m_2 \otimes \id) \circ (\id \otimes \delta_{DD}) = \xi \circ ((m_1 \circ m_2) \otimes \left|\id\right|) \circ (\id \otimes \delta_{DD}).
\]
We may use the Leibniz rule on $\widehat{A}$ to replace $m_1 \circ m_2$ with $m_2 \circ (m_1 \otimes \left|\id\right|) + m_2 \circ (\id \otimes \mu_1)$. Thus, the final of the three remaining terms is equal to
\begin{align*}
&\xi \circ (m_2 \otimes \id) \circ (m_1 \otimes \left|\id\right| \otimes \left|\id\right|) \circ (\id \otimes \delta_{DD}) \\
& \qquad + \xi \circ (m_2 \otimes \id) \circ (\id \otimes \mu_1 \otimes \left|\id\right|) \circ (\id \otimes \delta_{DD}).
\end{align*}
The second of these summands cancels with the first of the other three remaining terms listed above, so it remains to show that
\[
\xi \circ (m_2 \otimes \id) \circ (\id \otimes \delta_{DD}) \circ m_1 + \xi \circ (m_2 \otimes \id) \circ (m_1 \otimes \left|\id\right| \otimes \left|\id\right|) \circ (\id \otimes \delta_{DD}) = 0.
\]
This follows from the fact that 
\[
(m_1 \otimes \left|\id\right| \otimes \left|\id\right|) \circ (\id \otimes \delta_{DD}) = - (\id \otimes \delta_{DD}) \circ m_1;
\]
indeed, since all generators of $\DD \cong \I$ have bigrading $(0,0)$, the element $\delta(1)$ has homological degree $1$.

\end{proof}

Applying this construction to $\widehat{A} = \widehat{A}(m(N))$ with $\DD = {^{\B \astrosun m(\B)^!}}K^{m(\B \astrosun m(\B)^!)^{op}}$, which is a Type DD bimodule over $\B \astrosun m(\B)^!$ and $m(\B \astrosun m(\B)^!)$, we get a Type D structure $\A(m(N)) \boxtimes {^{\B \astrosun m(\B)^!}}K^{m(\B \astrosun m(\B)^!)^{op}}$ over $m(\B \astrosun m(\B)^!)$. We can apply another mirroring operation, analogous to Definition~\ref{DDMirroringDef} and in the spirit of Definition~\ref{IBetaGeneralMirrorDef}, to get a Type D structure $\D(N)$ over $\B \astrosun m(\B)^!$:

\begin{definition}\label{TypeDMirroringDef}
Let $\B$ be a differential bigraded algebra over the idempotent ring $\I_{\beta}$, and let $(\D,\delta)$ be a Type D structure over $\B$. The mirrored Type D structure $(m(\D),\delta')$ is defined as follows: as an $\I_{\beta}$-module, $m(\D)$ is the mirror of $\D$ as defined in Definition~\ref{IBetaGeneralMirrorDef}. As usual, denote the natural map from $\D$ to $m(\D)$ or $m(\D)$ to $\D$ by $mirr$. The Type D operation on $m(\D)$ is 
\[
\delta' = m(\D) \xrightarrow{mirr} \D \xrightarrow{\delta} \B \otimes \D  \xrightarrow{mirr \otimes mirr} m(\B) \otimes m(\D).
\]

\end{definition}

Thus, given $N$, we can construct a Type D structure over $\B \astrosun m(\B)^!$:
\begin{definition}\label{TypeDOverBomBDef}
Let $N$ be a chain complex of graded projective left $H^n$-modules satisfying the algebraic condition $C_{module}$ of Definition~\ref{CModuleForLeftMods}. The Type D structure $\D(N)$ over $\B \astrosun m(\B)^!$ is defined to be 
\[
\D(N) := m\bigg(\widehat{A}(m(N)) \boxtimes {^{\B \astrosun m(\B)^!}}K^{m(\B \astrosun m(\B)^!)^{op}}\bigg).
\]
\end{definition}
By Proposition~\ref{DandDDQuotient}, $\D(N)$ descends to a Type D structure over the quotient algebra $\B \Gamma_n$:
\begin{definition}\label{TypeDOverBGnDef} If $N$ is a chain complex $N$ of graded left projective $H^n$-modules satisfying $C_{module}$, the Type D structure $\D(N)$ over $\B \Gamma_n$ associated to $N$ is obtained from the Type D structure $\D(N)$ over $\B \astrosun m(\B)^!$ defined in Definition~\ref{TypeDOverBomBDef} using Proposition~\ref{DandDDQuotient}.
\end{definition}

For convenience, we describe the Type D operation $\delta$ on $\D(N)$ explicitly from the differential $d_N$ on $N$. Let $\{h \cdot x'_i: e(x'_i) = e_R(h)\}$ be a $\Z$-basis for $N$; by $C_{module}$, we may expand $d_N(x'_i)$ in this basis as
\[
d_N(x'_i) = \sum_{j} c'_{i,j} x'_j + \sum_{j, h' \in \beta_{mult}} \tilde{c'}_{i,j;h'} h' \cdot x'_j.
\]
Then
\[
d_N(h \cdot x'_i) = \sum_{j} c'_{i,j} x'_j \cdot h + \sum_{j, h' \in \beta_{mult}, h'' \in \beta} \tilde{c'}_{i,j;h'} \tilde{\tilde{c}}_{hh';h''} h'' \cdot x'_j.
\]

\begin{proposition}\label{ConcreteTypeDAction} Defining the coefficients $c'_{i,j}$ and $\tilde{c'}_{i,j;h'}$ as above, and $\tilde{\tilde{c}}_{hh';h''}$ as in Section~\ref{ModulePrelims}, the Type D structure operation $\delta$ on $\D(N)$ has a basis expansion given by
\begin{align*}
\delta(h \cdot x'_i) &= \sum_{j} c'_{i,j} h \cdot x'_j \\
&+ \sum_{j,h' \in \beta_{mult},h'' \in \beta} \tilde{c'}_{i,j;h'} \tilde{\tilde{c}}_{hh';h''} b_{*;h,h''} \otimes (h'' \cdot x'_j) \\
&+ \sum_{h'' \in \beta \textrm{ s.t. } m(b^*_{*;m(h),m(h'')}) \textrm{ exists})=} (-1)^{\deg_h x'_i} m(b^*_{*;m(h),m(h'')}) \otimes (h'' \cdot x'_i),
\end{align*}
where $b_{*;h,h''}$ means $b_{\gamma;h,h''}$ or $b_{C;h,h''}$ as appropriate.
\end{proposition}

\begin{proof}
$\delta$ is defined as the mirror of the Type D structure operation $\delta^{\boxtimes}$ on $\widehat{A}(m(N)) \boxtimes {^{\B \astrosun m(\B)^!}}K^{m(\B \astrosun m(\B)^!)^{op}}$, which in turn is defined as
\[
\delta^{\boxtimes} = m_1 + \xi \circ (m_2 \otimes \id) \circ (\id \otimes \delta_{DD}).
\]
Here, $m_1$ is the differential on $\widehat{A}(m(N))$, and $\delta_{DD}$ is the Type DD operation on ${^{\B \astrosun m(\B)^!}}K^{m(\B \astrosun m(\B)^!)^{op}}$. An arbitrary generator of $\widehat{A}(m(N))$ may be written as $m(x'_i) \cdot m(h)$, and we have
\[
m_1(m(x'_i) \cdot m(h)) = \sum_{j} c'_{i,j} m(x'_j) \cdot m(h),
\]
as well as
\[
m_2(m(x'_i) \cdot m(h), m(b^*_{*;h,h''})) = \sum_{j,h' \in \beta_{mult}} \tilde{c'}_{i,j;h} \tilde{\tilde{c}}_{hh';h''} m(x'_j) \cdot m(h'')
\]
and
\[
m_2(m(x'_i) \cdot m(h), b_{*;m(h),m(h'')}) = m(x'_i) \cdot m(h'').
\]
Here, $b_{*;m(h),m(h'')}$ stands for either $b_{\gamma;m(h),m(h'')}$ or $b_{C;m(h),m(h'')}$ as appropriate, and similarly for $m(b^*_{*;h,h''})$. Also note that $\tilde{\tilde{c}}_{hh';h''} = \tilde{\tilde{c}}_{m(h')m(h);m(h'')}$.

Thus,
\begin{align*}
\delta^{\boxtimes}&(m(x'_i) \cdot m(h)) = \sum_{j} c'_{i,j} m(x'_j) \cdot m(h) \\
&+ \sum_{j,h' \in \beta_{mult}, h'' \in \beta} \tilde{c'}_{i,j;h'} \tilde{\tilde{c}}_{hh';h''} m(b_{*;h,h''}) \otimes (m(x'_j) \cdot m(h'')) \\
&+ \sum_{h'' \in \beta, m(b^*_{*;m(h),m(h'')}) \textrm{ exists}} (-1)^{\deg_h x'_i} m(m(b^*_{*;m(h),m(h'')})) \otimes (m(x'_i) \cdot m(h'')).
\end{align*}
Taking the mirror of this formula, we get
\begin{align*}
\delta(h \cdot x'_i) &= \sum_{j} c'_{i,j} h \cdot x'_j \\
&+ \sum_{j,h' \in \beta_{mult}, h'' \in \beta} \tilde{c'}_{i,j;h'} \tilde{\tilde{c}}_{hh';h''} (b_{*;h,h''}) \otimes (h'' \cdot x'_j) \\
&+ \sum_{h'' \in \beta \textrm{ s.t. } m(b^*_{*;m(h),m(h'')}) \textrm{ exists}} (-1)^{\deg_h x'_i} m(b^*_{*;m(h),m(h'')}) \otimes (h'' \cdot x'_i).
\end{align*}

\end{proof}

When $N$ is Khovanov's complex $[T]^{Kh}$ for a tangle diagram $T$ in $\R_{\geq 0} \times \R$, the Type D structure $\D([T]^{Kh})$ over $\B \Gamma_n$ is the same as Roberts' Type D structure from \cite{RtypeD}:

\begin{proposition}\label{RobertsKhovanovTypeD}
$\D([T]^{Kh})$, as defined in Definition~\ref{TypeDOverBGnDef}, agrees with the Type D structure over $\B \Gamma_n$ which Roberts associates to $T$.
\end{proposition}

\begin{proof}
Roberts' Type D structure is defined as a bigraded $\Z$-module in Definition 31 of \cite{RtypeD}. As such, it agrees with $\D([T]^{Kh})$, and the action of the idempotent ring $\I_{\beta}$ is the same on both; Roberts defines the action of the idempotent ring at the end of Section 3.2 of \cite{RtypeD}. 

Lastly, the Type D operation $\delta$ on $\D([T]^{Kh})$, has an explicit form given in Proposition~\ref{ConcreteTypeDAction} above. The coefficients $c'_{i,j}$ and $\tilde{c'}_{i,j;h'}$ have the same description in terms of bridges and generators as the coefficients $c_{i,j}$ and $c_{i,j;h'}$ described in the proof of Proposition~\ref{RobertsKhovanovTypeA}. Thus, $\delta$ agrees with Roberts' Type D operation defined at the beginning of Section 5 of \cite{RtypeD}.
\end{proof}

\subsection{Pairing}\label{RobertsPairingSect}

Let $M$ be a complex of graded projective right $H^n$-modules and let $N$ be a complex of graded projective left $H^n$-modules, satisfying the algebraic condition $C_{module}$ of Definition~\ref{CModuleAlgCondition} and Definition~\ref{CModuleForLeftMods}. The natural way to pair $M$ and $N$ and get a chain complex over $\Z$ is to take the tensor product $M \otimes_{H^n} N$. However, we could also use Definition~\ref{FullTypeAStructureDefn} to construct a Type A structure $\widehat{A}(M)$ and use Definition~\ref{TypeDOverBomBDef} to construct a Type D structure $\widehat{D}(N)$, both over $\B \astrosun m(\B)^!$, and then take their box tensor product. This produces the same chain complex as $M \otimes_{H^n} N$, after a reversal of the intrinsic grading:

\begin{proposition}\label{TensorXBoxAgree}
As differential bigraded $\Z$-modules, $\widehat{A}(M) \boxtimes_{\B \astrosun m(\B)^!} \widehat{D}(N)$ is isomorphic to the module obtained from $M \otimes_{H^n} N$ by multiplying all intrinsic gradings on $M \otimes_{H^n} N$ by $-1$.
\end{proposition}

\begin{proof}
Let $\{x_i \cdot h: e(x_i) = e_L(h)\}$ and $\{h \cdot x'_i: e(x'_i) = e_R(h) \}$ be $\Z$-bases for $M$ and $N$ respectively, with $h \in H^n$ and each $x_i$ and $x'_i$ bigrading-homogeneous and having a unique idempotent $e(x_i)$ or $e(x'_i)$ in the idempotent ring of $H^n$. Then a $\Z$-basis for $M \otimes_{H^n} N$ is $\{ x_i \cdot h \cdot x'_j\}$ (we will suppress the idempotent conditions). 

Write the differentials on $M$ and $N$ as $d_M$ and $d_N$. As an element of $M \otimes_{H^n} N$, the differential of $x_i \cdot h \cdot x'_j$ is
\[
\partial^{\otimes} (x_i \cdot h \cdot x'_j) = (-1)^{\deg_h x'_j} d_M(x_i) \cdot h \cdot x'_j + x_i \cdot h \cdot d_N(x'_j).
\]
If we expand out $d_M$ as in Section~\ref{ModulePrelims} and $d_N$ as in the discussion preceding Proposition~\ref{ConcreteTypeDAction}, we may write this as
\begin{equation}\label{PartialTensorExpansion}
\begin{aligned}
(-1)^{\deg_h x'_j} &\bigg( \sum_{k} c_{i,k} x_k \cdot h \cdot x'_j + \sum_{k,h' \in \beta_{mult},h'' \in \beta} \tilde{c}_{i,k;h'} \tilde{\tilde{c}}_{h'h;h''}(x_k \cdot h'' \cdot x'_j) \bigg) \\
&+ \sum_{l} c'_{j,l} x_i \cdot h \cdot x'_l + \sum_{l,h' \in \beta_{mult}, h'' \in \beta} \tilde{c'}_{j,l;h'} \tilde{\tilde{c}}_{hh';h''} (x_i \cdot h'' \cdot x'_l).
\end{aligned}
\end{equation}

Now, as a bigraded $\Z$-module, $\widehat{A}(M) \boxtimes_{\B \astrosun m(\B)^!} \widehat{D}(N)$ is defined to be $\widehat{A}(M) \otimes_{\I_{\beta}} \widehat{D}(N)$. A $\Z$-basis for $\widehat{A}(M)$ (respectively $\widehat{D}(N)$) is also given by $\{x_i \cdot h\}$ (respectively $\{h \cdot x'_i\}$). A generator $x_i \cdot h$ of $\widehat{A}(M)$ has the same idempotent in $\I_{\beta}$ as another generator $h' \cdot x'_j$ of $\widehat{D}(N)$ if and only if $h = h'$. 

Thus, $\widehat{A}(M) \otimes_{\I_{\beta}} \widehat{D}(N)$ has a $\Z$-basis consisting of all elements $x_i \cdot h \cdot x'_j$, the same basis as for $M \otimes_{H^n} N$. The bigradings agree on these two $\Z$-modules after negating the intrinsic gradings on $M \otimes_{H^n} N$: note that for the intrinsic grading on $\widehat{A}(M) \otimes_{\I_{\beta}} \widehat{D}(N)$, the grading of $h$ in $x_i \cdot h \cdot x'_j$ is counted twice with coefficient $-\frac{1}{2}$, while for the intrinsic grading on $M \otimes_{H^n} N$, the grading of $h$ in $x_i \cdot h \cdot x'_j$ is counted once with coefficient $1$. This explains the factor of $\frac{1}{2}$ in Definition~\ref{WeirdHomSpaceGradingDef}.

It remains to show that the differential $\partial^{\otimes}$ on $M \otimes_{H^n} N$ agrees with the differential $\partial^{\boxtimes}$ on $\widehat{A}(M) \boxtimes_{\B \astrosun m(\B)^!} \widehat{D}(N)$. We will use $m_1$ and $m_2$ to denote the differential and algebra action on $\widehat{A}(M)$ and $\delta$ to denote the Type D operation on $\D(N)$. Applying $\partial^{\boxtimes}$ to a generator $(x_i \cdot h) \otimes (h \cdot x'_j)$, we get
\[
(-1)^{\deg_h(h \cdot x'_j)} (m_1(x_i \cdot h)) \otimes (h \cdot x'_j) + (m_2 \otimes \id) \circ ((x_i \cdot h) \otimes \delta(h \cdot x'_j)).
\]
Because $H^n$ is concentrated in homological degree zero, $\deg_h(h \cdot x'_j) = \deg_h(x'_j)$. Thus, the first term of $\partial^{\boxtimes}((x_i \cdot h) \otimes (h \cdot x'_j))$ is
\[
(-1)^{\deg_h x'_j} \sum_{k} c_{i,k} x_k \cdot h \cdot x'_j,
\]
which agrees with the first term of Equation~\ref{PartialTensorExpansion} for $\partial^{\otimes}(x_i \cdot h \cdot x'_j)$.

By Proposition~\ref{ConcreteTypeDAction}, the other term of $\partial^{\boxtimes}((x_i \cdot h) \otimes (h \cdot x'_j))$ can be expanded out as
\begin{align*}
(m_2 &\otimes \id) \circ \bigg( (x_i \cdot h) \otimes \bigg( \sum_{l} c'_{j,l} h \cdot x'_l \\
&+ \sum_{l,h' \in \beta_{mult},h'' \in \beta} \tilde{c'}_{j,l;h'} \tilde{\tilde{c}}_{hh';h''} b_{*;h,h''} \otimes (h'' \cdot x'_l) \\
&+ \sum_{l,h'' \in \beta \textrm{ s.t. } m(b^*_{*;m(h),m(h'')}) \textrm { exists}} (-1)^{\deg_h x'_j} m(b^*_{*;m(h),m(h'')}) \otimes (h'' \cdot x'_j) \bigg) \bigg),
\end{align*}
where $b_{*;h,h''}$ denotes either $b_{\gamma;h,h''}$ or $b_{C;h,h''}$ and $m(b^*_{*;m(h),m(h'')})$ denotes either $m(b^*_{\gamma;m(h),m(h'')})$ or $m(b^*_{C;m(h),m(h'')})$. This expansion gives us three remaining terms of $\partial^{\boxtimes}((x_i \cdot h) \otimes (h \cdot x'_j))$. The first of these is
\[
\sum_{l} c'_{j,l} (x_i \cdot h \cdot x'_l),
\]
which agrees with the third term of equation~\ref{PartialTensorExpansion}. The second is
\[
\sum_{l,h' \in \beta_{mult}, h'' \in \beta} \tilde{c'}_{j,l;h'} \tilde{\tilde{c}}_{hh';h''} (x_i \cdot h'' \cdot x'_j),
\]
which agrees with the fourth term of equation~\ref{PartialTensorExpansion}. Finally, the remaining term of $\partial^{\boxtimes}((x_i \cdot h) \otimes (h \cdot x'_j))$ is
\[
(-1)^{\deg_h x'_j} \sum_{k,h' \in \beta_{mult},h'' \in \beta} \tilde{c}_{i,k;h'} \tilde{\tilde{c}}_{h'h;h''} (x_k \cdot h'' \cdot x'_j),
\]
which agrees with the second term of equation~\ref{PartialTensorExpansion}. Thus, after reversing the intrinsic gradings on $M \otimes_{H^n} N$, we conclude that $M \otimes_{H^n} N$ is isomorphic to $\widehat{A}(M) \otimes_{\B \astrosun m(\B)^!} \widehat{D}(N)$ as differential bigraded $\Z$-modules.
\end{proof}

\begin{remark} The negation of the intrinsic gradings on $M \otimes_{H^n} N$ is done for the same reason as in Remark~\ref{FirstGradingRevRem}.
\end{remark}

\begin{proposition}\label{XBoxQuotientDescend} Let $\B$ be a differential bigraded algebra, and let $J$ be a bigrading homogeneous ideal of $\B$ which is preserved by the differential on $\B$. Let $\D$ be a Type D structure over $\B$, and let $\widehat{A}$ be a differential bigraded right $\B$-module which descends to a module over $\B / J$. By Proposition~\ref{DandDDQuotient}, $\D$ automatically descends to a Type D structure over $\B / J$, and we have
\[
\widehat{A} \boxtimes_{\B} \D \cong \widehat{A} \boxtimes_{\B / J} \D.
\]
\end{proposition}

\begin{proof}
This follows immediately from Definition~\ref{XBoxWithSigns}.
\end{proof}

\begin{corollary}
Let $T_1$ and $T_2$ be oriented tangle diagrams in $\R_{\geq 0} \times \R$ and $\R_{\leq 0} \times \R$ respectively, with orderings chosen of the crossings of $T_1$ and $T_2$. Assume that $T_1$ and $T_2$ have consistent orientations, so that their horizontal concatenation is an oriented link diagram $L$ in $\R^2$. Order the crossings of $L$ such that those of $T_1$ come before those of $T_2$. Then
\[
CKh(L) \cong \widehat{A}([T_2]^{Kh}) \boxtimes_{\B \astrosun m(\B)^!} \D([T_1]^{Kh}) \cong \widehat{A}([T_2]^{Kh}) \boxtimes_{\B \Gamma_n} \D([T_1]^{Kh}).
\]
\end{corollary}

\begin{proof} This is a corollary of Proposition~\ref{TensorXBoxAgree}, Proposition~\ref{XBoxQuotientDescend}, and Khovanov's results from \cite{KhovFunctor}.
\end{proof}
Identifying $\widehat{A}([T_2]^{Kh})$ with Roberts' Type A structure over $\B \Gamma_n$ by Proposition~\ref{RobertsKhovanovTypeA}, and identifying $\widehat{D}([T_1]^{Kh})$ with Roberts' Type A structure over $\B \Gamma_n$ by Proposition~\ref{RobertsKhovanovTypeD}, we obtain an alternate proof of Proposition 36 of Roberts~\cite{RtypeD}.

\subsection{Equivalences of Type A structures}\label{TypeAEquivSect}

We start by defining $\mathcal{A}_{\infty}$-morphisms. The following definition is general enough for our purposes, although it is not the most general definition possible. A more general definition is given in Definition 26 of Roberts \cite{RtypeA}; our sign conventions are the same as Roberts'.

\begin{definition}\label{TypeAMorphismDef} Let $\B$ be a differential bigraded algebra with idempotent ring $\I$. Let $\widehat{A}$ and $\widehat{A'}$ be differential bigraded right modules over $\B$. An $\mathcal{A}_{\infty}$-morphism $F$ from $\widehat{A}$ to $\widehat{A'}$ is a collection 
\[
F_n: \widehat{A} \otimes_{\I} \B^{\otimes (n-1)} \to \widehat{A'}[0,n-1]
\]
of bigrading-preserving $\I$-linear maps satisfying the compatibility condition
\begin{align*}
m'_1 \circ F_n &+ (-1)^n m'_2 \circ (F_{n-1} \otimes \left|\id\right|^n) \\
&= F_{n-1} \circ (m_2 \otimes \id^{\otimes(n-2)}) + (-1)^{n+1} F_n \circ (m_1 \otimes \left|\id\right|^{\otimes(n-1)}) \\
&+ (-1)^{n+1} \sum_{k=1}^{n-1} F_n \circ (\id^{\otimes k} \otimes \mu_1 \otimes \left|\id\right|^{\otimes(n-k-1)}) \\
&+ \sum_{k=1}^{n-2} (-1)^k F_{n-1} \circ (\id^{\otimes k} \otimes \mu_2 \otimes \left|\id\right|^{\otimes(n-k-2)})
\end{align*}
for all $n \geq 1$. Recall that $\left|\id\right|^n$ and $\left|\id\right|^{\otimes n}$ mean different things; see Section~\ref{DiffGrAlgModSect}.
\end{definition}

\begin{example}\label{F1F2NonzeroExample}
For an $\mathcal{A}_{\infty}$ morphism $F$ with only $F_1$ and $F_2$ nonzero, the condition of Definition~\ref{TypeAMorphismDef} is nontrivial only for $n = 1$, $2$, and $3$. The $n = 1$ condition is
\[
m'_1 \circ F_1 = F_1 \circ m_1,
\]
the $n = 2$ condition is
\[
m'_1 \circ F_2 + m'_2 \circ (F_1 \otimes \id) = F_1 \circ m_2 - F_2 \circ (m_1 \otimes \left|\id\right|) - F_2 \circ (\id \otimes \mu_1),
\]
and the $n = 3$ condition is
\[
-m'_2 \circ (F_2 \otimes \left|\id\right|) = F_2 \circ (m_2 \otimes \id) - F_2 \circ (\id \otimes \mu_2).
\]
\end{example}

Let $(M,d_M)$ and $(M',d_{M'})$ be two chain complexes of graded projective right $H^n$-modules satisfying the algebraic condition $C_{module}$ of Definition~\ref{CModuleAlgCondition}. Let $f: M \to M'$ be a bigrading-preserving $H^n$-linear map such that $d_{M'} f = f d_M$; as shorthand, we will say ``let $f$ be a chain map from $M$ to $M'$.'' We first show that certain chain maps $f$ induce $\mathcal{A}_{\infty}$-morphisms of Type A structures $\widehat{A}(M) \to \widehat{A}(M')$ over $\B \astrosun m(\B)^!$. Let $\{x_i \cdot h: e(x_i) = e_L(h)\}$ be a $\Z$-basis for $M$, such that each $x_i$ is bigrading-homogeneous with a unique idempotent, and let $\{x'_i \cdot h\}$ be such a basis for $M'$ (we will suppress the idempotent conditions). We may expand $f(x_i)$ in the basis for $M'$:
\[
f(x_i) = \sum_j f_{i,j} x'_j + \sum_{j,h' \in \beta, \deg h' \neq 0} \tilde{f}_{i,j;h'} x'_j \cdot h'.
\]
The algebraic condition we will assume of $f$ is the following:
\begin{definition}\label{CMorphismDef}
The chain map $f$ satisfies the algebraic condition $C_{morphism}$ if $\tilde{f}_{i,j;h'}$ is only nonzero when $h' \in \beta_{mult}$. 
\end{definition} 
For a chain map $f$ satisfying $C_{morphism}$, we may write
\[
f(x_i) = \sum_j f_{i,j} x'_j + \sum_{j,h' \in \beta_{mult}} \tilde{f}_{i,j;h'} x'_j \cdot h'.
\]
Thus, a basis expansion for $f(x_i \cdot h)$ is
\[
f(x_i \cdot h) = \sum_j f_{i,j} x'_j \cdot h + \sum_{j,h' \in \beta_{mult}, h'' \in \beta} \tilde{f}_{i,j;h'} \tilde{\tilde{c}}_{h'h;h''} x'_j \cdot h''.
\]

Since $M$ and $M'$ satisfy $C_{module}$, we also have
\[
d_M(x_i \cdot h) = \sum_j c_{i,j} x_j \cdot h + \sum_{j,h' \in \beta_{mult}, h'' \in \beta} \tilde{c}_{i,j;h'} \tilde{\tilde{c}}_{h'h;h''} x_j \cdot h''
\]
and
\[
d_{M'}(x'_i \cdot h) = \sum_j c'_{i,j} x'_j \cdot h + \sum_{j,h' \in \beta_{mult}, h'' \in \beta} \tilde{c'}_{i,j;h'} \tilde{\tilde{c}}_{h'h;h''} x'_j \cdot h''.
\]

\begin{proposition}\label{HnChainMapStructure}
Suppose the chain map $f$ satisfies $C_{morphism}$. The equation $d_{M'} f = f d_M$ gives us the following equations in the coefficients $f_{i,j}$, $\tilde{f}_{i,j;h'}$, $c_{i,j}$, $\tilde{c}_{i,j;h'}$, $c'_{i,j}$, and $\tilde{c'}_{i,j,h'}$:

\begin{enumerate}
\item\label{HnChainMapRels1}
For all generators $x_i$ of $M$ and $x'_k$ of $M'$,
\[
\sum_j f_{i,j} c'_{j,k} = \sum_j c_{i,j} f_{j,k}.
\]

\item\label{HnChainMapRels2}
For all generators $x_i \cdot h$ of $M$ and $x'_k \cdot h''$ of $M'$,
\begin{align*}
&\sum_{j,h' \in \beta_{\gamma}} \tilde{f}_{i,j;h'} \tilde{\tilde{c}}_{h'h;h''} c'_{j,k} + \sum_{j,h' \in \beta_{\gamma}} f_{i,j} \tilde{c'}_{j,k;h'} \tilde{\tilde{c}}_{h'h;h''} \\
&= \sum_{j,h' \in \beta_{\gamma}} c_{i,j} \tilde{f}_{j,k;h'} \tilde{\tilde{c}}_{h'h;h''} + \sum_{j,h' \in \beta_{\gamma}} \tilde{c}_{i,j;h'} \tilde{\tilde{c}}_{h'h;h''} f_{j,k}.
\end{align*}

\item\label{HnChainMapRels3}
For all generators $x_i \cdot h$ of $M$ and $x'_k \cdot h''$ of $M'$,
\begin{align*}
&\sum_{j,h' \in \beta_{\alpha}} \tilde{f}_{i,j;h'} \tilde{\tilde{c}}_{h'h;h''} c'_{j,k} + \sum_{j,h' \in \beta_{\alpha}} f_{i,j} \tilde{c'}_{j,k;h'} \tilde{\tilde{c}}_{h'h;h''} \\
&\qquad + \sum_{j,h' \in \beta_{\gamma}, h''' \in \beta_{\gamma}, h'''' \in \beta} \tilde{f}_{i,j;h'} \tilde{\tilde{c}}_{h'h;h''''} \tilde{c'}_{j,k;h'''} \tilde{\tilde{c}}_{h'''h'''';h''} \\
&= \sum_{j,h' \in \beta_{\alpha}} c_{i,j} \tilde{f}_{j,k;h'} \tilde{\tilde{c}}_{h'h;h''} + \sum_{j,h' \in \beta_{\alpha}} \tilde{c}_{i,j;h'} \tilde{\tilde{c}}_{h'h;h''} f_{j,k} \\
&\qquad + \sum_{j,h' \in \beta_{\gamma}, h''' \in \beta_{\gamma}, h'''' \in \beta} \tilde{c}_{i,j;h'} \tilde{\tilde{c}}_{h'h;h''''} \tilde{f}_{j,k;h'''} \tilde{\tilde{c}}_{h'''h'''';h''}.
\end{align*}

\item\label{HnChainMapRels4}
For all generators $x_i \cdot h$ of $M$ and $x'_k \cdot h''$ of $M'$,
\begin{align*}
&\sum_{j,h' \in \beta_{\gamma}, h''' \in \beta_{\alpha}, h'''' \in \beta} \tilde{f}_{i,j;h'} \tilde{\tilde{c}}_{h'h;h''''} \tilde{c'}_{j,k;h'''} \tilde{\tilde{c}}_{h'''h'''';h''} \\
&\qquad + \sum_{j,h' \in \beta_{\alpha}, h''' \in \beta_{\gamma}, h'''' \in \beta} \tilde{f}_{i,j;h'} \tilde{\tilde{c}}_{h'h;h''''} \tilde{c'}_{j,k;h'''} \tilde{\tilde{c}}_{h'''h'''';h''} \\
&= \sum_{j,h' \in \beta_{\gamma}, h''' \in \beta_{\alpha}, h'''' \in \beta} \tilde{c}_{i,j;h'} \tilde{\tilde{c}}_{h'h;h''''} \tilde{f}_{j,k;h'''} \tilde{\tilde{c}}_{h'''h'''';h''} \\
&\qquad + \sum_{j,h' \in \beta_{\alpha}, h''' \in \beta_{\gamma}, h'''' \in \beta} \tilde{c}_{i,j;h'} \tilde{\tilde{c}}_{h'h;h''''} \tilde{f}_{j,k;h'''} \tilde{\tilde{c}}_{h'''h'''';h''}.
\end{align*}

\item\label{HnChainMapRels5}
For all generators $x_i \cdot h$ of $M$ and $x'_k \cdot h''$ of $M'$,
\begin{align*}
&\sum_{j,h' \in \beta_{\alpha}, h''' \in \beta_{\alpha}, h'''' \in \beta} \tilde{f}_{i,j;h'} \tilde{\tilde{c}}_{h'h;h''''} \tilde{c'}_{j,k;h'''} \tilde{\tilde{c}}_{h'''h'''';h''} \\
&= \sum_{j,h' \in \beta_{\alpha}, h''' \in \beta_{\alpha}, h'''' \in \beta} \tilde{c}_{i,j;h'} \tilde{\tilde{c}}_{h'h;h''''} \tilde{f}_{j,k;h'''} \tilde{\tilde{c}}_{h'''h'''';h''}.
\end{align*}

\end{enumerate}
\end{proposition}

\begin{proof} The proof is very similar to that of Proposition~\ref{HnProjModStructure} and will be omitted. Note that explicitly writing the degree conditions in the sums is unnecessary, since the relevant products of coefficients are always zero unless the degree conditions are satisfied. In Proposition~\ref{HnProjModStructure}, we chose to write out the degree conditions for clarity.
\end{proof}

\begin{definition}\label{InducedAInftyMorphismDef} Suppose $(M,d_M)$ and $(M',d_{M'})$ satisfy $C_{module}$ and $f: M \to M'$ is a chain map satisfying $C_{morphism}$. Define the first component $\widehat{A}(f)_1$ of an $\mathcal{A}_{\infty}$-morphism $\widehat{A}(f): \widehat{A}(M) \to \widehat{A}(M')$ of Type A structures over $\B \astrosun m(\B)^!$ by
\[
\widehat{A}(f)_1 (x_i \cdot h) = \sum_j f_{i,j} x'_j \cdot h.
\]
The map $\widehat{A}(f)_1: \widehat{A}(M) \to \widehat{A}(M')$ respects the right actions of the idempotent ring $\I_{\beta}$, and it is bigrading-preserving because $f$ is. 

If $\widehat{A}(f)_1$ were the only nonzero component of $\widehat{A}(f)$, then $\widehat{A}(f)$ would be an ordinary chain map between differential bigraded $\B \astrosun m(\B)^!$-modules. However, $\widehat{A}(f)_2$ will also be nonzero in general; thus, we must deal with genuine higher $\mathcal{A}_{\infty}$ terms when working with these morphisms. The component 
\[
\widehat{A}(f)_2: \widehat{A}(M) \otimes_{\I_{\beta}} \B \astrosun m(\B)^! \to \widehat{A}(M)
\]
of $\widehat{A}(f)$ is defined on the generators $x_i \cdot h$ of $\widehat{A}(M)$ and $m(b^*_{*;m(h),m(h'')})$ of $\B \astrosun m(\B)^!$ by
\[
\widehat{A}(f)_2 (x_i \cdot h, m(b^*_{*;m(h),m(h'')})) = \sum_{j,h' \in \beta_{mult}} \tilde{f}_{i,j;h'} \tilde{\tilde{c}}_{h'h;h''} x'_j \cdot h'', 
\]
where $m(b^*_{*;m(h),m(h'')})$ denotes $m(b^*_{\gamma;m(h),m(h'')})$ or $m(b^*_{C;m(h),m(h'')})$ as appropriate. 

Suppose the algebra input $m(b^*_{*;m(h),m(h'')})$ is equal to $m(b^*_{\gamma;m(h),m(h'')})$. Let $h = (W(a)b,\sigma)$ and $h'' = (W(a')b,\sigma')$; then $\tilde{\tilde{c}}_{h'h;h''}$ is only nonzero for one value of $h'$, namely $h' = (W(a')a,\textrm{ all } +)$. For this value of $h'$, $\tilde{\tilde{c}}_{h'h;h''}$ is $1$. Thus,
\begin{equation}\label{TopAf2SimpleEqn}
\widehat{A}(f)_2 (x_i \cdot h, m(b^*_{\gamma;m(h),m(h'')})) = \sum_j \tilde{f}_{i,j;h'} x'_j \cdot h''.
\end{equation}

Now suppose the algebra input is equal to $m(b^*_{C;m(h),m(h'')})$. As before, write $h$ as $h = (W(a)b,\sigma)$. In this case, $\tilde{\tilde{c}}_{h'h;h''}$ will be nonzero for any $h'_{\alpha}$ which equals $(W(a)a,\textrm{ minus on }\alpha)$, where $h = (W(a)b,\sigma)$ and $\alpha$ is any arc of $a$ which is part of the circle $C$ in $W(a)a$. For $h'$ equal to one of the $h'_{\alpha}$, we have $\tilde{\tilde{c}}_{h'h;h''} = 1$, and for all other $h'$, $\tilde{\tilde{c}}_{h'h;h''}$ is zero. Thus,
\begin{equation}\label{BottomAf2SimpleEqn}
\widehat{A}(f)_2 (x_i \cdot h, m(b^*_{C;m(h),m(h'')})) = \sum_j \sum_{\textrm{left arcs } \alpha \textrm{ of }C} \tilde{f}_{i,j;h'_{\alpha}} x'_j \cdot h'',
\end{equation}

Any action of the form $\widehat{A}(f)_2(x_i \cdot h, b_{*;h;h''})$ is defined to be zero. Writing $\B \astrosun m(\B)^!$ as $T(V_{full}) / J_{full}$ as in Section~\ref{FullAlgebraSection}, the above formulas define a map
\[
\widehat{A}(f)_2: \widehat{A}(M) \otimes_{\I_{\beta}} V_{full} \to \widehat{A}(M').
\]
We can extend to a map
\[
\widehat{A}(f)_2: \widehat{A}(M) \otimes_{\I_{\beta}} T(V_{full}) \to \widehat{A}(M')
\]
which is defined as the sum, over all $n \geq 2$, of the maps
\begin{align*}
\sum_{k=1}^{n-1} m'_2 \circ (m'_2 \otimes \id) \circ \cdots &\circ (m'_2 \otimes \id^{\otimes(k-2)}) \\ 
&\circ (\widehat{A}(f)_2 \otimes \left|\id\right|^{\otimes(k-1)}) \\
&\circ (m_2 \otimes \id^{\otimes k}) \circ \cdots \circ (m_2 \otimes \id^{\otimes(n-2)})
\end{align*}
from $\widehat{A}(M) \otimes (V_{full})^{\otimes (n-1)}$ to $\widehat{A}(M')$. In Proposition~\ref{AInfMorphWellDef}, we show that $\widehat{A}(f)_2$ descends to a map
\[
\widehat{A}(f)_2: \widehat{A}(M) \otimes_{\I_{\beta}} \B \astrosun m(\B)^! \to \widehat{A}(M');
\]
in Proposition~\ref{AInftyRelsSatisfied} we verify that $\widehat{A}(f)_1$ and $\widehat{A}(f)_2$ together satisfy the conditions to form an $\mathcal{A}_{\infty}$ morphism $\widehat{A}(f)$.

\end{definition}

\begin{example}\label{LowSummandsAf2Ex}
The $n = 2$ summand of $\widehat{A}(f)_2: \widehat{A}(M) \otimes_{\I_{\beta}} T(V_{full}) \to \widehat{A}(M')$ is simply $\widehat{A}(f)_2$, the map from $\widehat{A}(M) \otimes_{\I_{\beta}} V_{full}$ to $\widehat{A}(M')$ defined above. The $n = 3$ summand of $\widehat{A}(f)_2: \widehat{A}(M) \otimes_{\I_{\beta}} T(V_{full}) \to \widehat{A}(M')$, or in other words the definition of $\widehat{A}(f)_2$ when the algebra input is a quadratic monomial in the generators of $V_{full}$, is
\[
\widehat{A}(f)_2 \circ (m_2 \otimes \id) + m'_2 \circ (\widehat{A}(f)_2 \otimes \left|\id\right|),
\]
where in this expression $\widehat{A}(f)_2$ again denotes the map from $\widehat{A}(M) \otimes_{\I_{\beta}} V_{full}$ to $\widehat{A}(M')$.
\end{example}

\begin{proposition}\label{AInfMorphWellDef} Write $\B \astrosun m(\B)^!$ as $T(V_{full}) / J_{full}$, and let $r$ be any element of $J_{full}$. The map 
\[
\widehat{A}(f)_2(-,r): \widehat{A}(M) \to \widehat{A}(M')
\] 
is identically zero. Thus, we get a well-defined map
\[
\widehat{A}(f)_2: \widehat{A}(M) \otimes_{\I_{\beta}} \B \astrosun m(\B)^! \to \widehat{A}(M')
\]
which is linear with respect to the right actions of the idempotent ring $\I_{\beta}$ on $\widehat{A}(M)$ and $\widehat{A}(M')$, and which preserves the intrinsic grading and decreases the homological grading by one.
\end{proposition}

\begin{proof}
First, $\widehat{A}(f)_2$ decreases the homological grading by $1$, since $m(b^*_{*;m(h),m(h'')})$ carries homological grading $1$ and $f:M \to M'$ preserves homological grading. Also, $\widehat{A}(f)_2$ preserves the intrinsic grading. To see this, note that as elements of $M$ and $M'$, $x_i \cdot h$ and $x'_j \cdot h''$ must have the same intrinsic grading whenever $x'_j \cdot h''$ appears with nonzero coefficient in the basis expansion of $f(x_i \cdot h)$, because $f$ preserves intrinsic grading. Also, as elements of $H^n$, the intrinsic degree of $h''$ is either one or two greater than that of $h$. Since, in $\widehat{A}(M)$ and $\widehat{A}(M')$, the intrinsic degrees of $h$ and $h''$ are multiplied by $-\frac{1}{2}$ whereas the intrinsic degrees of $x_i$ and $x'_j$ are multiplied by $-1$, the element $x'_j \cdot h''$ of $\widehat{A}(M')$ should have intrinsic degree which is $\frac{1}{2}$ or $1$ greater than the intrinsic degree of $x_i \cdot h \in \widehat{A}(M)$. This extra $\frac{1}{2}$ or $1$ is exactly compensated by the intrinsic degree of $m(b^*_{*;m(h),m(h'')})$, which is $\frac{1}{2}$ for $m(b^*_{\gamma;m(h),m(h'')})$ and $1$ for $m(b^*_{C;m(h),m(h'')})$.

To show that $\widehat{A}(f)_2(-,r)$ is zero for any $r \in J_{full}$, note first that Definition~\ref{InducedAInftyMorphismDef} implies that if we have elements $r$ and $r'$ of $T(V_{full})$ such that $m_2(-,r) = 0$, $m'_2(-,r') = 0$, $\widehat{A}(f)_2(-,r) = 0$, and $\widehat{A}(f)_2(-,r') = 0$, then $\widehat{A}(f)_2(-,r \cdot r') = 0$ as well.

Thus, we only need to show that $\widehat{A}(f)_2(-,r)$ is zero for the multiplicative generators $r$ of $J_{full}$. These were defined to be the generators of $J_{\B} \cap (T^1(V_{\B}) \oplus T^2(V_{\B}))$, $J_{m(\B)^!} \cap (T^1(V_{m(\B)^!}) \oplus T^2(V_{m(\B)^!}))$, and $J_{extra}$. For generators in $J_{\B} \cap (T^1(V_{\B}) \oplus T^2(V_{\B}))$, there is nothing to show, since $\widehat{A}(f)_2(-,b)$ is zero for any $b \in V_{\B}$.

For the generators in $J_{m(\B)^!} \cap (T^1(V_{m(\B)^!}) \oplus T^2(V_{m(\B)^!}))$, the proof closely follows the proof of Proposition~\ref{mBBangActionWellDef}. Write $J_{m(\B)^!} \cap (T^1(V_{m(\B)^!}) \oplus T^2(V_{m(\B)^!}))$ as $m(I^{\perp})$. 

The generators $m(r^*)$ of $m(I^{\perp})$ have intrinsic degree either $1$, $\frac{3}{2}$, or $2$. For those $m(r^*)$ of intrinsic degree $2$, the equations in item \ref{HnChainMapRels5} of Proposition~\ref{HnChainMapStructure} above imply that $\widehat{A}(f)_2(x_i \cdot h, m(r^*)) = 0$ for any $x_i \cdot h$. For those $m(r^*)$ of intrinsic degree $\frac{3}{2}$, the equations in item \ref{HnChainMapRels4} of this proposition similarly imply that $\widehat{A}(f)_2(-,m(r^*))$ is zero.

The generators $m(r^*)$ of $m(I^{\perp})$ which have intrinsic degree $1$ are sums of either one, two, three, or four terms $m(b^*_{\gamma}) m(b^*_{\gamma'})$ with all coefficients $+1$. For a fixed $m(r^*)$, let $h \in \I_{\beta}$ denote its left idempotent and let $h'' \in \I_{\beta}$ denote its right idempotent. The element $h''$ of $\beta$ has degree $2$ more than $h$, as elements of $H^n$ with its intrinsic grading, and $h''$ differs from $h$ by two surgeries on its left crossingless matching. As in Proposition~\ref{mBBangActionWellDef}, the left crossingless matchings of $h$ and $h''$ are different. 

For any generators of $\widehat{A}(M)$ and $\widehat{A}(M')$ of the form $x_i \cdot h$ and $x'_k \cdot h''$, where $h$ and $h''$ are as above, the equations from item~\ref{HnChainMapRels3} of Proposition~\ref{HnChainMapStructure} become
\begin{align*}
&\sum_{j,h' \in \beta_{\gamma}, h''' \in \beta_{\gamma}, h'''' \in \beta} \tilde{f}_{i,j;h'} \tilde{\tilde{c}}_{h'h;h''''} \tilde{c'}_{j,k;h'''} \tilde{\tilde{c}}_{h'''h'''';h''} \\
&= \sum_{j,h' \in \beta_{\gamma}, h''' \in \beta_{\gamma}, h'''' \in \beta} \tilde{c}_{i,j;h'} \tilde{\tilde{c}}_{h'h;h''''} \tilde{f}_{j,k;h'''} \tilde{\tilde{c}}_{h'''h'''';h''};
\end{align*}
the terms involving $h' \in \beta_{\alpha}$ vanish for these choices of $h$ and $h''$. These equations imply that for all generators $m(r^*)$ of $m(I^{\perp})$ of intrinsic degree $1$, $\widehat{A}(f)_2(-,m(r^*))$ is zero.

Finally, the generators of $J_{extra}$ are listed in items \ref{JExtraRels1}-\ref{JExtraRels5} of Definition~\ref{JExtraDef}. If $r$ is one of these generators, the proof that the map $\widehat{A}(f)_2(-,r)$ is zero is similar to the proof of Proposition~\ref{FullMTwoActionWellDef}.

In more detail, consider a relation 
\[
r = b_{\gamma;h_1,h_2} m(b^*_{\eta';m(h_2),m(h_3)}) - m(b^*_{\eta;m(h_1),m(h'_2)}) b_{\gamma';h'_2,h_3}
\]
from item \ref{JExtraRels1} of Definition~\ref{JExtraDef}. Write $h_1 = (W(a_1)b_1,\sigma_1)$, and let $x_i \cdot h_1$ be a generator of $\widehat{A}(M)$. By Example~\ref{LowSummandsAf2Ex}, we have
\begin{equation}\label{ActionDefConsEqn}
\begin{aligned}
\widehat{A}(f)_2(x_i \cdot h_1, r) &= \widehat{A}(f)_2(m_2(x_i \cdot h_1, b_{\gamma;h_1,h_2}), m(b^*_{\eta';m(h_2),m(h_3)})) \\
&- m_2(\widehat{A}(f)_2(x_i \cdot h_1, m(b^*_{\eta;m(h_1),m(h'_2)})), b_{\gamma';h'_2,h_3}).
\end{aligned}
\end{equation}

Write $h_2$ as $(W(a_1)b_2,\sigma_2)$ and $h_3$ as $(W(a_2)b_2,\sigma_3)$. Let $h' = (W(a_2)a_1, \textrm{ all }+)$, an element of $\beta_{\gamma}$. For the first term in equation~\ref{ActionDefConsEqn} above, we first multiply $x_i \cdot h_1$ by $b_{\gamma;h_1,h_2}$ to get $x_i \cdot h_2$. Applying $\widehat{A}(f)_2(-,m(b^*_{\eta';m(h_2),m(h_3)}))$ to the element $x_i \cdot h_2$, by equation~\ref{TopAf2SimpleEqn} we get
\[
\sum_j \tilde{f}_{i,j;h'} x_j \cdot h_3.
\]
For the second term in equation~\ref{ActionDefConsEqn}, $\widehat{A}(f)_2(x_i \cdot h_1, m(b^*_{\eta;m(h_1),m(h'_2)}))$ equals
\[
\sum_j \tilde{f}_{i,j;h'} x_j \cdot h'_2
\]
by equation~\ref{TopAf2SimpleEqn} again, where $h'$ is still equal to $(W(a_2)a_1, \textrm{ all } +)$. If we multiply this result by $b_{\gamma';h'_2,h_3}$, we get
\[
\sum_j \tilde{f}_{i,j;h'} x_j \cdot h_3.
\]
Thus, if a generator $r$ of $J_{extra}$ comes from item \ref{JExtraRels1} of Definition~\ref{JExtraDef}, then the map $\widehat{A}(f)_2(-,r): \widehat{A}(M) \to \widehat{A}(M')$ is zero.

For generators of $J_{extra}$ from items \ref{JExtraRels2}-\ref{JExtraRels5} of Definition~\ref{JExtraDef}, the proof is analogous to that of Proposition~\ref{FullMTwoActionWellDef} in the same way as above. We will leave the remaining cases to the reader.
\end{proof}
 
\begin{proposition}\label{AInftyRelsSatisfied} $\widehat{A}(f)$ satisfies the $\mathcal{A}_{\infty}$-morphism compatibility conditions.
\end{proposition}

\begin{proof} 

Since $\widehat{A}(f)_n$ is zero for $n > 2$, it suffices to show that the $n = 1$, $n = 2$, and $n = 3$ conditions listed in Example~\ref{F1F2NonzeroExample} hold. For the $n = 1$ condition, we want to show that
\[
m'_1 (\widehat{A}(f)_1(x_i \cdot h)) = \widehat{A}(f)_1(m_1(x_i \cdot h)).
\]
The left side is 
\[
m'_1\bigg(\sum_{j} f_{i,j} x_j \cdot h\bigg) = \sum_{j,k} f_{i,j} c'_{j,k} x'_k \cdot h,
\]
while the right side is
\[
\widehat{A}(f)_1\bigg(\sum_j c_{i,j} x_j \cdot h\bigg) = \sum_{j,k} c_{i,j} f_{j,k} x'_k \cdot h.
\]
These are equal by item~\ref{HnChainMapRels1} of Proposition~\ref{HnChainMapStructure}.

We may write the $n = 3$ condition of Example~\ref{F1F2NonzeroExample} as 
\[
\widehat{A}(f)_2 \circ (id \otimes \mu_2) = \widehat{A}(f)_2 \circ (m_2 \otimes \id) + m'_2 \circ (\widehat{A}(f)_2 \otimes |\id|).
\]
In this form, it is clear from the definition of $\widehat{A}(f)_2$ in Definition~\ref{InducedAInftyMorphismDef} that this condition holds.

For the $n = 2$ condition, we want to show that
\begin{equation}\label{UsualN2FirstEqn}
\begin{aligned}
&m'_1 \circ \widehat{A}(f)_2 + m'_2 \circ (\widehat{A}(f)_1 \otimes \id) \\
&= \widehat{A}(f)_1 \circ m_2 - \widehat{A}(f)_2 \circ (m_1 \otimes \left|\id\right|) - \widehat{A}(f)_2 \circ (\id \otimes \mu_1)
\end{aligned}
\end{equation}
as maps from $\widehat{A}(M) \otimes_{\I_{\beta}} \B \astrosun m(\B)^!$ to $\widehat{A}(M')$.

We will first reduce to the case of proving the above equation when applied to terms of the form $(x_i \cdot h) \otimes b_{*;h,h''}$ or $(x_i \cdot h) \otimes m(b^*_{*;m(h),m(h'')})$, where $b_{*;h,h''}$ and $m(b^*_{*;m(h),m(h'')})$ are the multiplicative generators of $\B \astrosun m(\B)^!$:

\begin{claim*}
If $b_1$ and $b_2$ are two elements of $\B \astrosun m(\B)^!$ such that the $n = 2$ condition is satisfied both when the algebra input is $b_1$ and when it is $b_2$, then the $n = 2$ condition is also satisfied when the algebra input is $b_1 b_2$.
\end{claim*}

\begin{proof}[Proof of claim.]
We want to show that 
\begin{equation}\label{ProductN2ConsistencyRel}
\begin{aligned}
&m'_1 \circ \widehat{A}(f)_2 \circ (\id \otimes \mu_2) + m'_2 \circ (\widehat{A}(f)_1 \otimes \id) \circ (\id \otimes \mu_2) \\
&= \widehat{A}(f)_1 \circ m_2 \circ (\id \otimes \mu_2) - \widehat{A}(f)_2 \circ (m_1 \otimes \left|\id\right|) \circ(\id \otimes \mu_2) \\
&\qquad - \widehat{A}(f)_2 \circ (\id \otimes \mu_1) \circ (\id \otimes \mu_2)
\end{aligned}
\end{equation}
when the algebra input to these maps is $b_1 \otimes b_2$, assuming the usual $n = 2$ condition \ref{UsualN2FirstEqn} holds when the algebra input is $b_1$ or $b_2$. In the proof of this claim, the algebra input to all maps will be assumed to be $b_1 \otimes b_2$.

The left side of equation~\ref{ProductN2ConsistencyRel} can be rewritten as
\begin{equation}\label{LHSTerms1}
\begin{aligned}
&m'_1 \circ m'_2 \circ (\widehat{A}(f)_2 \otimes \left|\id\right|) \\
&+ m'_1 \circ \widehat{A}(f)_2 \circ (m_2 \otimes \id) \\
&+ m'_2 \circ (m'_2 \otimes \id) \circ (\widehat{A}(f)_1 \otimes \id \otimes \id),
\end{aligned}
\end{equation}
using the $n = 3$ consistency condition for $\widehat{A}(f)$ and the associativity of the action of $\B \astrosun m(\B)^!$ on $\widehat{A}(M')$. Call these terms $LHS_1$, $LHS_2$, and $LHS_3$. Now we may use the assumption that the $n = 2$ consistency condition holds when the algebra input is $b_1$ to write the third term $LHS_3$ of \ref{LHSTerms1} as 
\begin{align*}
&m'_2 \circ (\widehat{A}(f)_1 \otimes \id) \circ (m_2 \otimes \id) \\
&- m'_2 \circ (m'_1 \otimes \left|\id\right|) \circ (\widehat{A}(f)_2 \otimes \left|\id\right|) \\
&- m'_2 \circ (\widehat{A}(f)_2 \otimes \id) \circ (m_1 \otimes \left|\id\right| \otimes \id) \\
&- m'_2 \circ (\widehat{A}(f)_2 \otimes \id) \circ (\id \otimes \mu_1 \otimes \id).
\end{align*}
Call these four terms $LHS_{3a}$, $LHS_{3b}$, $LHS_{3c}$, and $LHS_{3d}$.

On the other hand, the right side of equation~\ref{ProductN2ConsistencyRel} can be rewritten as 
\begin{equation}\label{RightSideProductTerms}
\begin{aligned}
&\widehat{A}(f)_1 \circ m_2 \circ (m_2 \otimes \id) \\
&- m'_2 \circ (\widehat{A}(f)_2 \otimes \id) \circ (m_1 \otimes \left|\id\right| \otimes \id) \\
&- \widehat{A}(f)_2 \circ (m_2 \otimes \id) \circ (m_1 \otimes \left|\id\right| \otimes \left|\id\right|) \\
&- \widehat{A}(f)_2 \circ (\id \otimes \mu_2) \circ (\id \otimes \mu_1 \otimes \left|\id\right|) \\
&- \widehat{A}(f)_2 \circ (\id \otimes \mu_2) \circ (\id \otimes \id \otimes \mu_1),
\end{aligned}
\end{equation}
using the $n = 3$ consistency equation for $\widehat{A}(f)$, the associativity of the action of $\B \astrosun m(\B)^!$ on $\widehat{A}(M)$, and the Leibniz rule for the derivative $\mu_1$ on $\B \astrosun m(\B)^!$. Call these five terms $RHS_1$, $RHS_2$, $RHS_3$, $RHS_4$, and $RHS_5$. Using the $n = 3$ consistency equation again, we can rewrite the term $RHS_4$ as
\begin{align*}
&-m'_2 \circ (\widehat{A}(f)_2 \otimes \id) \circ (\id \otimes \mu_1 \otimes \id) \\
&- \widehat{A}(f)_2 \circ (m_2 \otimes \id) \circ (\id \otimes \mu_1 \otimes \left|\id\right|);
\end{align*}
call these two terms $RHS_{4a}$ and $RHS_{4b}$. Similarly, we can rewrite the term $RHS_5$ as
\begin{align*}
&-m'_2 \circ (\widehat{A}(f)_2 \otimes \left|\id\right|) \circ (\id \otimes \id \otimes \mu_1) \\
&- \widehat{A}(f)_2 \circ (m_2 \otimes \id) \circ (\id \otimes \id \otimes \mu_1);
\end{align*}
call these terms $RHS_{5a}$ and $RHS_{5b}$. Using the assumption that the $n = 2$ consistency condition holds when the algebra input is $b_2$, we can write the term $RHS_1$ as
\begin{align*}
&m'_2 \circ (\widehat{A}(f)_1 \otimes \id) \circ (m_2 \otimes \id) \\
&+ m'_1 \circ \widehat{A}(f)_2 \circ (m_2 \otimes \id) \\
&+ \widehat{A}(f)_2 \circ (m_1 \otimes \left|\id\right|) \circ (m_2 \otimes \id) \\
& + \widehat{A}(f)_2 \circ (m_2 \otimes \id) \circ (\id \otimes \id \otimes \mu_1).
\end{align*}
Call these four terms $RHS_{1a}$, $RHS_{1b}$, $RHS_{1c}$, and $RHS_{1d}$.

After rewriting the left and right sides of equation~\ref{ProductN2ConsistencyRel} in this way, we want to show that
\begin{align*}
&LHS_1 + LHS_2 + LHS_{3a} + LHS_{3b} + LHS_{3c} + LHS_{3d} \\
&= RHS_{1a} + RHS_{1b} + RHS_{1c} + RHS_{1d} + RHS_2 + RHS_3 \\
& \qquad + RHS_{4a} + RHS_{4b} + RHS_{5a} + RHS_{5b}.
\end{align*}
Several terms cancel: 
\begin{itemize}
\item $LHS_2 = RHS_{1b}$;
\item $LHS_{3a} = RHS_{1a}$;
\item $LHS_{3c} = RHS_2$;
\item $LHS_{3d} = RHS_{4a}$;
\item $RHS_{1d} + RHS_{5b} = 0$; and
\item $RHS_{1c} + RHS_3 + RHS_{4b} = 0$.
\end{itemize}
The final equality follows from the Leibniz rule for the differential $m_1$ on $\widehat{A}(M)$. Canceling corresponding terms between the sides, it remains to prove
\[
LHS_1 + LHS_{3b} =  RHS_{5a}.
\]

The Leibniz rule for the differential $m'_1$ on $\widehat{A}(M')$ lets us rewrite $LHS_1 + LHS_{3b}$ as
\[
m'_2 \circ (\id \otimes \mu_1) \circ (\widehat{A}(f)_2 \otimes \left|\id\right|)
\]
This term equals the remaining right-side term $RHS_{5a}$, because $\mu_1$ increases homological grading by one.
\end{proof}

Thus, we have reduced to showing that the $n = 2$ consistency condition \ref{UsualN2FirstEqn} holds for $\widehat{A}(f)$ when the algebra input is either $b_{\gamma;h,h''}$, $b_{C;h,h''}$, $m(b^*_{\gamma;m(h),m(h'')})$, or $m(b^*_{C;m(h),m(h'')})$. If the input $b_{*;h,h''}$ is $b_{\gamma;h,h''}$ or $b_{C;h,h''}$, the $\widehat{A}(f)_2$ terms in the $n = 2$ equation are zero, and we must show that
\begin{align*}
m'_2 \circ (\widehat{A}(f)_1 \otimes \id) = \widehat{A}(f)_1 \circ m_2
\end{align*}
for these algebra inputs. If $x_i \cdot h$ is a generator of $\widehat{A}(M)$, then
\begin{align*}
m'_2 \circ (\widehat{A}(f)_1 \otimes \id) (x_i \cdot h, b_{*;h,h''}) &= m'_2(\sum_j f_{i,j} x'_j \cdot h, b_{*;h,h''}) \\
&= \sum_j f_{i,j} x'_j \cdot h'',
\end{align*}
while
\begin{align*}
\widehat{A}(f)_1 \circ m_2 (x_i \cdot h, b_{*;h,h''}) &= \widehat{A}(f)_1 (x_i \cdot h'') \\
&= \sum_j f_{i,j} x'_j \cdot h'',
\end{align*}
and these are equal.

Now let the algebra input be a generator $m(b^*_{\gamma;m(h),m(h'')})$. The left side of the $n = 2$ condition with this algebra input is 
\[
\sum_{j,k,h' \in \beta_{\gamma}} \tilde{f}_{i,j;h'} \tilde{\tilde{c}}_{h'h;h''} c'_{j,k} (x'_k \cdot h'') + \sum_{j,k,h' \in \beta_{\gamma}} f_{i,j} \tilde{c'}_{j,k;h'} \tilde{\tilde{c}}_{h'h;h''} (x'_k \cdot h''),
\]
and the right side is 
\[
\sum_{j,k,h' \in \beta_{\gamma}} \tilde{c}_{i,j;h'} \tilde{\tilde{c}}_{h'h;h''} f_{j,k} (x'_k \cdot h'') + \sum_{j,k,h' \in \beta_{\gamma}} c_{i,j} \tilde{f}_{j,k;h'} \tilde{\tilde{c}}_{h'h;h''} (x'_k \cdot h'').
\]
These are equal by item~\ref{HnChainMapRels2} of Proposition~\ref{HnChainMapStructure}.

Finally, let the algebra input be a generator $m(b^*_{C;m(h),m(h'')})$. The left side of the $n = 2$ condition with this algebra input and module input $x_i \cdot h$ is 
\[
\sum_{j,k,h' \in \beta_{\alpha}} \tilde{f}_{i,j;h'} \tilde{\tilde{c}}_{h'h;h''} c'_{j,k} (x'_k \cdot h'') + \sum_{j,k,h' \in \beta_{\alpha}} f_{i,j} \tilde{c'}_{j,k;h'} \tilde{\tilde{c}}_{h'h;h''} (x'_k \cdot h'').
\]
As in Proposition~\ref{LeibnizRuleOnAM}, we have
\[
\mu_1(m(b^*_{C;m(h),m(h'')})) = - \sum_{h'''' \in \beta} m(b^*_{\gamma;m(h),m(h'''')}) m(b^*_{\gamma^{\dagger};m(h''''),m(h'')}),
\]
where the sum is implicitly over those $h''''$ such that generators $m(b^*_{\gamma;m(h),m(h'''')})$ and $m(b^*_{\gamma^{\dagger};m(h''''),m(h'')})$ exist. Thus,
\begin{align*}
&\widehat{A}(f)_2 \circ (\id \otimes \mu_1)(x_i \cdot h, m(b^*_{C;m(h),m(h'')})) \\
&= -\sum_{h''''} \widehat{A}(f)_2 (x_i \cdot h, m(b^*_{\gamma;m(h),m(h'''')}) m(b^*_{\gamma^{\dagger};m(h''''),m(h'')})) \\
&= \sum_{h'''' \in \beta} m'_2 \circ (\widehat{A}(f)_2 \otimes \id) (x_i \cdot h, m(b^*_{\gamma;m(h),m(h'''')}) m(b^*_{\gamma^{\dagger};m(h''''),m(h'')})) \\
&\qquad - \sum_{h'''' \in \beta} \widehat{A}(f)_2 \circ (m_2 \otimes \id) (x_i \cdot h, m(b^*_{\gamma;m(h),m(h'''')}) m(b^*_{\gamma^{\dagger};m(h''''),m(h'')}))
\end{align*}
by the $n = 3$ consistency conditions; note that $\deg_h m(b^*_{\gamma^{\dagger};m(h''''),m(h'')}) = 1$. Expanding the above expression out, the top line is
\[
\sum_{j,k,h' \in \beta_{\gamma}, h''' \in \beta_{\gamma}, h'''' \in \beta} \tilde{f}_{i,j;h'} \tilde{\tilde{c}}_{h'h;h''''} \tilde{c'}_{j,k;h'''} \tilde{\tilde{c}}_{h'''h'''';h''} (x'_k \cdot h''),
\]
and the bottom line is 
\[
-\sum_{j,k,h' \in \beta_{\gamma}, h''' \in \beta_{\gamma}, h'''' \in \beta} \tilde{c}_{i,j;h'} \tilde{\tilde{c}}_{h'h;h''''} \tilde{f}_{j,k;h'''} \tilde{\tilde{c}}_{h'''h'''';h''} (x'_k \cdot h'')
\]
Thus, the right side of the $n = 2$ condition is 
\begin{align*}
&\sum_{j,k,h' \in \beta_{\alpha}} \tilde{c}_{i,j;h'} \tilde{\tilde{c}}_{h'h;h''} f_{j,k} (x'_k \cdot h'') + \sum_{j,k,h' \in \beta_{\alpha}} c_{i,j} \tilde{f}_{j,k;h'} \tilde{\tilde{c}}_{h'h;h''} (x'_k \cdot h'') \\
&- \sum_{j,k,h' \in \beta_{\gamma}, h''' \in \beta_{\gamma}, h'''' \in \beta} \tilde{f}_{i,j;h'} \tilde{\tilde{c}}_{h'h;h''''} \tilde{c'}_{j,k;h'''} \tilde{\tilde{c}}_{h'''h'''';h''} (x'_k \cdot h'') \\
&+ \sum_{j,k,h' \in \beta_{\gamma}, h''' \in \beta_{\gamma}, h'''' \in \beta} \tilde{c}_{i,j;h'} \tilde{\tilde{c}}_{h'h;h''''} \tilde{f}_{j,k;h'''} \tilde{\tilde{c}}_{h'''h'''';h''} (x'_k \cdot h'').
\end{align*}
This is equal to the left side of the $n = 2$ condition,
\[
\sum_{j,k,h' \in \beta_{\alpha}} \tilde{f}_{i,j;h'} \tilde{\tilde{c}}_{h'h;h''} c'_{j,k} (x'_k \cdot h'') + \sum_{j,k,h' \in \beta_{\alpha}} f_{i,j} \tilde{c'}_{j,k;h'} \tilde{\tilde{c}}_{h'h;h''} (x'_k \cdot h''),
\]
by item~\ref{HnChainMapRels3} of Proposition~\ref{HnChainMapStructure}.

\end{proof}
We define the composition of two $\mathcal{A}_{\infty}$-morphisms as in Roberts \cite{RtypeA}, Definition 27:
\begin{definition} The composition of two $\mathcal{A}_{\infty}$-morphisms $F:\widehat{A} \to \widehat{A'}$ and $G: \widehat{A'} \to \widehat{A''}$ is the $\mathcal{A}_{\infty}$-morphism $G \circ F$ defined as
\[
(G \circ F)_n := \sum_{i + j = n+1} (-1)^{(i+1)(j+1)} G_i \circ (F_j \otimes \left|\id\right|^{(j+1) \otimes (i-1)}).
\]
\end{definition}

For the purposes of this section, we will only need to compose morphisms $\widehat{A}(f)$ and $\widehat{A}(g)$ such that either $f$ or $g$ satisfies a more restrictive condition than $C_{morphism}$:
\begin{definition}
A chain map $f: M \to M'$ of complexes of graded projective right $H^n$-modules satisfies the algebraic condition $\tilde{C}_{morphism}$ if it satisfies the condition $C_{morphism}$ of Definition~\ref{CMorphismDef}, and furthermore $\tilde{f}_{i,j;h'} = 0$ for all $h' \in \beta_{mult}$.
\end{definition}

\begin{proposition}\label{AInftyCompProp} Let $f: M \to M'$ and $g: M' \to M''$ be chain maps between complexes of projective graded right $H^n$ modules, such that $M$, $M'$, and $M''$ satisfy the algebraic condition $C_{module}$ of Definition~\ref{CModuleAlgCondition}, while $f$ and $g$ satisfy the condition $C_{morphism}$ and either $f$ or $g$ satisfies the condition $\tilde{C}_{morphism}$. Then $g \circ f$ satisfies $C_{morphism}$, and
\[
\widehat{A}(g \circ f) = \widehat{A}(g) \circ \widehat{A}(f).
\]
\end{proposition}

\begin{proof}
By the conditions on $f$ and $g$, the chain map $g \circ f$ satisfies the condition $C_{morphism}$, and thus $\widehat{A}(g \circ f)$ is a well-defined $\mathcal{A}_{\infty}$ morphism. We have 
\[
(g \circ f)_{i,k} = \sum_j f_{i,j} g_{j,k}.
\]
If $g$ satisfies $\tilde{C}_{morphism}$, then for $h' \in \beta_{mult}$ we have
\[
\widetilde{(g \circ f)}_{i,k;h'} = \sum_j \tilde{f}_{i,j;h'} g_{j,k},
\]
while if $f$ satisfies $\tilde{C}_{morphism}$, then
\[
\widetilde{(g \circ f)}_{i,k;h'} = \sum_j f_{i,j} \tilde{g}_{j,k;h'}.
\]

Let $x_i \cdot h$ be a generator of $\widehat{A}(M)$. We have 
\[
\widehat{A}(g \circ f)_1 (x_i \cdot h) = \sum_{j,k} f_{i,j} g_{j,k} x''_k \cdot h,
\]
and this sum also equals $(\widehat{A}(g) \circ \widehat{A}(f))_1 (x_i \cdot h)$.

Let $b_{*;h,h''}$ be a generator of $\B \subset \B \astrosun m(\B)^!$. By the definition of the operation $f \mapsto \widehat{A}(f)$, 
\[
\widehat{A}(g \circ f)_2 (x_i \cdot h, b_{*;h,h''}) = 0 = (\widehat{A}(f) \circ \widehat{A}(g))_2 (x_i \cdot h, b_{*;h,h''}).
\]

Finally, let $m(b^*_{*;m(h),m(h'')})$ be a generator of $m(\B)^! \subset \B \astrosun m(\B)^!$. Suppose $g$ satisfies the condition $\tilde{C}_{morphism}$. Then 
\[
\widehat{A}(g \circ f)_2 (x_i \cdot h, m(b^*_{*;m(h),m(h'')})) = \sum_{j,k,h' \in \beta_{mult}} \tilde{f}_{i,j;h'} \tilde{\tilde{c}}_{h'h;h''} g_{j,k} x''_k \cdot h'',
\]
while
\begin{align*}
(\widehat{A}(g) \circ \widehat{A}(f))_2 (x_i \cdot h, m(b^*_{*;m(h),m(h'')})) &= (\widehat{A}(g)_1 \circ \widehat{A}(f)_2) (x_i \cdot h,  m(b^*_{*;m(h),m(h'')})) \\
&= \sum_{j,k,h' \in \beta_{mult}} \tilde{f}_{i,j;h'} \tilde{\tilde{c}}_{h'h;h''} g_{j,k} x''_k \cdot h''.
\end{align*}
The case when $f$ satisfies $\tilde{C}_{morphism}$ is analogous. Thus, $\widehat{A}(g \circ f)_2 = (\widehat{A}(g) \circ \widehat{A}(f))_2$. We have 
\[
\widehat{A}(g \circ f)_n = (\widehat{A}(g) \circ \widehat{A}(f))_n = 0
\]
for all $n > 2$, so $\widehat{A}(g \circ f) = \widehat{A}(g) \circ \widehat{A}(f)$.
\end{proof}

Now we will consider homotopies. The following definition is a special case of Definition 28 of Roberts \cite{RtypeA}:
\begin{definition} 

Let $\B$ be a differential bigraded algebra with idempotent ring $\I$. Let $\widehat{A}$ and $\widehat{A'}$ be differential bigraded right modules over $\B$. Let $F = \{F_n\}$ and $G = \{G_n\}$ be $\mathcal{A}_{\infty}$-morphisms from $\widehat{A}$ to $\widehat{A'}$. An $\mathcal{A}_{\infty}$ homotopy $H$ between $F$ and $G$ is a collection 
\[
H_n: \widehat{A} \otimes_{\I} \B^{\otimes (n-1)} \to \widehat{A'}[0,n]
\]
of bigrading-preserving $\I$-linear maps satisfying the relation
\begin{align*}
F_n - G_n &= m'_1 \circ H_n + (-1)^{n-1} m'_2 \circ (H_{n-1} \otimes \left|\id\right|^{n-1}) \\
&+ (-1)^{n+1}H_n \circ (m_1 \otimes \left|\id\right|^{\otimes(n-1)}) + H_{n-1} \circ (m_2 \otimes \id^{\otimes(n-2)}) \\
&+ (-1)^{n+1} \sum_{k=1}^{n-1} H_n \circ (\id^{\otimes k} \otimes \mu_1 \otimes \left|\id\right|^{\otimes(n-k-1)}) \\
&+ \sum_{k=1}^{n-2} (-1)^k H_{n-1} \circ (\id^{\otimes k} \otimes \mu_2 \otimes \id^{\otimes(n-k-2)}).
\end{align*}
for all $n \geq 1$.
\end{definition}

\begin{example}\label{SimplestHtpyExample}
Suppose $H$ is an $\mathcal{A}_{\infty}$ homotopy between $F$ and $G$ with $H_n = 0$ for $n > 1$. Then the $n = 1$ $\mathcal{A}_{\infty}$ homotopy condition for $H$ becomes
\[
F_1 - G_1 = m'_1 \circ H_1 + H_1 \circ m_1
\]
and the $n = 2$ homotopy condition becomes
\[
F_2 - G_2 = -m'_2 \circ (H_1 \otimes \left|\id\right|) + H_1 \circ m_2.
\]
For $n > 2$, the homotopy condition is $F_n - G_n = 0$.
\end{example}

We will get homotopies $H = \widehat{A}(\psi)$ from certain homotopies $\psi$ between chain maps $f$ and $g$ from a chain complex of projective graded right $H^n$ modules $M$ to another such complex $M'$; that is, $H^n$-linear maps $\psi: M \to M'$ of bidegree $(0,-1)$ satisfying
\[
f - g = d_{M'} \psi + \psi d_M.
\]
We will only need to consider homotopies $\psi$ which satisfy the analogue of the more restrictive condition $\tilde{C}_{morphism}$ on chain maps:
\begin{definition} A homotopy $\psi$ as above satisfies the condition $\tilde{C}_{homotopy}$ if, for all generators $x_i$ of $M$,
\[
\psi(x_i) = \sum_j \psi_{i,j} x'_j
\]
is a basis expansion of $\psi(x_i)$ in the basis $\{ x'_j \cdot h\}$ for $M'$, for some integer coefficients $\psi_{i,j}$.
\end{definition}

If $\psi$ satisfies the condition $\tilde{C}_{homotopy}$, $M$ and $M'$ satisfy the condition $C_{module}$, and $f$ and $g$ satisfy the condition $C_{morphism}$, then the homotopy relation $f - g = d_{M'} \psi + \psi d_M$ becomes the two sets of equations
\begin{equation}\label{TopBasicHtpyRel}
f_{i,k} - g_{i,k} = \sum_j \psi_{i,j} c'_{j,k} + \sum_j c_{i,j} \psi_{j,k}
\end{equation}
for all generators $x_i$ of $M$ and $x'_k$ of $M'$, and
\begin{equation}\label{BottomBasicHtpyRel}
\tilde{f}_{i,k;h'} - \tilde{g}_{i,k;h'} = \sum_j \psi_{i,j} \tilde{c'}_{j,k;h'} + \sum_j \tilde{c}_{i,j;h'} \psi_{j,k}
\end{equation}
for all generators $x_i \in M$, $x'_k \in M'$, and $h' \in \beta_{mult}$. 

\begin{definition}\label{InducedHtpyDef} Suppose $M$ and $M'$ are chain complexes of graded projective right $H^n$-modules, satisfying the condition $C_{module}$, and $f$ and $g$ are chain maps from $M$ to $M'$ satisfying the condition $C_{morphism}$. Let $\psi$ be an $H^n$-linear chain homotopy between $f$ and $g$ satisfying the condition $\tilde{C}_{homotopy}$ defined above.

Define an $\mathcal{A}_{\infty}$ homotopy $\widehat{A}(\psi)$ between $\widehat{A}(f)$ and $\widehat{A}(g)$ by
\[
\widehat{A}(\psi)_1 (x_i \cdot h) := \sum_j \psi_{i,j} (x'_j \cdot h).
\]
and
\[
\widehat{A}(\psi)_n = 0
\]
for $n > 1$.
\end{definition}

\begin{proposition}
$\widehat{A}(\psi)$ is a valid $\mathcal{A}_{\infty}$ homotopy between $\widehat{A}(f)$ and $\widehat{A}(g)$.
\end{proposition}

\begin{proof}
First, $\widehat{A}(\psi)_1$ respects the right action of the idempotent ring $\I_{\beta}$ on $\widehat{A}(M)$ and $\widehat{A}(M')$, and $\widehat{A}(\psi)_1$ preserves intrinsic grading and decreases homological grading by one because $\psi$ has the same properties.

By Example~\ref{SimplestHtpyExample}, the $n = 1$ condition is
\[
\widehat{A}(f)_1 - \widehat{A}(g)_1 = m'_1 \circ \widehat{A}(\psi)_1 + \widehat{A}(\psi)_1 \circ m_1.
\]
If $x_i \cdot h$ is a generator of $\widehat{A}(M)$, then
\begin{align*}
(\widehat{A}(f)_1 - \widehat{A}(g)_1)(x_i \cdot h) &= \sum_k (f_{i,k} - g_{i,k}) (x'_k \cdot h) \\
&= \sum_{j,k} (\psi_{i,j} c'_{j,k}) (x'_k \cdot h) + \sum_{j,k} (c_{i,j} \psi_{j,k}) (x'_k \cdot h)
\end{align*}
by equation~\ref{TopBasicHtpyRel}, while 
\[
m'_1 \circ \widehat{A}(\psi)_1 (x_i \cdot h) = \sum_{j,k} \psi_{i,j} c'_{j,k} (x'_k \cdot h)
\]
and
\[
\widehat{A}(\psi)_1 \circ m_1 (x_i \cdot h) = \sum_{j,k} c_{i,j} \psi_{j,k} (x'_k \cdot h).
\]
Thus, the $n = 1$ condition is satisfied.

By Example~\ref{SimplestHtpyExample}, the $n = 2$ condition is
\begin{equation}\label{N2HtpyCond}
\widehat{A}(f)_2 - \widehat{A}(g)_2 = -m'_2 \circ (\widehat{A}(\psi)_1 \otimes \left|\id\right|) + \widehat{A}(\psi)_1 \circ m_2.
\end{equation}
As in Proposition~\ref{AInftyRelsSatisfied}, we first reduce to the case where the algebra input is one of the generators $b_{*;h,h''}$ or $m(b^*_{*;m(h),m(h'')})$ of $\B \astrosun m(\B)^!$: 
\begin{claim*}
If $b_1$ and $b_2$ are two elements of $\B \astrosun m(\B)^!$ such that the $n = 2$ homotopy condition \ref{N2HtpyCond} is satisfied both when the algebra input is $b_1$ and when it is $b_2$, then the $n = 2$ homotopy condition is also satisfied when the algebra input is $b_1 b_2$.
\end{claim*}

\begin{proof}[Proof of claim.]
When the algebra input is $b_1 \otimes b_2$, we want to show that the maps
\[
\widehat{A}(f)_2 \circ (\id \otimes \mu_2) - \widehat{A}(g)_2 \circ (\id \otimes \mu_2) 
\]
and
\[
-m'_2 \circ (\widehat{A}(\psi)_1 \otimes \left|\id\right|) \circ (\id \otimes \mu_2) + \widehat{A}(\psi)_1 \circ m_2 \circ (\id \otimes \mu_2)
\]
take the same value. In the proof of this claim, the algebra input to all maps will be assumed to be $b_1 \otimes b_2$.

By Example~\ref{LowSummandsAf2Ex} and the associativity of the algebra actions $m_2$ and $m'_2$, we want to show that
\begin{equation}\label{HtpyN2ProductRels}
\begin{aligned}
&m'_2 \circ (\widehat{A}(f)_2 \otimes \left|\id\right|) \\
& \qquad + \widehat{A}(f)_2 \circ (m_2 \otimes \id) \\
& \qquad - m'_2 \circ (\widehat{A}(g)_2 \otimes \left|\id\right|) \\
& \qquad - \widehat{A}(g)_2 \circ (m_2 \otimes \id) \\
&= - m'_2 \circ (m'_2 \otimes \id) \circ (\widehat{A}(\psi)_1 \otimes \left|\id\right| \otimes \left|\id\right|) \\
& \qquad + \widehat{A}(\psi)_1 \circ m_2 \circ (m_2 \otimes \id)
\end{aligned}
\end{equation}
when the algebra input is $b_1 \otimes b_2$.

Call the terms on the left side of equation~\ref{HtpyN2ProductRels} $LHS_1$, $LHS_2$, $LHS_3$, and $LHS_4$; call the terms on the right side $RHS_1$ and $RHS_2$. Using the $n = 2$ homotopy condition for the algebra input $b_1$, the term $RHS_1$ can be written as 
\begin{align*}
&m'_2 \circ (\widehat{A}(f)_2 \otimes \left|\id\right|) \\
&- m'_2 \circ (\widehat{A}(g)_2 \otimes \left|\id\right|) \\
&- m'_2 \circ (\widehat{A}(\psi)_1 \otimes \left|\id\right|) \circ (m_2 \otimes \id);
\end{align*}
call these terms $RHS_{1a}$, $RHS_{1b}$, and $RHS_{1c}$. Using the $n = 2$ homotopy condition for the algebra input $b_2$, the term $RHS_2$ can be written as
\begin{align*}
&\widehat{A}(f)_2 \circ (m_2 \otimes \id) \\
&- \widehat{A}(g)_2 \circ (m_2 \otimes \id) \\
&+ m'_2 \circ (\widehat{A}(\psi)_1 \otimes \left|\id\right|) \circ (m_2 \otimes \id);
\end{align*}
call these terms $RHS_{2a}$, $RHS_{2b}$, and $RHS_{2c}$. Then 
\begin{itemize}
\item $LHS_1 = RHS_{1a}$;
\item $LHS_2 = RHS_{2a}$;
\item $LHS_3 = RHS_{1b}$;
\item $LHS_4 = RHS_{2b}$; and
\item $RHS_{1c} + RHS_{2c} = 0$,
\end{itemize}
proving the claim.
\end{proof}

It remains to show that the $n = 2$ homotopy condition \ref{N2HtpyCond} is satisfied when the algebra input is one of the multiplicative generators of $\B \astrosun m(\B)^!$. When the input is $b_{*;h,h''}$, the left side of the $n = 2$ condition is zero, so we want to show that the right side is also zero. If $(x_i \cdot h)$ is a generator of $\widehat{A}(M)$, then the right side of the $n = 2$ condition with algebra input $b_{*;h,h''}$ is
\[
-\sum_j \psi_{i,j} (x'_j \cdot h'') + \sum_j \psi_{i,j} (x'_j \cdot h''),
\]
which is zero as desired.

Finally, suppose the algebra input is $m(b^*_{*;m(h),m(h'')})$ and let $x_i \cdot h$ be a generator of $\widehat{A}(M)$. The left side of the $n = 2$ condition applied to these inputs is
\[
\sum_{k,h' \in \beta} (\tilde{f}_{i,k;h'} - \tilde{g}_{i,k;h'}) \tilde{\tilde{c}}_{h'h;h''} (x'_k \cdot h''),
\]
which equals
\[
\sum_{j,k,h' \in \beta} \psi_{i,j} \tilde{c'}_{j,k;h'} \tilde{\tilde{c}}_{h'h;h''} + \sum_{j,k,h' \in \beta} \tilde{c}_{i,j;h'} \psi_{j,k} \tilde{\tilde{c}}_{h'h;h''}
\]
by equation~\ref{BottomBasicHtpyRel}. This expression is also equal to the right side of the $n = 2$ condition applied to these inputs, since $m(b^*_{m(h),m(h'')})$ has homological degree 1. Thus, the $\mathcal{A}_{\infty}$ homotopy relations are satisfied for $\widehat{A}(\psi)$.

\end{proof}

\begin{corollary}\label{GetInducedHtpyEquivCorr}
Let $M$ and $M'$ be chain complexes of projective graded right $H^n$-modules satisfying the algebraic condition $C_{module}$. Suppose there exist chain maps $f: M \to M'$ and $g: M' \to M$ satisfying the condition $C_{morphism}$, with either $f$ or $g$ satisfying the more restrictive condition $\tilde{C}_{morphism}$, and chain homotopies $\psi$ between $g \circ f$ and $\id_M$ and $\psi'$ between $f \circ g$ and $\id_{M'}$, both satisfying the condition $\tilde{C}_{homotopy}$.

Then $\widehat{A}(M)$ and $\widehat{A}(M')$ are $\mathcal{A}_{\infty}$-homotopy equivalent Type A structures over $\B \astrosun m(\B)^!$.
\end{corollary}

\begin{proof} By Proposition~\ref{AInftyCompProp}, 
\[
\widehat{A}(g) \circ \widehat{A}(f) = \widehat{A}(g \circ f)
\]
and
\[
\widehat{A}(f) \circ \widehat{A}(g) = \widehat{A}(f \circ g).
\]
Thus, $\widehat{A}(\psi)$ provides an $\mathcal{A}_{\infty}$ homotopy between $\widehat{A}(g) \circ \widehat{A}(f)$ and $\id_{\widehat{A}(M)} = \widehat{A}(\id_M)$, and $\widehat{A}(\psi')$ provides an $\mathcal{A}_{\infty}$ homotopy between $\widehat{A}(f) \circ \widehat{A}(g)$ and $\id_{\widehat{A}(M')} = \widehat{A}(\id_{M'})$.
\end{proof}

The case of interest to us is when $M = [T]^{Kh}$ and $M' = [T']^{Kh}$ for two oriented tangle diagrams $T$ and $T'$ in $\R_{\leq 0} \otimes \R$ which are related by a Reidemeister move. In \cite{KhovFunctor}, Khovanov shows that $[T]^{Kh}$ is chain homotopy equivalent to $[T']^{Kh}$. In the following two propositions, we verify that the maps involved in these homotopy equivalences satisfy the relevant algebraic conditions:

\begin{proposition}\label{SubcxHtpyEquiv1}
Let $(M,d_M)$ be a chain complex of projective graded right $H^n$-modules, with a basis $\{x_i \cdot h\}$ as usual. Assume the following conditions hold:
\begin{enumerate}
\item
$M$ satisfies the algebraic condition $C_{module}$ with respect to the basis $\{x_i \cdot h\}$. 

\item 
$M \cong M_1 \oplus M_2$ as right $H^n$-modules; furthermore, $M_1$ is the submodule of $M$ spanned over $H^n$ by some subset of the $x_i$, while $M_2$ is the submodule spanned by the rest of the $x_i$. 

\item\label{SubcxCondition3}
$M_2$ is a subcomplex of $M$; write $d_M$ in matrix form with respect to the direct sum decomposition as
\[
d_M = \begin{bmatrix} d_1 & 0 \\ d_{1,2} & d_2 \end{bmatrix}.
\]
Note that $d_M^2 = 0$ is equivalent to the equations $d_1^2 = 0$, $d_{1,2} \circ d_1 + d_2 \circ d_{1,2} = 0$, and $d_2^2 = 0$.

\item\label{SubcxCondition4}
There exists an $H^n$-linear map $\psi': M_2 \to M_2$ of bidegree $(0,-1)$ with $\id_{M_2} = d_2 \psi' + \psi' d_2$, and such that we may write
\[
\psi'(x_i) = \sum_j \psi'_{i,j} x_j,
\]
where the $\psi'_{i,j}$ are integer coefficients.
\end{enumerate}

Among the equations implied by $d^2_M = 0$ is $d^2_1 = 0$; thus $(M_1,d_1)$ is a chain complex of projective graded right modules over $H^n$. It satisfies the condition $C_{module}$ because $(M,d_M)$ does. Define $f: (M,d_M) \to (M_1,d_1)$, $g: (M_1,d_1) \to (M,d_M)$, and $\psi: M \to M$ by the matrix formulas
\[
f = \begin{bmatrix} \id_{M_1} & 0 \end{bmatrix}, \quad g = \begin{bmatrix} \id_{M_1} \\ -\psi' d_{1,2} \end{bmatrix}, \textrm{ and } \quad \psi = \begin{bmatrix} 0 & 0 \\ 0 & -\psi' \end{bmatrix}.
\]
Then $f$ and $g$ are chain maps, $f \circ g = \id_{M_1}$, and $g \circ f - \id_{M} = d_M \psi + \psi d_M$. Furthermore, $f$ satisfies the condition $\tilde{C}_{morphism}$, $g$ satisfies the condition $C_{morphism}$, and $\psi$ satisfies the condition $\tilde{C}_{homotopy}$.
\end{proposition}

\begin{proof}
Both $f$ and $g$ are bigrading-preserving and $H^n$-linear; $\psi$ preserves the intrinsic grading and decreases the homological grading by one, because the same holds for $\psi'$. The map $f$ is a chain map because it is the projection map onto a quotient complex. To show that $g$ is a chain map, we want to show that $g(d_1(x_i)) = d_M(g(x_i))$ for all $x_i$ in $M_1$. We have
\[
g (d_1(x_i)) = d_1(x_i) - \psi' d_{1,2} \circ d_1(x_i),
\]
and 
\begin{align*}
d_M (g(x_i)) &= d_M(x_i)  - d_M(\psi' d_{1,2} (x_i)) \\
&= d_M(x_i) + \psi'(d_2 \circ d_{1,2}) (x_i) - d_{1,2}(x_i) \\
&= d_1(x_i) - \psi'(d_{1,2} \circ d_1) (x_i).
\end{align*}
In the second line, we use that
\[
d_{1,2}(x_i) = \id_{M'}(d_{1,2}(x_i)) = d_M (\psi' d_{1,2} (x_i)) + \psi'(d_2 \circ d_{1,2})(x_i),
\]
and in the third line we use the equation $d_2 \circ d_{1,2} + d_{1,2} \circ d_1 = 0$ from item~\ref{SubcxCondition3} above. Thus, $g$ is a chain map as well, and by definition, $f \circ g = \id_{M'}$. To verify that $\psi$ is a homotopy between $g \circ f$ and $\id_M$, we can write out the terms of the relevant equation as matrices:
\[
g \circ f = \begin{bmatrix} \id_{M_1} & 0 \\ -\psi' d_{1,2} & 0 \end{bmatrix}, \quad d_M \psi = \begin{bmatrix} 0 & 0 \\ 0 & -d_2 \psi \end{bmatrix}, \quad \psi d_M = \begin{bmatrix} 0 & 0 \\ -\psi d_{1,2} & - \psi d_2 \end{bmatrix}.
\]
Thus, the equation $g \circ f - \id_{M} = d_M \psi + \psi d_M$ holds.

From their definitions, $f$ satisfies the condition $\tilde{C}_{morphism}$ and $\psi$ satisfies the condition $\tilde{C}_{homotopy}$. By item~\ref{SubcxCondition4} above, $g$ satisfies the condition $C_{morphism}$. 
\end{proof}

\begin{proposition}\label{SubcxHtpyEquiv2}
Let $M$ be a chain complex of projective graded right $H^n$-modules, with a basis $\{x_i \cdot h\}$ as usual. Assume the following conditions hold:
\begin{enumerate}
\item
$M$ satisfies the algebraic condition $C_{module}$ with respect to the basis $\{x_i \cdot h\}$. 

\item 
$M \cong M_1 \oplus M_2$ as right $H^n$-modules, and $M_1$ is the submodule of $M$ spanned over $H^n$ by some subset of the $x_i$, say $M_1 = \{x_i \cdot h: i \in S\}$. The submodule $M_2$ has a $\Z$-basis $\{z_i: i \notin S\}$, where 
\[
z_i = x_i + \sum_{j \in S} \tau_{i,j} x_j + \sum_{j \in S, h' \in \beta_{mult}} \tilde{\tau}_{i,j;h'} x_j \cdot h'
\]
for some integer coefficients $\tau_{i,j}$ and $\tilde{\tau}_{i,j;h'}$.

\item\label{Subcx2Condition3}
$M_2$ is a subcomplex of $M$; write $d_M$ in matrix form with respect to the direct sum decomposition as
\[
d_M = \begin{bmatrix} d_1 & 0 \\ d_{1,2} & d_2 \end{bmatrix}.
\]

\item\label{Subcx2Condition4}
There exists an $H^n$-linear map $\psi': M_2 \to M_2$ of bidegree $(0,-1)$ with $\id_{M_2} = d_2 \psi' + \psi' d_2$, and such that we may write
\[
\psi'(z_i) = \sum_j \psi'_{i,j} z_j,
\]
for some integer coefficients $\psi'_{i,j}$.

\item\label{Subcx2Condition5}
Write
\[
d_M(x_i) = \sum_j c_{i,j} x_j + \sum_{j, h' \in \beta_{mult}} \tilde{c}_{i,j;h'} x_j \cdot h'.
\]
For all indices $i \in S$, $j \notin S$, $k \notin S$ and elements $h'$ of $\beta_{mult}$, we have $\tilde{c}_{i,j;h'} \psi'_{j,k} = 0$.

\item\label{Subcx2Condition6}
For all indices $i \in X$, $j \notin S$, $k \in S$, and elements $h', h'''$ of $\beta_{mult}$, we have $\tilde{c}_{i,j;h'} \tau_{j,k} = 0$ and $\tilde{c}_{i,j;h'} \tilde{\tau}_{j,k;h'''} = 0$.

\item\label{Subcx2Condition7}
For all indices $i \notin S$, $j \notin S$, $k \in S$ and elements $h'$ of $\beta_{mult}$, we have $\psi'_{i,j} \tau_{j,k} = 0$ and $\psi'_{i,j} \tilde{\tau}_{j,k;h'} = 0$.
\end{enumerate}
As bigraded right $H^n$-modules, $M_1 \cong M / M_2$. Since $M_2$ is a subcomplex of $M$, the differential on $M$ induces a differential on $M/M_2$ and thus on $M_1$. Call this differential $d_M \textrm{ mod } M_2$; then $(M_1,d_M \textrm{ mod } M_2)$ is a chain complex of projective graded right modules over $H^n$. In fact, $d_M \textrm{ mod } M_2$ agrees with $d_1$, so we will identify these differentials. Define $f: (M,d_M) \to (M_1,d_1)$, $g: (M_1,d_1) \to (M,d_M)$, and $\psi: M \to M$ by the matrix formulas
\[
f = \begin{bmatrix} \id_{M_1} & 0 \end{bmatrix}, \quad g = \begin{bmatrix} \id_{M_1} \\ -\psi' d_{1,2} \end{bmatrix}, \textrm{ and } \quad \psi = \begin{bmatrix} 0 & 0 \\ 0 & -\psi' \end{bmatrix}.
\]
Then $f$ and $g$ are chain maps, $f \circ g = \id_{M_1}$, and $g \circ f - \id_{M} = d_M \psi + \psi d_M$. Furthermore, $(M_1,d_1)$ satisfies the condition $C_{module}$, $f$ satisfies the condition $C_{morphism}$, $g$ satisfies the condition $\tilde{C}_{morphism}$, and $\psi$ satisfies the condition $\tilde{C}_{homotopy}$.
\end{proposition}

\begin{proof} The proof that $f$ and $g$ are chain maps, and that $\psi$ is a homotopy between $f$ and $g$, is the same as in Proposition~\ref{SubcxHtpyEquiv1}. Note that here, all the matrices are chosen with respect to the basis $\{x_i, z_j: i \in S, j \notin S\}$ of $M$, since this basis is compatible with the direct sum decomposition $M \cong M_1 \oplus M_2$. For the algebraic conditions, we need to use the basis $\{x_i,x_j:i \in S, j \notin S\}$ instead, which is not compatible with the direct sum decomposition. Thus, we want to express $d_1$, $f$, $g$, and $\psi$ in this basis and show they satisfy $C_{module}$, $C_{morphism}$, $\tilde{C}_{morphism}$, and $\tilde{C}_{homotopy}$ respectively.

We may write
\begin{align*}
d_M(x_i) &= \sum_{k \in S} c_{i,k} x_k + \sum_{k \in S, h' \in \beta_{mult}} \tilde{c}_{i,k;h'} x_k \cdot h' \\
&\qquad + \sum_{j \notin S} c_{i,j} z_j - \sum_{j \notin S, k \in S} c_{i,j} \tau_{j,k} x_k - \sum_{j \notin S, k \in S, h' \in \beta_{mult}} c_{i,j} \tilde{\tau}_{j,k;h'} x_k \cdot h' \\
&\qquad + \sum_{j \notin S, h' \in \beta_{mult}} \tilde{c}_{i,j;h'} z_j \cdot h';
\end{align*}
the final two terms which would appear in this expression are zero by item~\ref{Subcx2Condition6} of the above assumptions. The third and sixth terms of this expression make up $d_{1,2}(x_i)$:
\[
d_{1,2}(x_i) = \sum_{j \notin S} c_{i,j} z_j + \sum_{j \notin S, h' \in \beta_{mult}} \tilde{c}_{i,j;h'} z_j \cdot h',
\]
and the rest of the terms make up $d_1(x_i)$:
\begin{align*}
d_1(x_i) &= \sum_{k \in S} c_{i,k} x_k + \sum_{k \in S, h' \in \beta_{mult}} \tilde{c}_{i,k;h'} x_k \cdot h' \\
&\qquad - \sum_{j \notin S, k \in S} c_{i,j} \tau_{j,k} x_k - \sum_{j \notin S, k \in S, h' \in \beta_{mult}} c_{i,j} \tilde{\tau}_{j,k;h'} x_k \cdot h'. \\
\end{align*}
From this formula we can see that $(M_1, d_1)$ satisfies the condition $C_{module}$.

For $x_i$ with $i \in S$, we have $f(x_i) = x_i$. For $z_j$ with $j \notin S$, we have $f(z_j) = 0$. We may write this equation as
\[
f\bigg(x_j + \sum_{k \in S} \tau_{j,k} x_k + \sum_{k \in S, h' \in \beta_{mult}} \tilde{\tau}_{j,k;h'} x_k \cdot h'\bigg) = 0,
\]
or equivalently
\[
f(x_j) = - \sum_{k \in S} \tau_{j,k} x_k - \sum_{k \in S, h' \in \beta_{mult}} \tilde{\tau}_{j,k;h'} x_k \cdot h'.
\]
By this formula, we see that $f$ satisfies condition $C_{morphism}$.

For the map $g$, we can write
\begin{align*}
\psi' \circ d_{1,2}(x_i) &= \psi'\bigg(\sum_{j \notin S} c_{i,j} z_j + \sum_{j \notin S, h' \in \beta_{mult}} \tilde{c}_{i,j;h'} z_j \cdot h'\bigg) \\
&= \sum_{j \notin S, k \notin S} c_{i,j} \psi'_{j,k} z_k
\end{align*}
by assumption~\ref{Subcx2Condition5} above. Using assumption~\ref{Subcx2Condition7}, we can write this sum as
\[
\sum_{j \notin S, k \notin S} c_{i,j} \psi'_{j,k} x_k.
\] 
Thus,
\[
g(x_i) = x_i - \sum_{j \notin S, k \notin S} c_{i,j} \psi'_{j,k} x_k,
\]
and $g$ satisfies condition $\tilde{C}_{morphism}$.

Finally, if $x_i \in S$, then $\psi(x_i) = 0$ by definition. For $x_j \notin S$, we have
\begin{align*}
\psi(x_j) &= \psi\bigg(z_j - \sum_{k \in S} \tau_{j,k} x_k - \sum_{k \in S, h' \in \beta} \tilde{\tau}_{j,k} x_k \cdot h'\bigg) \\
&= \psi(z_j) \\
&= -\sum_{k \notin S} \psi'_{j,k} z_k \\
&= -\sum_{k \notin S} \psi'_{j,k} x_k;
\end{align*}
the last line follows from assumption~\ref{Subcx2Condition7}. Thus, $\psi$ satisfies the condition $\tilde{C}_{homotopy}$.
\end{proof}

\begin{figure}
\includegraphics[scale=0.625]{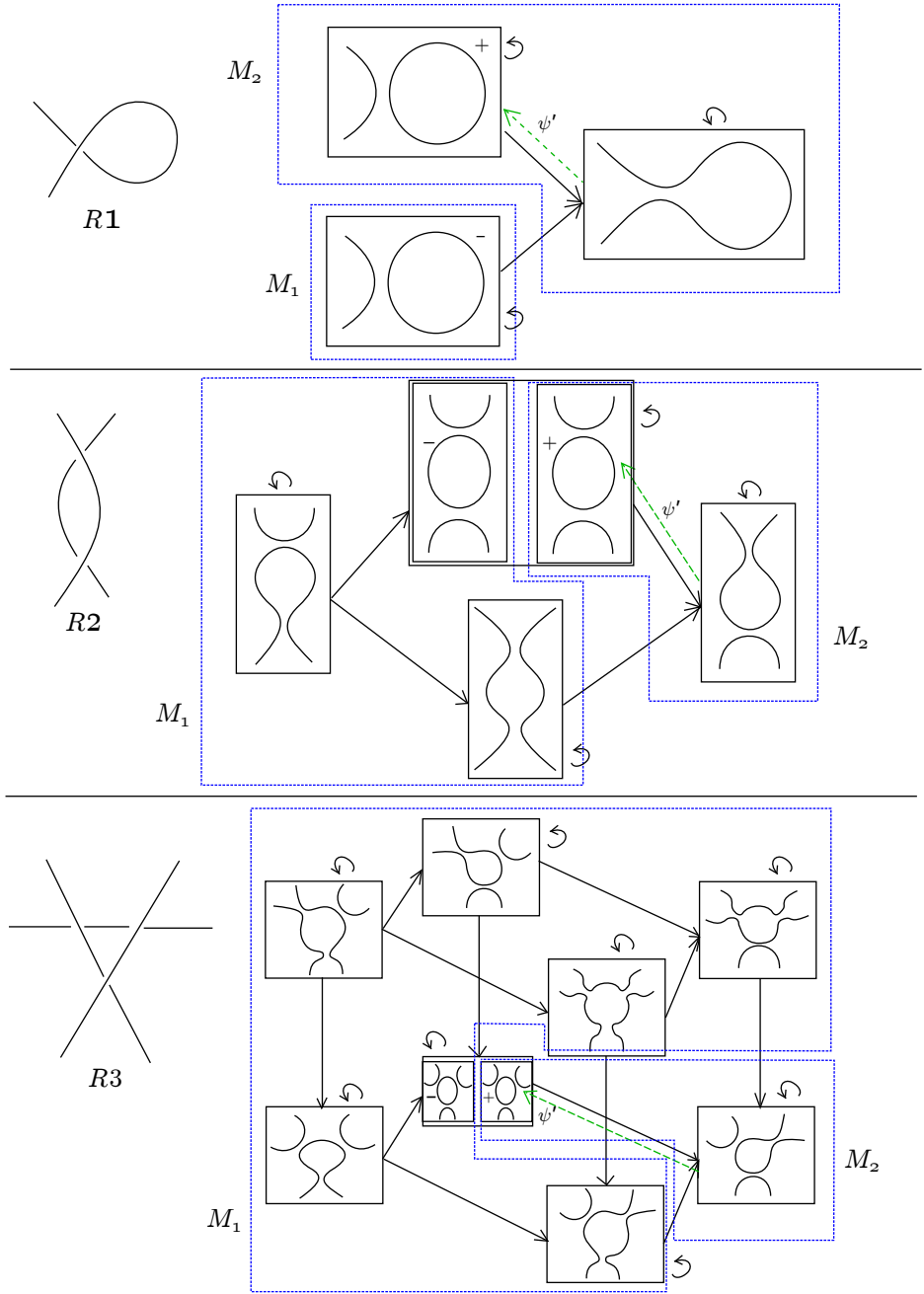}
\caption{First step of R1, R2, and R3 moves.}
\label{Fig1}
\end{figure}

\begin{figure}
\includegraphics[scale=0.625]{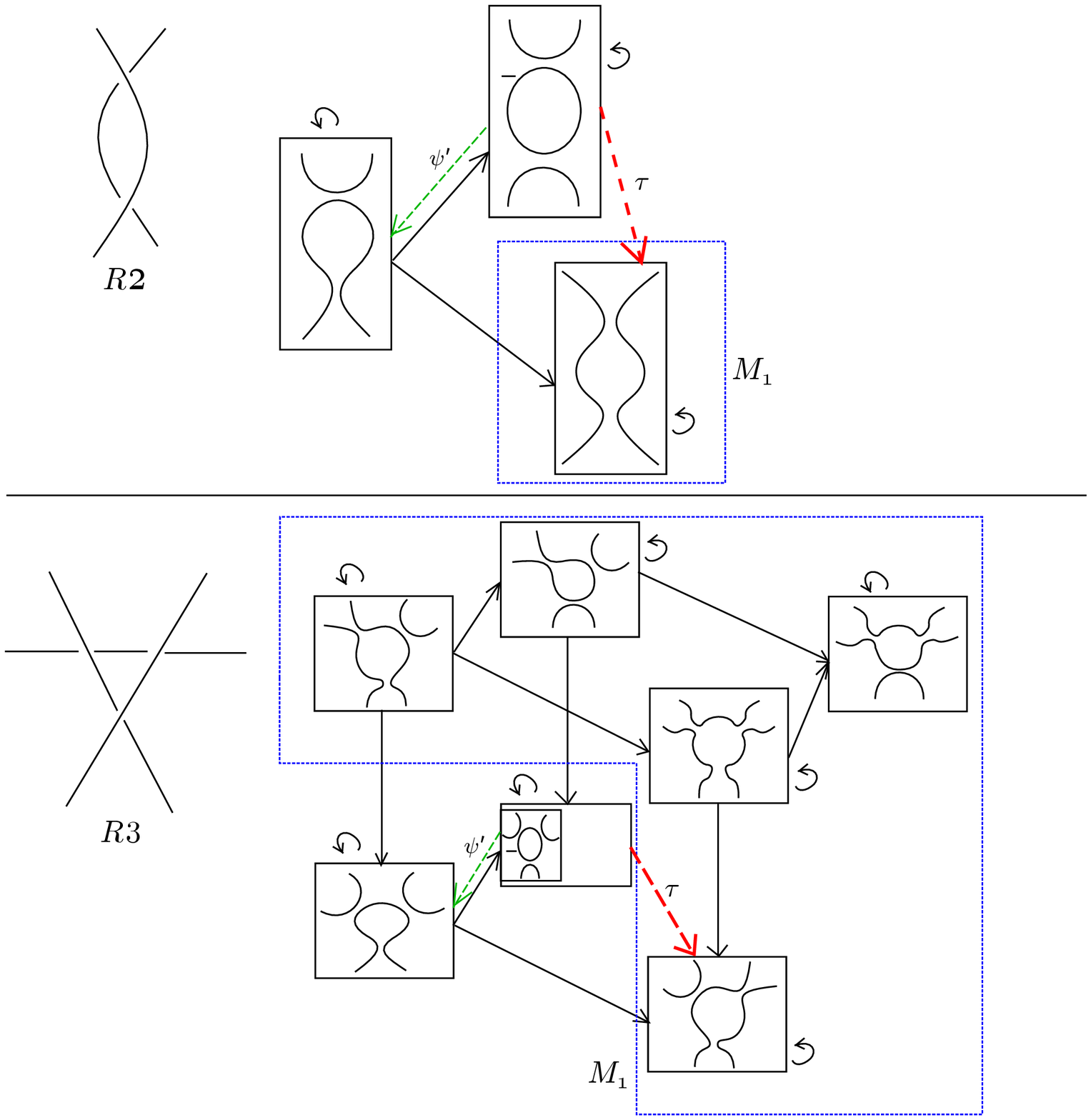}
\caption{Second step of R2 and R3 moves.}
\label{Fig2}
\end{figure}

\begin{corollary}\label{BomBHtpyEquivCorr} If $T$ and $T'$ are oriented tangle diagrams in $\R_{\leq 0} \otimes \R$ which are related by a Reidemeister move, then $\A([T]^{Kh})$ and $\A([T']^{Kh})$ are $\mathcal{A}_{\infty}$ homotopy equivalent as Type A structures over $\B \astrosun m(\B)^!$.
\end{corollary}

\begin{proof}
When $T$ and $T'$ are related by an R1 move, Khovanov's homotopy equivalence between $[T]^{Kh}$ and $[T']^{Kh}$ is of the type constructed in Proposition~\ref{SubcxHtpyEquiv1}. This is most easily seen by looking at the top diagram of Figure~\ref{Fig1}; one can check that the four conditions of Proposition~\ref{SubcxHtpyEquiv1} are satisfied.

When $T$ and $T'$ are related by an R2 move, Khovanov's homotopy equivalence is a composition of a homotopy equivalence from Proposition~\ref{SubcxHtpyEquiv1} followed by one from Proposition~\ref{SubcxHtpyEquiv2}. The relevant diagrams are the middle diagram of Figure~\ref{Fig1} and the top diagram of Figure~\ref{Fig2}. The first homotopy equivalence is very similar to the R1 move. For the second homotopy equivalence, one can see from the diagram in Figure~\ref{Fig2} that the conditions of Proposition~\ref{SubcxHtpyEquiv2} are satisfied.

Finally, when $T$ and $T'$ are related by an R3 move, Khovanov's homotopy equivalence comes from doing a homotopy equivalence from Proposition~\ref{SubcxHtpyEquiv1} and then a homotopy equivalence from Proposition~\ref{SubcxHtpyEquiv2}, to both $[T]^{Kh}$ and $[T']^{Kh}$. After these homotopy equivalences, the complexes are isomorphic. Thus, the full homotopy equivalence from $[T]^{Kh}$ to $[T']^{Kh}$ is a composition of four homotopy equivalences from Proposition~\ref{SubcxHtpyEquiv1} and Proposition~\ref{SubcxHtpyEquiv2}. The relevant diagrams are the bottom diagram of Figure~\ref{Fig1} and the bottom diagram of Figure~\ref{Fig2}. Again, one can check that the conditions of Proposition~\ref{SubcxHtpyEquiv1} are satisfied using Figure~\ref{Fig1}, and that the conditions of Proposition~\ref{SubcxHtpyEquiv2} are satisfied using Figure~\ref{Fig2}.
\end{proof}

\begin{proposition}\label{HomotopyEqDescendProp} If $T$ and $T'$ are oriented tangle diagrams in $\R_{\leq 0} \otimes \R$ which are related by a Reidemeister move, then the $\mathcal{A}_{\infty}$ homotopy equivalence between $\widehat{A}([T]^{Kh})$ and $\widehat{A}([T']^{Kh})$ of Corollary~\ref{BomBHtpyEquivCorr} descend to $\mathcal{A}_{\infty}$ homotopy equivalences of Type A structures over the quotient algebra $\B \Gamma_n$ of $\B \astrosun m(\B)^!$.
\end{proposition}

\begin{proof} 
For homotopy equivalences $\{f: M \to M_1, g: M_1 \to M, \psi: M \to M \}$ coming from Proposition~\ref{SubcxHtpyEquiv1}, when doing an R1 move, the first step of an R2 move, or the first or fourth step of an R3 move, we only need to show that $\widehat{A}(g)_2$ descends from a map
\[
\widehat{A}(M_1) \otimes_{\I_{\beta}} \B \astrosun m(\B)^! \to \widehat{A}(M)
\]
to a map
\[
\widehat{A}(M_1) \otimes_{\I_{\beta}} \B \Gamma_n \to \widehat{A}(M);
\]
since $f$ satisfies $\tilde{C}_{morphism}$ and $\psi$ satisfies $\tilde{C}_{homotopy}$, we have $\widehat{A}(f)_2 = 0$ and $\widehat{A}(\psi)_2 = 0$.

As in Proposition~\ref{KhovRobertsLSQuotient}, let $a$, $b$, $c$, and $d$ be vertices of a tetrahedron in the graph $G$. We want to show that $\widehat{A}(g)(-,a+c) = 0$, $\widehat{A}(g)(-,a+d) = 0$, and $\widehat{A}(g)(-,b+c) = 0$. We will show only that $\widehat{A}(g)(-,a+c) = 0$; by symmetry, the proof is the same for the other two extra relations.

Write 
\[
a = m(b^*_{\gamma;m(h_1),m(h_2)}) m(b^*_{\eta';m(h_2),m(h_3)})
\]
and 
\[
c = m(b^*_{\eta;m(h_1),m(h''_2)}) m(b^*_{\gamma';m(h''_2),m(h_3)}).
\]
Let $x_i \cdot h_1$ be a generator of $\widehat{A}(M_1)$; we have $i \in S$, in the notation of Proposition~\ref{SubcxHtpyEquiv1}. Then
\begin{equation}\label{HtpyDescendsEqn1}
\begin{aligned}
&\widehat{A}(g)_2(x_i \cdot h_1, a) \\
&= (\widehat{A}(g)_2 \circ (m_2 \otimes \id))(x_i \cdot h, m(b^*_{\gamma;m(h_1),m(h_2)}) \otimes m(b^*_{\eta';m(h_2),m(h_3)})) \\
&+ (m'_2 \circ (\widehat{A}(g)_2 \otimes \left|\id\right|)) (x_i \cdot h, m(b^*_{\gamma;m(h_1),m(h_2)}) \otimes m(b^*_{\eta';m(h_2),m(h_3)})), \\
\end{aligned}
\end{equation}
using the $n = 3$ consistency condition for the $\mathcal{A}_{\infty}$ morphism $\widehat{A}(g)_2$. The first term on the right side of equation~\ref{HtpyDescendsEqn1} can be expanded out as 
\begin{equation}\label{FinalPropRHS1}
-\sum_{i,j \in S, k,l \notin S} \tilde{c}_{i,j;h'_{\gamma}} \tilde{c}_{j,k;h'_{\eta'}} \psi'_{k,l} (x_l \cdot h_3),
\end{equation}
where $h'_{\gamma}$ and $h'_{\eta'}$ are determined by $\gamma$ and $\eta'$, while the second term on the right side of equation~\ref{HtpyDescendsEqn1} can be expanded out as
\begin{equation}\label{FinalPropRHS2}
\sum_{i \in S, j,k,l \notin S} \tilde{c}_{i,j;h'_{\gamma}} \tilde{c}_{j,k;h'_{\eta'}} \psi'_{k,l} (x_l \cdot h_3).
\end{equation}
Similarly, we may write $\widehat{A}(g)_2(x_i \cdot h_1, c)$ as the sum of the expressions
\begin{equation}\label{FinalPropRHS3}
-\sum_{i,j \in S, k,l \notin S} \tilde{c}_{i,j;h'_{\eta}} \tilde{c}_{j,k;h'_{\gamma'}} \psi'_{k,l} (x_l \cdot h_3)
\end{equation}
and
\begin{equation}\label{FinalPropRHS4}
\sum_{i \in S, j,k,l \notin S} \tilde{c}_{i,j;h'_{\eta}} \tilde{c}_{j,k;h'_{\gamma'}} \psi'_{k,l} (x_l \cdot h_3).
\end{equation}

Examining Figure~\ref{Fig1}, we see that if the expression \ref{FinalPropRHS1} or \ref{FinalPropRHS4} is nonzero for a particular choice of $x_i \cdot h_1$ and $x_l \cdot h_3$, then both \ref{FinalPropRHS2} and \ref{FinalPropRHS3} are zero for these generators of $\widehat{A}(M_1)$ and $\widehat{A}(M)$. Similarly, if \ref{FinalPropRHS2} or \ref{FinalPropRHS3} is nonzero, then both \ref{FinalPropRHS1} and \ref{FinalPropRHS4} are zero.

Thus, suppose that we have $x_i \cdot h_1$ and $x_l \cdot h_3$ such that the expression \ref{FinalPropRHS1} is nonzero. We want to show that, in fact, the expressions \ref{FinalPropRHS1} and \ref{FinalPropRHS4} sum to zero (the case where expressions \ref{FinalPropRHS2} and \ref{FinalPropRHS3} are nonzero instead is very similar). 

Indeed, generators of all complexes $(M,d_M)$ and $(M_1,d_1)$ under consideration come from generators of the Khovanov complex $[T]^{Kh}$ of a tangle $T$, and by Remark~\ref{OrderingIsoIndep} we may choose any ordering we like for the crossings of $T$. We will order the crossings of $T$ such that the one, two, or three crossings ``local'' to the Reidemeister move being performed come first in the ordering. 

Now, to each quadruple $(i \in S, j \in S, k \notin S, l \notin S)$ giving rise to a nonzero term of expression \ref{FinalPropRHS1}, we may associate a pair of indices $(j' \notin S, k' \notin S)$, such that $\tilde{c}_{i,j';h'_{\eta}} \tilde{c}_{j',k';h'_{\gamma'}} \psi'_{k',l} \neq 0$. In fact, with the above ordering convention, we have
\[
\tilde{c}_{i,j;h'_{\gamma}} \tilde{c}_{j,k;h'_{\eta'}} \psi'_{k,l} = \tilde{c}_{i,j';h'_{\eta}} \tilde{c}_{j',k';h'_{\gamma'}} \psi'_{k',l}.
\]
To see that this equation holds, first note that the only component of $g$ relevant for $\widehat{A}(g)_2$ is $-\psi' \circ d_{1,2}$. Recall that the terms $\tilde{c}$, and thus also the terms $\psi'$, are computed as in Example~\ref{HnCCoeffsEx}. A term like $\tilde{c}_{i,j;h'_{\gamma}} \tilde{c}_{j,k;h'_{\eta'}} \psi'_{k,l}$ corresponds to doing one step of $d_1$ by changing some crossing away from the local area (and thus higher in the ordering) from $0$ to $1$, then doing one step of $d_{1,2}$ by changing one of the local crossings from $0$ to $1$, and then finally doing one step of $\psi'$ by changing a different local crossing from $1$ to $0$. The corresponding term $\tilde{c}_{i,j';h'_{\eta}} \tilde{c}_{j',k';h'_{\gamma'}} \psi'_{k',l}$ does the $d_{1,2}$ and $-\psi'$ steps before changing the non-local crossing from $0$ to $1$. But, when changing the local crossings, the signs are the same for both terms because the local crossings occur at the beginning of the ordering. When changing the non-local crossing, the signs are also the same for both terms because doing $d_{1,2}$ and $-\psi$ on the local crossings does not increase or decrease the number of crossings with a $1$-resolution ($d_{1,2}$ increases this number by $1$, and then $-\psi'$ decreases it by $1$).

Thus, $\widehat{A}(g)_2(-,a + c) = 0$. By symmetry, we also have $\widehat{A}(g)_2(-,a+d) = 0$ and $\widehat{A}(g)_2(-,b+c) = 0$, so $\widehat{A}(g)$ descends to an $\mathcal{A}_{\infty}$ morphism of Type A structures over $\B \Gamma_n$.

For homotopy equivalences $\{f: M \to M_1, g: M_1 \to M, \psi: M \to M \}$ coming from Proposition~\ref{SubcxHtpyEquiv2}, the argument is similar enough that we will simply outline the differences with the above proof. Homotopy equivalences from Proposition~\ref{SubcxHtpyEquiv2} arise when doing the second step of an R2 move or the second or third step of an R3 move. For these equivalences, we only need to show that $\widehat{A}(f)_2$ descends from a map
\[
\widehat{A}(M) \otimes_{\I_{\beta}} \B \astrosun m(\B)^! \to \widehat{A}(M_1)
\]
to a map
\[
\widehat{A}(M) \otimes_{\I_{\beta}} \B \Gamma_n \to \widehat{A}(M_1),
\]
since we automatically have $\widehat{A}(g)_2 = 0$ by condition $\tilde{C}_{morphism}$ on $g$ and $\widehat{A}(\psi)_2 = 0$ by condition $\tilde{C}_{homotopy}$ on $\psi$.

The only terms of $f$ in the basis expansion $\{x_i \cdot h: i \in S, i \notin S\}$ of $M$ which are relevant for $\widehat{A}(f)_2$ are the terms with coefficients $-\tilde{\tau}_{j,k;h'}$; see the proof of Proposition~\ref{SubcxHtpyEquiv2}. These $\tau$ terms play a role analogous to the $-\psi' \circ d_{1,2}$ terms in the proof above for homotopy equivalences from Proposition~\ref{SubcxHtpyEquiv1}. Indeed, doing one of these $\tau$ terms amounts to doing one step of $\psi'$ by changing a local crossing from a $1$ to a $0$, and then doing one step of $d_M$ by changing a local crossing from a $0$ to a $1$. Thus, an argument analogous to the one above shows that $\widehat{A}(f)_2(-,a+c) = 0$, and by symmetry that $\widehat{A}(f)_2(-,a+d) = 0$ and $\widehat{A}(f)_2(-,b+c) = 0$. Hence $\widehat{A}(f)$ descends to an $\mathcal{A}_{\infty}$ morphism of Type A structures over $\B \Gamma_n$.

\end{proof}

Proposition~\ref{HomotopyEqDescendProp} gives us an alternate proof of Corollary 33 of Roberts \cite{RtypeA}.

\subsection{Equivalences of Type D structures}\label{TypeDEquivSect}

We first define morphisms and homotopies of Type D structures with sign conventions following Definition 37 of Roberts \cite{RtypeD}:
\begin{definition} Let $\B$ be a differential bigraded algebra with idempotent ring $\I$. Let $(\D, \delta)$ and $(\D', \delta')$ be Type D structures over $\B$. A morphism of Type D structures $F: \D \to \D'$ is a bigrading-preserving $\I$-linear map $F: \D \to \B \otimes_{\I} \D'$ satisfying the Type D morphism relation
\[
(\mu_1 \otimes \left|\id\right|) \circ F =  (\mu_2 \otimes \id) \circ (\id \otimes F) \circ \delta - (\mu_2 \otimes \id) \circ (\id \otimes \delta') \circ F.
\]

The composition of two morphisms of Type D structures $F: \D \to \D'$ and $G: \D' \to \D''$ is
\[
G \circ F := (\mu_2 \otimes \id) \circ (\id \otimes G) \otimes F,
\]
a grading-preserving $\I$-linear map from $\D$ to $\B \otimes_{\I} \D''$ which also satisfies the Type D morphism relation.
\end{definition}

\begin{definition} Let $F: \D \to \D'$, $G: \D \to \D'$ be morphisms of Type D structures over $\B$. A homotopy of morphisms of Type D structures between $F$ and $G$ is a bigrading-preserving $\I$-linear map $H: \D \to (\B \otimes_{\I} \D')[0,1]$ satisfying
\[
F - G = (\mu_2 \otimes \id) \circ (\id \otimes H) \circ \delta + (\mu_2 \otimes \id) \circ (\id \otimes \delta') \circ H + (\mu_1 \otimes \left|\id\right|) \circ H.
\]
If a homotopy exists between $F$ and $G$, then $F$ is said to be homotopic to $G$.

Two Type D structures $\D$ and $\D'$ are homotopy equivalent if there exist Type D structure morphisms $F: \D \to \D'$ and $G: \D' \to \D$, such that $G \circ F$ is homotopic to $\id_{\D}$ and $F \circ G$ is homotopic to $\id_{\D'}$.
\end{definition}

\begin{remark}\label{TypeDHtpyDescRem}
Suppose $\D$ and $\D'$ are homotopy equivalent Type D structures over $\B$, and $J$ is a bigrading-homogeneous ideal of $\B$ which is closed under the differential on $\B$. By Proposition~\ref{DandDDQuotient}, $\D$ and $\D'$ descend to Type D structures over $\B / J$. As such, they are still homotopy equivalent. Indeed, one may simply postcompose the algebra outputs of $F$, $G$, and $H$ with the projection map from $\B$ onto $\B / J$, and all the relevant conditions are still satisfied.
\end{remark}

Now let $N$ be a chain complex of projective graded left $H^n$-modules. In Section~\ref{RobertsTypeDStrSect}, the Type D structure $\D(N)$ over $\B \astrosun m(\B)^!$ was defined by applying the mirroring operation of Definition~\ref{TypeDMirroringDef} to $\widehat{A}(m(N)) \boxtimes {^{\B \astrosun m(\B)^!}}K^{m(\B \astrosun m(\B)^!)^{op}}$. One can check that the mirroring operation of Definition~\ref{TypeDMirroringDef} respects homotopy equivalences of Type D structures. Thus, to show that $\D([T]^{Kh})$ is a tangle invariant up to homotopy equivalence, it would suffice to show the following general result: if $\widehat{A}$ and $\widehat{A'}$ are Type A structures over a differential bigraded algebra $\B$ which are free as $\Z$-modules, $\DD$ is a Type DD bimodule over $\B$ and another differential bigraded algebra $\B'$, and $\widehat{A}$ and $\widehat{A'}$ are $\mathcal{A}_{\infty}$ homotopy equivalent, then $\widehat{A} \boxtimes \DD$ and $\widehat{A'} \boxtimes \DD$ are homotopy equivalent as Type D structures over $\B'$. Over $\Z/2\Z$, this is a standard property of the box tensor product. Here, we are working over $\Z$, but we will only need a simpler version of this result.

\begin{definition}\label{FboxtimesIdDef} Let $\B$ and $\B'$ be differential bigraded algebras over an idempotent ring $\I$. Let $\widehat{A}$ and $\widehat{A'}$ be differential bigraded right modules over $\B$, and let $(\DD,\delta_{DD})$ be a rank-one Type DD bimodule over $\B$ and $\B'$. Assume that $\widehat{A}$ and $\widehat{A'}$ are free as $\Z$-modules, with $\Z$-bases consisting of elements which are grading-homogeneous and have a unique right idempotent. 

Let $F: \widehat{A} \to \widehat{A'}$ be an $\mathcal{A}_{\infty}$ morphism with $F_n = 0$ for $n > 2$. Define a morphism of Type D structures $F \boxtimes \id_{DD}$ from $\widehat{A} \boxtimes \DD$ to $\widehat{A'} \boxtimes \DD$, or in other words a map
\[
F \boxtimes \id_{DD}: (\widehat{A} \boxtimes \DD) \to \B' \otimes_{\I} (\widehat{A'} \boxtimes \DD),
\]
by the formula
\[
F \boxtimes \id_{DD} := F_1 + \xi \circ (F_2 \otimes \left|\id\right|) \circ (\id \otimes \delta_{DD}),
\]
where we are identifying $\widehat{A} \boxtimes \DD$ with $\widehat{A}$ and $\widehat{A'} \boxtimes \DD$ with $\widehat{A'}$. Recall that 
\[
\xi: \widehat{A} \otimes_{\I} (\B')^{op} \to \B' \otimes_{\I} \widehat{A}
\]
was defined in Definition~\ref{XBoxWithDDDef}, and
\[
\xi: \widehat{A'} \otimes_{\I} (\B')^{op} \to \B' \otimes_{\I} \widehat{A'}
\]
is defined analogously. The map $F \boxtimes \id_{DD}$ is bigrading-preserving and respects the actions of $\I$ on $\widehat{A} \boxtimes \DD$ and $\widehat{A'} \boxtimes \DD$.
\end{definition}

\begin{proposition}\label{FboxIdIsTypeDMorph}
The map $F \boxtimes \id_{DD}$ defined in Definition~\ref{FboxtimesIdDef} is a morphism of Type D structures from $\widehat{A} \boxtimes \DD$ to $\widehat{A'} \boxtimes \DD$.
\end{proposition}

\begin{proof} We want to show that
\begin{equation}\label{TypeDRelsForFxId}
\begin{aligned}
(\mu_1 \otimes \left|\id\right|) \circ (F \boxtimes \id_{DD}) &=  (\mu_2 \otimes \id) \circ (\id \otimes (F \boxtimes \id_{DD})) \circ \delta \\
&- (\mu_2 \otimes \id) \circ (\id \otimes \delta') \circ (F \boxtimes \id_{DD}).
\end{aligned}
\end{equation}
The left side of equation~\ref{TypeDRelsForFxId} is
\[
(\mu_1 \otimes \left|\id\right|) \circ \xi \circ (F_2 \otimes \left|\id\right|) \circ (\id \otimes \delta_{DD}),
\]
which can be written as 
\[
- \xi \circ (F_2 \otimes \left|\id\right|) \circ (\id \otimes \id \otimes \mu_1) \circ (\id \otimes \delta_{DD}).
\]
Using the DD bimodule relations for $\delta_{DD}$, we may further rewrite this term as
\begin{align*}
&\xi \circ (F_2 \otimes \id) \circ (\id \otimes \mu_1 \otimes \id) \circ (\id \otimes \delta_{DD}) \\
&+ \xi \circ (F_2 \otimes \left|\id\right|) \circ (\id \otimes \mu_2 \otimes \mu_2) \circ (\id \otimes \sigma) \circ (\id \otimes \id \otimes \delta_{DD} \otimes \id) \circ (\id \otimes \delta_{DD}).
\end{align*}
Using the $n = 2$ and $n = 3$ $\mathcal{A}_{\infty}$ consistency conditions for $F$, the sum of these two terms is
\begin{align*}
&\xi \circ (F_1 \otimes \id) \circ (m_2 \otimes \id) \circ (\id \otimes \delta_{DD}) \\
&- \xi \circ (F_2 \otimes \id) \circ (\id \otimes \left|\id\right| \otimes \id) \circ (\id \otimes \delta_{DD}) \circ m_1 \\
&- \xi \circ (m'_1 \otimes \id) \circ (F_2 \otimes \id) \circ (\id \otimes \delta_{DD}) \\
&- \xi \circ (m'_2 \otimes \id) \circ (\id \otimes \delta_{DD}) \circ F_1 \\
&+ \xi \circ (F_2 \otimes \left|\id\right|) \circ (m_2 \otimes \id \otimes \id) \circ (\id \otimes \id \otimes \id \otimes \mu_2) \circ \sigma \\
&\qquad \circ (\id \otimes \id \otimes \delta_{DD} \otimes \id) \circ (\id \otimes \delta_{DD})\\
&+ (\left|\id\right| \otimes \id) \circ \xi \circ (m'_2 \otimes \id) \circ (F_2 \otimes \left|\id\right| \otimes \id) \circ (\id \otimes \id \otimes \id \otimes \mu_2) \circ \sigma \\
&\qquad \circ (\id \otimes \id \otimes \delta_{DD} \otimes \id) \circ (\id \otimes \delta_{DD}).
\end{align*}
We will refer to these six terms as $LHS_1$, $LHS_2$, $LHS_3$, $LHS_4$, $LHS_5$, and $LHS_6$.

The right side of equation~\ref{TypeDRelsForFxId} is
\begin{align*}
&F_1 \circ m_1 \\
&+ \xi \circ (F_2 \otimes \left|\id\right|) \circ (\id \otimes \delta_{DD}) \circ m_1 \\
&+ (\id \otimes F_1) \circ \xi \circ (m_2 \otimes \id) \circ (\id \otimes \delta_{DD}) \\
&+ (\mu_2 \otimes \id) \circ (\id \otimes \xi) \circ (\id \otimes F_2 \otimes \left|\id\right|) \circ (\id \otimes \id \otimes \delta_{DD}) \\
&\qquad \qquad \circ \xi \circ (m_2 \otimes \id) \circ (\id \otimes \delta_{DD}) \\
&- m'_1 \circ F_1 \\
&- \xi \circ (m'_2 \otimes \id) \circ (\id \otimes \delta_{DD}) \circ F_1 \\
&- (\id \otimes m'_1) \circ \xi \circ (F_2 \otimes \left|\id\right|) \circ (\id \otimes \delta_{DD}) \\
&- (\mu_2 \otimes \id) \circ (\id \otimes \xi) \circ (\id \otimes m'_2 \otimes \id) \circ (\id \otimes \id \otimes \delta_{DD}) \\
&\qquad \qquad \circ \xi \circ (F_2 \otimes \left|\id\right|) \circ (\id \otimes \delta_{DD}).
\end{align*}

We will refer to these eight terms as $RHS_1$-$RHS_8$. We have:
\begin{itemize}
\item $RHS_1 + RHS_5 = 0$ by the $n = 1$ consistency conditions for $F$;
\item $LHS_1 = RHS_3$ because $F_1$ is bigrading-preserving;
\item $LHS_2 = RHS_2$, since
\[
(\id \otimes \left|\id\right|) \circ \delta_{DD} = - (\left|\id\right| \otimes \id) \circ \delta_{DD};
\]
\item $LHS_3 = RHS_7$; and
\item $LHS_4 = RHS_6$.
\end{itemize}

It remains to show that $LHS_5 = RHS_4$ and that $LHS_6 = RHS_8$. These claims follows from direct computation: let $\delta_{DD}(1) = \sum_i b_i \otimes (b'_i)^{op}$. The term $LHS_5$, when applied to a generator $x$ of $\widehat{A}$, gives the sum
\begin{align*}
&\sum_{i,j} (-1)^{(\deg_h b'_i)(\deg_h b_j)} (-1)^{\deg_h b'_i + \deg_h b'_j} \\
& \qquad \cdot (-1)^{(\deg_h x + \deg_h b_i + \deg_h b_j - 1)(\deg_h b'_i + \deg_h b'_j)} b'_i b'_j \otimes F_2(xb_i, b_j) \\
&\sum_{i,j} (-1)^{(\deg_h x)(\deg_h b'_i + \deg_h b'_j) + (\deg_h b_i)(\deg_h b'_j)} b'_i b'_j \otimes F_2(xb_i, b_j);
\end{align*}
to see that the second sum is equal to the first, use the fact that $\deg_h b_i + \deg_h b'_i = 1$ and $\deg_h b_j + \deg_h b'_j = 1$. In particular, 
\[
(-1)^{(\deg_h b_i)(\deg_h b'_i)} = 1
\]
and 
\[
(-1)^{(\deg_h b_j)(\deg_h b'_j)} = 1.
\]
Applying the term $RHS_4$ to $x$ gives 
\begin{align*}
&\sum_{i,j} (-1)^{(\deg_h x + \deg_h b_i)(\deg_h b'_i)}(-1)^{\deg_h b'_j}(-1)^{(\deg_h x + \deg_h b_i + \deg_h b_j - 1)(\deg_h b'_j)} \\
& \qquad \cdot b'_i b'_j \otimes F_2(xb_i, b_j) \\
&\sum_{i,j} (-1)^{(\deg_h x)(\deg_h b'_i + \deg_h b'_j) + (\deg_h b_i)(\deg_h b'_j)} b'_i b'_j \otimes F_2(xb_i, b_j);
\end{align*}
thus, $LHS_5 = RHS_4$.

Similarly, applying the term $LHS_6$ to $x$ gives the sum 
\begin{align*}
&\sum_{i,j} (-1)^{(\deg_h b_j)(\deg_h b'_i)}(-1)^{\deg_h b_j}(-1)^{(\deg_h x + \deg_h b_i + \deg_h b_j - 1)(\deg_h b'_i + \deg_h b'_j)} \\
&\qquad \cdot (-1)^{\deg_h b'_i + \deg_h b'_j} b'_i b'_j \otimes F_2(xb_i, b_j) \\
&=  \sum_{i,j} -(-1)^{(\deg_h x)(\deg_h b'_i + \deg_h b'_j) + (\deg_h b'_i)(\deg_h b'_j)} b'_i b'_j \otimes F_2(xb_i, b_j)
\end{align*}
and applying the term $RHS_8$ to $x$ gives 
\begin{align*}
&- \sum_{i,j} (-1)^{\deg_h b'_i}(-1)^{(\deg_h x + \deg_h b_i - 1)(\deg_h b'_i)}(-1)^{(\deg_h x + \deg_h b_i + \deg_h b_j - 1)(\deg_h b'_j)} \\
&\qquad \cdot b'_i b'_j \otimes F_2(xb_i, b_j) \\
&= \sum_{i,j} -(-1)^{(\deg_h x)(\deg_h b'_i + \deg_h b'_j) + (\deg_h b'_i)(\deg_h b'_j)} b'_i b'_j \otimes F_2(xb_i, b_j).
\end{align*}
Thus, $LHS_6 = RHS_8$, and $F \boxtimes \id_{DD}$ is a valid morphism of Type D structures from $\widehat{A} \boxtimes \DD$ to $\widehat{A'} \boxtimes \DD$.
\end{proof}

\begin{proposition}\label{TypeDCompProp} If $F$ and $G$ are $\mathcal{A}_{\infty}$ morphisms from $\widehat{A}$ to $\widehat{A'}$ as in Definition~\ref{FboxtimesIdDef}, with $F_n, G_n = 0$ for $n > 2$ and either $F_2 = 0$ or $G_2 = 0$, then 
\[
(G \circ F) \boxtimes \id_{DD} = (G \boxtimes \id_{DD}) \circ (F \boxtimes \id_{DD}).
\]
\end{proposition}

\begin{proof} First, suppose $G_2 = 0$. Then $(G \circ F)_1 = G_1 \circ F_1$, and $(G \circ F)_2 = G_1 \circ F_2$. We have
\begin{align*}
(G \circ F) \boxtimes \id_{DD} &= (G \circ F)_1 + \xi \circ ((G \circ F)_2 \otimes \left|\id\right|) \circ (\id \otimes \delta_{DD}) \\
&= G_1 \circ F_1 + \xi \circ (G_1 \otimes \id) \circ (F_2 \otimes \left|\id\right|) \circ (\id \otimes \delta_{DD}).
\end{align*}
On the other hand,
\begin{align*}
&(G \boxtimes \id_{DD}) \circ (F \boxtimes \id_{DD}) \\
&= (\mu_2 \otimes \id) \circ (\id \otimes (G \boxtimes \id_{DD})) \circ (F \boxtimes \id_{DD}) \\
&= (\mu_2 \otimes \id) \circ (\id \otimes G_1) \circ (F_1 + \xi \circ (F_2 \otimes \left|\id\right|) \circ (\id \otimes \delta_{DD})) \\
&= G_1 \circ F_1 + (\id \otimes G_1) \circ \xi \circ (F_2 \otimes \left|\id\right|) \circ (\id \otimes \delta_{DD}).
\end{align*}
This expression equals $(G \circ F) \boxtimes \id_{DD}$ because $G_1$ is bigrading-preserving.

Now suppose instead that $F_2 = 0$. Then $(G \circ F)_1$ is still $G_1 \circ F_1$, and $(G \circ F)_2 = G_2 \circ (F_1 \otimes \id)$. We have
\begin{align*}
(G \circ F) \boxtimes \id_{DD} &= (G \circ F)_1 + \xi \circ ((G \circ F)_2 \otimes \left|\id\right|) \circ (\id \otimes \delta_{DD}) \\
&= G_1 \circ F_1 + \xi \circ (G_2 \otimes \left|\id\right|) \circ (F_1 \otimes \id \otimes \id) \circ (\id \otimes \delta_{DD}).
\end{align*}
On the other hand,
\begin{align*}
&(G \boxtimes \id_{DD}) \circ (F \boxtimes \id_{DD}) \\
&= (\mu_2 \otimes \id) \circ (\id \otimes (G \boxtimes \id_{DD})) \circ (F \boxtimes \id_{DD}) \\
&= (\mu_2 \otimes \id) \circ (\id \otimes (G_1 + \xi \circ (G_2 \otimes \left|\id\right|) \circ (\id \otimes \delta_{DD}))) \circ F_1 \\
&= G_1 \circ F_1 + \xi \circ (G_2 \otimes \left|\id\right|) \circ (\id \otimes \delta_{DD}) \circ F_1,
\end{align*}
which equals $(G \circ F) \boxtimes \id_{DD}$ because 
\[
(F_1 \otimes \id \otimes \id) \circ (\id \otimes \delta_{DD}) = (\id \otimes \delta_{DD}) \circ F_1.
\]
\end{proof}

\begin{proposition}\label{TypeDValidHtpyProp} As in Proposition~\ref{TypeDCompProp}, let $F$ and $G$ be $\mathcal{A}_{\infty}$ morphisms from $\widehat{A}$ to $\widehat{A'}$ with $F_n, G_n = 0$ for $n > 2$. Let $H$ be an $\mathcal{A}_{\infty}$ homotopy between $F$ and $G$ with $H_n = 0$ for $n > 1$. Define
\[
H \boxtimes \id_{DD} := H_1,
\]
where the algebra output of $H \boxtimes \id_{DD}$ is always $1$; then $H \boxtimes \id_{DD}$ is a homotopy of Type D morphisms between $F \boxtimes \id_{DD}$ and $G \boxtimes \id_{DD}$.
\end{proposition}

\begin{proof}
Let $\delta$ and $\delta'$ denote the Type D operations on $\widehat{A} \boxtimes \DD$ and $\widehat{A'} \boxtimes \DD$ respectively. We want to show that
\[
F \boxtimes \id_{DD} - G \boxtimes \id_{DD} = H_1 \circ \delta + \delta' \circ H_1;
\]
the other term in the Type D homotopy relations is zero for this special type of $H$. Expanding out the left side, we want to show that
\[
F_1 - G_1 + \xi \circ ((F_2 - G_2) \otimes \left|\id\right|) \circ (\id \otimes \delta_{DD}) = H_1 \circ \delta + \delta' \circ H_1.
\]
By Example~\ref{SimplestHtpyExample}, the $\mathcal{A}_{\infty}$ homotopy relations for $H$ give us the following two equations:
\[
F_1 - G_1 = m'_1 \circ H_1 + H_1 \circ m_1
\]
and
\[
F_2 - G_2 = -m'_2 \circ (H_1 \otimes \left|\id\right|) + H_1 \circ m_2.
\]
Thus, the left side of the Type D homotopy relation is
\begin{align*}
&m'_1 \circ h_1 + H_1 \circ m_1 + \xi \circ ((-m'_2 \circ (H_1 \otimes \left|\id\right|) + H_1 \circ m_2) \otimes \left|\id\right|) \circ (\id \otimes \delta_{DD}) \\
&=m'_1 \circ h_1 + H_1 \circ m_1 - \xi \circ (m'_2 \otimes \id) \circ (H_1 \otimes \left|\id\right| \otimes \left|\id\right|) \circ (\id \otimes \delta_{DD}) \\
&\qquad + \xi \otimes (H_1 \otimes \left|\id\right|) \circ (m_2 \otimes \id) \circ (\id \otimes \delta_{DD}) \\
&= m'_1 \circ h_1 + H_1 \circ m_1 + \xi \circ (m'_2 \otimes \id) \circ (H_1 \otimes \id \otimes \id) \circ (\id \otimes \delta_{DD}) \\
&\qquad + (\id \otimes H_1) \circ \xi \circ (m_2 \otimes \id) \circ (\id \otimes \delta_{DD}) \\
&= m'_1 \circ h_1 + H_1 \circ m_1 + \xi \circ (m'_2 \otimes \id) \circ (\id \otimes \delta_{DD}) \circ H_1 \\
&\qquad + (\id \otimes H_1) \circ \xi \circ (m_2 \otimes \id) \circ (\id \otimes \delta_{DD}).
\end{align*}

On the other hand, using the definition of $\delta$ and $\delta'$ in Definition~\ref{XBoxWithDDDef}, the right side of the Type D homotopy relation can be expanded out as
\begin{align*}
&H_1 \circ m_1 + (\id \otimes H_1) \circ \xi \circ (m_2 \otimes \id) \circ (\id \otimes \delta_{DD}) \\
&+ m'_1 \circ H_1 + \xi \circ (m'_2 \otimes \id) \circ (\id \otimes \delta_{DD}) \circ H_1,
\end{align*}
which is identical to the previous expression after rearranging terms. Thus, $H \boxtimes \id_{DD}$ is a valid Type D homotopy between $F \boxtimes \id_{DD}$ and $G \boxtimes \id_{DD}$.
\end{proof}

\begin{corollary}\label{TypeDGeneralHtpyCorr} Let $\B$ and $\B'$ be differential bigraded algebras over an idempotent ring $\I$. Let $\widehat{A}$ and $\widehat{A'}$ be differential bigraded right modules over $\B$, and let $(\DD,\delta_{DD})$ be a rank-one Type DD bimodule over $\B$ and $\B'$. Assume that $\widehat{A}$ and $\widehat{A'}$ are free as $\Z$-modules, with $\Z$-bases consisting of elements which are grading-homogeneous and have a unique right idempotent. 

Suppose there exist $\mathcal{A}_{\infty}$ morphisms $F: \widehat{A} \to \widehat{A'}$ and $G: \widehat{A'} \to \widehat{A}$ with $F_n = 0$ and $G_n = 0$ for $n > 2$, and such that either $F_2 = 0$ or $G_2 = 0$. Furthermore, suppose that $G \circ F$ is $\mathcal{A}_{\infty}$ homotopic to $\id_{\widehat{A}}$ via an $\mathcal{A}_{\infty}$ homotopy $H$ with $H_n = 0$ for $n > 1$, and $F \circ G$ is $\mathcal{A}_{\infty}$ homotopic to $\id_{\widehat{A'}}$ via another $\mathcal{A}_{\infty}$ homotopy $H'$ with $H'_n = 0$ for $n > 1$.

Then the Type D structures $\widehat{A} \boxtimes \DD$ and $\widehat{A'} \boxtimes \DD$ over $\B'$, defined in Definition~\ref{XBoxWithDDDef}, are homotopy equivalent.
\end{corollary}

\begin{proof} This follows from Proposition~\ref{FboxIdIsTypeDMorph}, Proposition~\ref{TypeDCompProp}, Proposition~\ref{TypeDValidHtpyProp}, together with the fact that the box tensor product with $\id_{DD}$ on morphisms sends identity morphisms to identity morphisms.
\end{proof}

\begin{corollary}\label{TypeDHtpyCorrFinal} If $T$ and $T'$ are oriented tangle diagrams in $\R_{\geq 0} \otimes \R$ which are related by a Reidemeister move, then $\D([T]^{Kh})$ and $\D([T']^{Kh})$ are homotopy equivalent as Type D structures over $\B \astrosun m(\B)^!$. Thus, they are also homotopy equivalent as Type D structures over the quotient algebra $\B \Gamma_n$.
\end{corollary}

\begin{proof}
The first claim follows from Corollary~\ref{TypeDGeneralHtpyCorr} and the proof of Corollary~\ref{GetInducedHtpyEquivCorr}, in which the $\mathcal{A}_{\infty}$ morphisms $F = \widehat{A}(f)$ and $G = \widehat{A}(g)$ and the $\mathcal{A}_{\infty}$ homotopy $H = \widehat{A}(\psi)$ used to realize the $\mathcal{A}_{\infty}$ homotopy equivalences satisfy the conditions of Corollary~\ref{TypeDGeneralHtpyCorr}. The second claim follows from Remark~\ref{TypeDHtpyDescRem} above.
\end{proof}

Corollary~\ref{TypeDHtpyCorrFinal} gives us an alternate proof of Theorem 44 of Roberts~\cite{RtypeD}.

\bibliographystyle{plain}
\bibliography{biblio}

\end{document}